\documentclass[a4paper]{amsart}
\usepackage[foot]{amsaddr}

\usepackage[latin1]{inputenc}
\usepackage{amssymb,amsmath,amsthm}
\usepackage{amsfonts}
\usepackage{amsaddr}
\usepackage{enumitem}
\usepackage{booktabs}
\usepackage{ifthen}
\usepackage{tikz}
\usetikzlibrary{math,calc}
\usepackage{url}
\usepackage[english]{babel}
\usepackage{hyphenat}
\usepackage{hyperref}

\theoremstyle{plain}
\newtheorem{theorem}{Theorem}[section]
\newtheorem{proposition}[theorem]{Proposition}
\newtheorem{lemma}[theorem]{Lemma}
\newtheorem{corollary}[theorem]{Corollary}

\theoremstyle{definition}
\newtheorem{definition}[theorem]{Definition}

\newtheorem{remark}[theorem]{Remark}

\newcommand{\IN}{\ensuremath{\mathbb{N}}}

\newcommand{\FF}{\ensuremath{\mathbb{F}}}

\newcommand{\nset}[1]{\ensuremath{[{#1}]}}
\newcommand{\card}[1]{\ensuremath{\lvert{#1}\rvert}}
\newcommand{\gendefault}{}
\newcommand{\gen}[2][\gendefault]{\ensuremath{\langle{#2}\rangle_{#1}}}
\newcommand{\clonegen}[1]{\gen[]{#1}}
\newcommand{\vect}[1]{\ensuremath{\mathbf{#1}}}

\newcommand{\clAll}{\ensuremath{\mathsf{\Omega}}}
\newcommand{\clTa}[1]{\ensuremath{\mathsf{T}_{#1}}}
\newcommand{\clTo}{\ensuremath{\mathsf{T}_0}}
\newcommand{\clTi}{\ensuremath{\mathsf{T}_1}}
\newcommand{\clTc}{\ensuremath{\mathsf{T}_\mathrm{c}}}

\newcommand{\clM}{\ensuremath{\mathsf{M}}}
\newcommand{\clS}{\ensuremath{\mathsf{S}}}
\newcommand{\clSc}{\ensuremath{\mathsf{S}_\mathrm{c}}}
\newcommand{\clSM}{\ensuremath{\mathsf{SM}}}
\newcommand{\clL}{\ensuremath{\mathsf{L}}}
\newcommand{\clLa}[1]{\ensuremath{\mathsf{L}_{#1}}}
\newcommand{\clLo}{\ensuremath{\mathsf{L}_0}}
\newcommand{\clLi}{\ensuremath{\mathsf{L}_1}}
\newcommand{\clLc}{\ensuremath{\mathsf{L}_\mathrm{c}}}
\newcommand{\clLS}{\ensuremath{\mathsf{LS}}}

\newcommand{\clLambdac}{\ensuremath{\mathsf{\Lambda}_\mathrm{c}}}

\newcommand{\clVc}{\ensuremath{\mathsf{V}_\mathrm{c}}}
\newcommand{\clIstar}{\ensuremath{\mathsf{I}^{*}}}
\newcommand{\clIa}[1]{\ensuremath{\mathsf{I}_{#1}}}
\newcommand{\clIo}{\ensuremath{\mathsf{I}_0}}
\newcommand{\clIi}{\ensuremath{\mathsf{I}_1}}
\newcommand{\clIc}{\ensuremath{\mathsf{I}_\mathrm{c}}}
\newcommand{\clEmpty}{\ensuremath{\mathsf{\emptyset}}}
\newcommand{\clD}[1]{\ensuremath{\mathsf{D}_{#1}}}
\newcommand{\clChar}[1]{\ensuremath{\mathsf{X}_{#1}}}
\newcommand{\clVal}[2]{\ensuremath{\clAll_{\ifthenelse{\equal{#1}{}}{\mathord{*}}{#1}\ifthenelse{\equal{#2}{}}{\mathord{*}}{#2}}}}
\newcommand{\clEq}{\ensuremath{\clAll_{=}}}
\newcommand{\clNeq}{\ensuremath{\clAll_{\neq}}}
\newcommand{\clCon}[1]{\ensuremath{\clVal{#1}{}}}
\newcommand{\clYksi}[1]{\ensuremath{\clVal{}{#1}}}
\newcommand{\clBoth}[2]{\ensuremath{\clVal{#1}{#2}}}
\newcommand{\clEven}{\ensuremath{\clEq}}
\newcommand{\clOdd}{\ensuremath{\clNeq}}
\newcommand{\ParRel}{\ensuremath{\approx}}
\newcommand{\clParity}[1]{\ensuremath{\ifthenelse{\equal{#1}{0}}{\clEven}{\ifthenelse{\equal{#1}{1}}{\clOdd}{\ifthenelse{\equal{#1}{a}}{\clAll_{\ParRel}}{\marginpar{???}}}}}}
\newcommand{\QuantifyParRel}{\ensuremath{\mathord{\ParRel} \in \{\mathord{=}, \mathord{\neq}\}}}
\newcommand{\id}{\ensuremath{\mathrm{id}}}
\newcommand{\cf}[2]{\ensuremath{\mathrm{c}^{({#1})}_{#2}}}
\newcommand{\GF}{\ensuremath{\mathrm{GF}}}
\newcommand{\pr}{\ensuremath{\mathrm{pr}}}
\newcommand{\Char}{\ensuremath{\mathrm{ch}}}

\newcommand{\monomials}[1]{\ensuremath{M_{#1}}}
\newcommand{\arity}[1]{\ensuremath{\mathrm{ar}({#1})}}
\newcommand{\parity}[1]{\ensuremath{\mathrm{par}({#1})}}
\newcommand{\charrank}[1]{\ensuremath{\chi({#1})}}
\newcommand{\symmdiff}{\bigtriangleup}
\newcommand{\monster}[1]{\ensuremath{W_{#1}}}
\newcommand{\closys}[1]{\ensuremath{\mathcal{L}_{#1}}}
\newcommand{\clProj}[1]{\ensuremath{\mathsf{J}_{#1}}}

\begin{document}
\title[Stability of Boolean function classes]{Stability of Boolean function classes with respect to clones of linear functions}

\author{Miguel Couceiro}
\address[M. Couceiro]%
   {Universit\'e de Lorraine, CNRS, LORIA \\
    F-54000 Nancy \\
    France}

\author{Erkko Lehtonen}
\address[E. Lehtonen]%
   {Centro de Matem\'atica e Aplica\c{c}\~oes \\
    Faculdade de Ci\^encias e Tecnologia \\
    Universidade Nova de Lisboa \\
    Quinta da Torre \\
    2829-516 Caparica \\
    Portugal}

\dedicatory{This paper is dedicated to Maurice Pouzet to whom we are deeply thankful for his guidance, friendship, knowledgeable support, and for being always a source of great motivation and inspiration.}

\date{\today}

\begin{abstract}
We consider classes of Boolean functions stable under compositions both from the right and from the left with clones.
Motivated by the question how many properties of Boolean functions can be defined by means of linear equations,
we focus on stability under compositions with the clone of linear idempotent functions.
It follows from a result by Sparks that there are countably many such linearly definable classes of Boolean functions.
In this paper, we refine this result by completely describing these classes.
This work is tightly related with the theory of function minors, stable classes, clonoids, and hereditary classes, topics that have been widely investigated in recent years by several authors including Maurice Pouzet and his coauthors.
\end{abstract}

\maketitle

%%%%%%%%%%%%%%%%%%%%%%%%%%%%%%%%%%%%%%%%%%%%%%%%%%

\section{Introduction}

This paper is a study of classes of functions of several arguments from a set $A$ to a set $B$ that are closed under composition from the right with a clone $C_1$ on $A$ and under composition from the left with a clone $C_2$ on $B$, in brief, \emph{$(C_1,C_2)$\hyp{}stable} classes of functions (see \cite{CouFol-2009}).
Special instances of the notion of $(C_1,C_2)$\hyp{}stability appear in the literature.
For example, if both $C_1$ and $C_2$ are clones of projections on the respective sets, then we get \emph{minor\hyp{}closed classes} or \emph{minions} or \emph{equational classes} (see Pippenger~\cite{Pippenger}, Ekin et al.\ \cite{EFHH}).
If $C_1$ is the clone of projections and $C_2$ is the clone of an algebra $\mathbf{B}$, then we get \emph{clonoids} with source set $A$ and target algebra $\mathbf{B}$ (see Aichinger and Mayr \cite{AicMay-2016}).

If both $C_1$ and $C_2$ are equal to the clone $\clLc$ of idempotent linear functions on $\{0,1\}$,
then the $(C_1,C_2)$\hyp{}stable classes are exactly the classes of Boolean functions definable by linear equations (see \cite{CouFol-2004}).
It was already observed in \cite{CouFol-2004} that there are infinitely many such linearly definable classes, but at the time it remained an open question whether there are countably or uncountably many such classes and exactly what these classes are.
It follows from the results of Aichinger and Mayr~\cite{AicMay-2016} that for all clones $C$, the number of $(\clLc,C)$\hyp{}stable classes is at most $\aleph_0$.
Indeed, there is no infinite descending chain of $(\clLc,\clIc)$\hyp{}stable classes by \cite{AicMay-2016}; thus each such class is determined by a single relation pair $(R,S)$.
One of the main goals of the current paper is to provide an explicit description of the countably infinite number of $(\clLc,\clLc)$\hyp{}stable classes (in brief, \emph{$\clLc$\hyp{}stable} classes).

More generally, we would like to describe $(C_1,C_2)$\hyp{}stable classes.
This problem seems unfeasible in full generality, since there are uncountably many clones on sets with at least three elements (see Yanov and Muchnik \cite{YanMuc-1959}).
However, Post~\cite{Post} showed that there are countably many clones on the two\hyp{}element set and he provided a classification thereof, currently known as Post's lattice, thus making the description of $(C_1,C_2)$\hyp{}stable classes of Boolean functions potentially feasible.
This is the research program we are currently setting forth.
A complete description still eludes us but, motivated by the framework of \cite{CouFol-2004,CouFol-2007} dealing with linear definability of function classes, we address stability with respect to those clones that contain the clone $\clLc$ of idempotent linear functions.

The paper is organized as follows.
\begin{itemize}
\item Section~\ref{sec:preliminaries}: We provide the basic definitions and preliminary results that are needed in the sequel.
\item Section~\ref{sec:stab}: We establish some auxiliary tools for studying $(C_1,C_2)$\hyp{}stability.
\item Section~\ref{sec:ff}: We make a little diversion to clones on finite fields, and we describe the $\clL$\hyp{}stable classes, where $\clL$ denotes the clone of all linear functions on the finite field $\FF_p$ of prime order $p$. 
\item Section~\ref{sec:Bf}: We define various properties of Boolean functions that are needed for describing the $\clLc$\hyp{}stable classes.
\item Section~\ref{sec:Lc-stable}: We present our main result: an explicit description of the $\clLc$\hyp{}stable classes of Boolean functions.
The proof has two parts.
First we show that the listed classes are $\clLc$\hyp{}stable; this is straightforward verification.
The more difficult part of the proof is to show that there are no further $\clLc$\hyp{}stable classes.
\item Section~\ref{sec:C1C2}: With the help of the result on $\clLc$\hyp{}stable classes, we obtain with little effort also a description of $(C_1,C_2)$\hyp{}stable classes for clones $C_1$ and $C_2$, where $C_1$ is arbitrary and $\clLc \subseteq C_2$.
\item Section~\ref{sec:remarks}: We make some concluding remarks and indicate directions for future research.
\end{itemize}

The main results of this paper were presented without proofs in the 1st International Conference on Algebras, Graphs and Ordered Sets (ALGOS~2020) \cite{CL-ALGOS2020}.
The reader should be cautious about the fact that some notation and terminology have been slightly changed from the conference paper.

%%%%%%%%%%%%%%%%%%%%%%%%%%%%%%%%%%%%%%%%%%%%%%%%%%

\section{Preliminaries}
\label{sec:preliminaries}

The symbols $\IN$ and $\IN_{+}$ denote the set of all nonnegative integers and the set of all positive integers, respectively.
For any $n \in \IN$, the symbol $\nset{n}$ denotes the set $\{ \, i \in \IN \mid 1 \leq i \leq n \, \}$.

\begin{definition}
Let $A$ and $B$ be sets.
A mapping of the form $f \colon A^n \to B$ for some $n \in \IN_{+}$ is called a \emph{function of several arguments from $A$ to $B$} (or simply a \emph{function}).
The number $n$ is called the \emph{arity} of $f$ and denoted by $\arity{f}$.
If $A = B$, then such a function is called an \emph{operation on $A$.}
We denote by $\mathcal{F}_{AB}$ and $\mathcal{O}_A$ the set of all functions of several arguments from $A$ to $B$ and the set of all operations on $A$, respectively.
For any $n \in \IN_{+}$, we denote by $\mathcal{F}_{AB}^{(n)}$ the set of all $n$\hyp{}ary functions in $\mathcal{F}_{AB}$, and for any $C \subseteq \mathcal{F}_{AB}$, we let $C^{(n)} := C \cap \mathcal{F}_{AB}^{(n)}$ and call it the \emph{$n$\hyp{}ary part} of $C$.
\end{definition}

\begin{definition}
For $b \in B$ and $n \in \IN$, the $n$\hyp{}ary \emph{constant function} $\cf{n}{b} \colon A^n \to B$ is given by the rule $(a_1, \dots, a_n) \mapsto b$ for all $a_1, \dots, a_n \in A$.
In the case when $A = B$, for $n \in \IN$ and $i \in \nset{n}$, the $i$\hyp{}th $n$\hyp{}ary \emph{projection} $\pr_i^{(n)} \colon A^n \to A$ is given by the rule $(a_1, \dots, a_n) \mapsto a_i$ for all $a_1, \dots, a_n \in A$.
\end{definition}

\begin{definition}
Let $f \colon A^n \to B$ and $i \in \nset{n}$.
The $i$\hyp{}th argument is \emph{essential} in $f$ if there exist $a_1, \dots, a_n, a'_i \in A$ such that
\[
f(a_1, \dots, a_n) \neq f(a_1, \dots, a_{i-1}, a'_i, a_{i+1}, \dots, a_n).
\]
An argument that is not essential is \emph{fictitious.}
The \emph{essential arity} of $f$ is the number of its essential arguments.
\end{definition}

\begin{definition}
Let $f \colon B^n \to C$ and $g_1, \dots, g_n \colon A^m \to B$.
The \emph{composition} of $f$ with $g_1, \dots, g_n$ is the function $f(g_1, \dots, g_n) \colon A^m \to C$ given by the rule
\[
f(g_1, \dots, g_n)(\vect{a}) :=
f(g_1(\vect{a}), \dots, g_n(\vect{a}))
\]
for all $\vect{a} \in A^m$.
The function $f$ is called the \emph{outer function} and $g_1, \dots, g_n$ are called the \emph{inner functions} of the composition.
\end{definition}

\begin{definition}
\label{def:minor}
Let $f \colon A^n \to B$ and $\sigma \colon \nset{n} \to \nset{m}$.
Define the function $f_\sigma \colon A^m \to B$ by the rule
\[
f_\sigma(a_1, \dots, a_m) = f(a_{\sigma(1)}, \dots, a_{\sigma(n)}),
\]
for all $a_1, \dots, a_m \in A$.
Such a function $f_\sigma$ is called a \emph{minor} of $f$, formed via the \emph{minor formation map} $\sigma$.
Intuitively, minors of $f$ are all those functions that can be obtained from $f$ by manipulation of its arguments: permutation of arguments, introduction of fictitious arguments, identification of arguments.
It is clear from the definition that the minor $f_\sigma$ can be obtained as a composition of $f$ with $m$\hyp{}ary projections on $A$:
\[
f_\sigma = f(\pr_{\sigma(1)}^{(m)}, \dots, \pr_{\sigma(n)}^{(m)}).
\]

An important special case of minors is the identification of a pair of arguments.
This is obtained with minor formation maps of the following form: for $i, j \in \nset{n}$ with $i < j$, let $\sigma_{ij} \colon \nset{n} \to \nset{n-1}$ be given by
\[
\sigma_{ij}(m) =
\begin{cases}
m, & \text{if $m < j$,} \\
i, & \text{if $m = j$,} \\
m-1, & \text{if $m > j$.}
\end{cases}
\]
We call such a map $\sigma_{ij}$ an \emph{identification map,} and
we write $f_{ij}$ for $f_{\sigma_{ij}}$.
More explicitly,
\[
f_{ij}(a_1, \dots, a_{n-1}) = f(a_1, \dots, a_i, \dots, a_{j-1}, a_i, a_j, \dots, a_{n-1}).
\] 

We write $f \leq g$ if $f$ is a minor of $g$.
The minor relation $\leq$ is a quasiorder (a reflexive and transitive relation) on $\mathcal{F}_{AB}$, and it induces an equivalence relation $\equiv$ on $\mathcal{F}_{AB}$ and a partial order on the quotient $\mathcal{F}_{AB} / {\equiv}$ in the usual way: $f \equiv g$ if $f \leq g$ and $g \leq f$, and $f / {\equiv} \leq g / {\equiv}$ if $f \leq g$.
\end{definition}

The effect of successive formations of minors is captured by the composition of minor\hyp{}forming maps.

\begin{lemma}
\label{lem:minor-composition}
Let $f \colon A^n \to B$, $\sigma \colon \nset{n} \to \nset{m}$, and $\tau \colon \nset{m} \to \nset{\ell}$.
Then $(f_\sigma)_\tau = f_{\tau \circ \sigma}$.
\end{lemma}

\begin{proof}
For all $a_1, \dots, a_\ell \in A$, we have
\begin{align*}
(f_\sigma)_\tau(a_1, \dots, a_\ell)
&
= (f_\sigma)(a_{\tau(1)}, \dots, a_{\tau(m)})
= f(a_{\tau(\sigma(1))}, \dots, a_{\tau(\sigma(n))})
\\ &
= f(a_{(\tau \circ \sigma)(1)}, \dots, a_{(\tau \circ \sigma)(n)})
= f_{\tau \circ \sigma}(a_1, \dots, a_\ell).
\qedhere
\end{align*}
\end{proof}

\begin{remark}
\label{rem:minor}
It is well known that any function can be decomposed into a surjection and an injection.
This obviously holds for minor formation maps $\sigma \colon \nset{n} \to \nset{m}$; we obtain $\sigma = \rho \circ \tau$ where $\tau \colon \nset{n} \to \nset{\ell}$ is surjective and $\rho \colon \nset{\ell} \to \nset{m}$ is injective.
Moreover, as explained in \cite[Section~2.2]{Lehtonen:habilitation},
we can choose the surjective map $\tau$ so that it is a composition of a number of identification maps: $\tau = \sigma_{i_k j_k} \circ \dots \circ \sigma_{i_1 j_1}$ (we regard the empty composition as the identity map on $\nset{\ell}$).

Intuitively, this means that any minor of a function $f \colon A^n \to B$ can be formed by first successively identifying pairs of arguments, and then introducing fictitious arguments and permuting arguments.
\end{remark}

Composition of functions satisfies the so\hyp{}called superassociative law.
Consequently, formation of minors commutes with composition.

\begin{lemma}
\label{lem:comp-minors}
Let $f \colon C^n \to D$, $g_1, \dots, g_n \colon B^m \to C$, and $h_1, \dots, h_m \colon A^\ell \to B$.
Then
\[
(f(g_1, \dots, g_n))(h_1, \dots, h_m) = f(g_1(h_1, \dots, h_m), \dots, g_n(h_1, \dots, h_m)).
\]
Consequently, \hspace{-0.55pt}for any $\sigma \colon \nset{\ell} \to \nset{m}$, \hspace{-0.55pt}we have $(f(g_1, \dots, g_n))_\sigma = f((g_1)_\sigma, \dots, (g_n)_\sigma))$.
\end{lemma}

\begin{proof}
For any $\vect{a} \in A^\ell$, we have
\begin{align*}
&
(f(g_1, \dots, g_n))(h_1, \dots, h_m)(\vect{a})
= (f(g_1, \dots, g_n))(h_1(\vect{a}), \dots, h_m(\vect{a}))
\\ &
= f(g_1(h_1(\vect{a}), \dots, h_m(\vect{a})), \dots, g_n(h_1(\vect{a}), \dots, h_m(\vect{a})))
\\ &
= f(g_1(h_1, \dots, h_m)(\vect{a}), \dots, g_n(h_1, \dots, h_m)(\vect{a}))
\\ &
= f(g_1(h_1, \dots, h_m), \dots, g_n(h_1, \dots, h_m))(\vect{a}).
\end{align*}
The statement about minors follows by taking $h_i := \pr^{(\ell)}_{\sigma(i)}$, $1 \leq i \leq m$.
\end{proof}

The notion of functional composition extends naturally to classes of functions.

\begin{definition}
Let $C \subseteq \mathcal{F}_{BC}$ and $K \subseteq \mathcal{F}_{AB}$.
The \emph{composition} of $C$ with $K$ is defined as
\[
CK := \{ \, f(g_1, \dots, g_n) \mid f \in C^{(n)}, \, g_1, \dots, g_n \in K^{(m)}, \, n, m \in \IN_{+} \, \}.
\]
\end{definition}

\begin{remark}
\label{rem:comp-monot}
It follows immediately from the definition of function class composition that if $C, C' \subseteq \mathcal{F}_{BC}$ and $K, K' \subseteq \mathcal{F}_{AB}$ satisfy $C \subseteq C'$ and $K \subseteq K'$, then $CK \subseteq C'K'$.
\end{remark}

\begin{lemma}
\label{lem:union-left}
For any $C_1, C_2 \subseteq \mathcal{F}_{BC}$, $K \subseteq \mathcal{F}_{AB}$, it holds that $(C_1 \cap C_2) K \subseteq C_1 K \cap C_2 K$ and $(C_1 \cup C_2) K = C_1 K \cup C_2 K$.
\end{lemma}

\begin{proof}
We clearly have
\begin{align*}
& (C_1 \cap C_2) K = (C_1 \cap C_2) K \cap (C_1 \cap C_2) K \subseteq C_1 K \cap C_2 K,
\\ &
C_1 K \cup C_2 K \subseteq (C_1 \cup C_2) K \cup (C_1 \cup C_2) K = (C_1 \cup C_2) K.
\end{align*}
In order to prove the inclusion $(C_1 \cup C_2) K \subseteq C_1 K \cup C_2 K$, let $h \in (C_1 \cup C_2) K$.
Then $h = f(g_1, \dots, g_n)$ for some $f \in C_1 \cup C_2$ and $g_1, \dots, g_2 \in K$.
Since $f \in C_1$ or $f \in C_2$, we have that $f(g_1, \dots, g_n)$ belongs to $C_1 K$ or $C_2 K$; therefore $h = f(g_1, \dots, g_n) \in C_1 K \cup C_2 K$.
\end{proof}

\begin{remark}
The inclusion $C_1 K \cap C_2 K \subseteq (C_1 \cap C_2) K$ does not hold in general.
For a counterexample, let $C_1 := \{\pi_1^{(1)}\}$, $C_2 := \{c_0^{(1)}\}$, $K := \{c_0^{(1)}\}$, subsets of $\mathcal{O}_{\{0,1\}}$, where $c_0^{(1)}$ denotes the unary constant function taking value $0$.
Then clearly $C_1 K = C_2 K = \{c_0^{(1)}\}$, so $C_1 K \cap C_2 K = \{c_0^{(1)}\}$, but $(C_1 \cap C_2) K = \emptyset$ because $C_1 \cap C_2 = \emptyset$.
\end{remark}

Recall that a map $c \colon \mathcal{P}(A) \to \mathcal{P}(A)$ (here $\mathcal{P}(A)$ stands for the power set of $A$) is called a \emph{closure operator} if for all $X, Y \subseteq A$, we have $X \subseteq c(X)$ (extensivity), $X \subseteq Y$ implies $c(X) \subseteq c(Y)$ (monotonicity), and $c(c(X)) = c(X)$ (idempotence).
The image $c(X)$ of $X$ is called the \emph{closure} of $X$; we also say that $c(X)$ is \emph{generated} by $X$.
The \emph{closed sets,} i.e., sets of the form $c(X)$, form a complete lattice in which $A$ is the greatest element and the meet operation (of an arbitrary family) coincides with the intersection.
A collection of subsets of $A$ with these properties is called a \emph{closure system} on $A$.
It is well known that the image of every closure operator is a closure system, and, conversely, for every closure system $S$, there exists a closure operator whose closed sets are precisely the elements of $S$.

\begin{definition}
\label{def:clone}
A class $C \subseteq \mathcal{O}_A$ is called a \emph{clone} on $A$ if $C C \subseteq C$ and $C$ contains all projections.
The set of all clones on $A$ is a closure system in which the greatest and least elements are the clone $\mathcal{O}_A$ of all operations on $A$ and the clone of all projections on $A$, respectively.
For any $K \subseteq \mathcal{O}_A$, we denote by $\clonegen{K}$ the clone generated by $K$, i.e., the smallest clone on $A$ containing $K$.
\end{definition}

\begin{definition}
Let $K \subseteq \mathcal{F}_{AB}$, $C_1 \subseteq \mathcal{O}_A$, and $C_2 \subseteq \mathcal{O}_B$.
We say that $K$ is \emph{stable under right composition} with $C_1$ if $K C_1 \subseteq K$,
and that $K$ is \emph{stable under left composition} with $C_2$ is $C_2 K \subseteq K$.
If both $K C_1 \subseteq K$ and $C_2 K \subseteq K$ hold, we say that $K$ is \emph{$(C_1,C_2)$\hyp{}stable.}
If $K, C \subseteq \mathcal{O}_A$ and $K$ is $(C,C)$\hyp{}stable, we say that $K$ is \emph{$C$\hyp{}stable.}

The set of all $(C_1,C_2)$\hyp{}stable subsets of $\mathcal{F}_{AB}$ constitutes a closure system, and for any $K \subseteq \mathcal{F}_{AB}$, we denote by $\gen[(C_1,C_2)]{K}$ the \emph{$(C_1,C_2)$\hyp{}closure} of $K$, i.e., the smallest $(C_1,C_2)$\hyp{}stable class containing $K$.
We also write $\gen[C]{K}$ for $\gen[(C,C)]{K}$ and call it the \emph{$C$\hyp{}closure} of $K$.
\end{definition}

\begin{remark}
A set $K \subseteq \mathcal{F}_{AB}$ is minor\hyp{}closed if and only if it is stable under right composition with the set of all projections on $A$.
Every clone is minor\hyp{}closed.
A clone $C$ is $(C,C)$\hyp{}stable.
\end{remark}

\begin{lemma}
\label{lem:stable-impl-stable}
Let $C_1$ and $C'_1$ be clones on $A$ and $C_2$ and $C'_2$ clones on $B$ such that $C_1 \subseteq C'_1$ and $C_2 \subseteq C'_2$.
Then for every $K \subseteq \mathcal{F}_{AB}$, it holds that if $K$ is $(C'_1,C'_2)$\hyp{}stable then $K$ is $(C_1,C_2)$\hyp{}stable.
\end{lemma}

\begin{proof}
Assume that $K$ is $(C'_1,C'_2)$\hyp{}stable.
Then, in view of Remark~\ref{rem:comp-monot}, we have $K C_1 \subseteq K C'_1 \subseteq K$ and $C_2 K \subseteq C'_2 K \subseteq K$, i.e., $K$ is $(C_1,C_2)$\hyp{}stable.
\end{proof}

%%%%%%%%%%%%%%%%%%%%%%%%%%%%%%%%%%%%%%%%%%%%%%%%%%%%%%%%%%%%

\section{Stability and generators}
\label{sec:stab}

The task of verifying whether a function class is stable under right or left compositions with certain clones may appear complicated because the defining conditions involve compositions with arbitrary members of each clone.
We now develop helpful tools that simplify this task.

For right stability, it is enough to consider closure under minors and certain simple compositions involving only generators of the clone.
In order to formalize this, let us consider the elementary superposition operations $\zeta$ (cyclic shift of arguments), $\tau$ (transposition of the first two arguments), $\Delta$ (identification of arguments or diagonalization), $\nabla$ (introduction of a fictitious argument or cylindrification), and $\ast$ (composition) defined by Mal'cev~\cite{Malcev} (see also \cite[Section~II.1.2]{Lau}).
The algebra $(\mathcal{O}_A; \zeta, \tau, \Delta, \nabla, {\ast})$ is called the \emph{iterative function algebra} on $A$, and its subuniverses are called \emph{closed classes.} Closed classes containing all projections are precisely the clones on $A$.

Let $F \subseteq \mathcal{O}_A$ and $f \in \mathcal{O}_A$.
We say that $f$ is a \emph{superposition} of $F$ if $f$ can be obtained from the members of $F$ by a finite number of applications of the operations $\zeta$, $\tau$, $\Delta$, $\nabla$, $\ast$.

\begin{lemma}
\label{lem:composition-superposition}
For any $f \in \mathcal{O}_A^{(n)}$ and $g_1, \dots, g_n \in \mathcal{O}_A^{(m)}$\hspace{-2.1pt},
\hspace{-1pt}the composition $f(g_1, \dots, g_n)$ is a superposition of $\{f, g_1, \dots, g_n\}$.
\end{lemma}

\begin{proof}
Let $f_0 := (\zeta f) \ast g_n$, and
for $i = 1, \dots, n-1$, let $f_i := (\zeta f_{i-1}) \ast g_{n-i}$.
Then
\begin{align*}
& f_1(x_1, \dots, x_{n+m-1})
= (\zeta f)(g_n(x_1, \dots, x_m), x_{m+1}, \dots, x_{n+m-1})
\\ &= f(x_{m+1}, \dots, x_{n+m-1}, g_n(x_1, \dots, x_m)),
\displaybreak[0] \\
& f_2(x_1, \dots, x_{n+2m-2})
= (\zeta f_1)(g_{n-1}(x_1, \dots, x_m), x_{m+1}, \dots, x_{n+2m-2})
\\ &= f_1(x_{m+1}, \dots, x_{n+2m-2}, g_{n-1}(x_1, \dots, x_m))
\\ &= f(x_{2m+1}, \dots, x_{n+2m-2}, g_{n-1}(x_1, \dots, x_m), g_n(x_{m+1}, \dots, x_{2m})),
\displaybreak[0] \\
& f_3(x_1, \dots, x_{n+3m-3})
= (\zeta f_2)(g_{n-2}(x_1, \dots, x_m), x_{m+1}, \dots, x_{n+3m-3})
\\ &= f_2(x_{m+1}, \dots, x_{n+3m-3}, g_{n-2}(x_1, \dots, x_m))
\\ &= f(x_{3m+1}, \dots, x_{n+3m-3}, \begin{array}[t]{@{}l}g_{n-2}(x_1, \dots, x_m), \\ g_{n-1}(x_{m+1}, \dots, x_{2m}), g_n(x_{2m+1}, \dots, x_{3m})),\end{array}
\\ & \;\;\vdots \\
& f_n(x_1, \dots, x_{nm})
= f(g_1(x_1, \dots, x_m), \begin{array}[t]{@{}l}g_2(x_{m+1}, \dots, x_{2m}), \dots, \\ g_n(x_{(n-1)m + 1}, \dots, x_{nm})).\end{array}
\end{align*}
Let $\theta$ be the composition of elementary operations that identifies arguments $x_i$ and $x_j$ if and only if $i \equiv j \pmod{m}$.
Then
\begin{align*}
\theta f_n (x_1, \dots, x_m)
&= f(g_1(x_1, \dots, x_m), g_2(x_1, \dots, x_m), \dots, g_n(x_1, \dots, x_m))
\\ &
= f(g_1, \dots, g_n)(x_1, \dots, x_m).
\end{align*}
By construction, $f_1, \dots, f_n$ and $\theta f_n = f(g_1, \dots, g_n)$ are superpositions of  $\{f, \linebreak g_1, \dots, g_n\}$.
\end{proof}

\begin{lemma}
\label{lem:right-stab-gen}
Let $F \subseteq \mathcal{O}_A$.
Let $C$ be a clone on $A$, and let $G$ be a generating set of $C$.
Then the following conditions are equivalent.
\begin{enumerate}[label=\upshape{(\roman*)}, leftmargin=*, widest=iii]
\item\label{FCsubF} $F C \subseteq F$.
\item\label{minorC} $F$ is minor\hyp{}closed and $f \ast g \in F$ whenever $f \in F$ and $g \in C$.
\item\label{minorG} $F$ is minor\hyp{}closed and $f \ast g \in F$ whenever $f \in F$ and $g \in G$.
\end{enumerate}
\end{lemma}

\begin{proof}
\ref{FCsubF} $\implies$ \ref{minorG}:
For any $f \in F$, any minor of $f$ is of the form
$f(\pr^{(m)}_{i_1}, \dots, \linebreak \pr^{(m)}_{i_m})$,
for some $m \in \IN$ and $i_1, \dots, i_m \in \nset{m}$.
Since all projections are members of the clone $C$, we have $f(\pr^{(m)}_{i_1}, \dots, \pr^{(m)}_{i_m}) \in F C \subseteq F$.
Thus $F$ is minor\hyp{}closed.

Let $g \in G$ and define $g' := g(\pr^{(m+n-1)}_1, \dots, \pr^{(m+n-1)}_m)$.
Then $g' \in C$, and we have
$f \ast g = f(g',\pr^{(m+n-1)}_{m+1}, \dots, \pr^{(m+n-1)}_{m+n-1}) \in F C \subseteq F$.

\ref{minorG} $\implies$ \ref{minorC}:
Let $g \in C$.
If $g$ is a projection, then for every $f \in F$, the function $f \ast g$ is a minor of $f$, obtained by introducing fictitious arguments, so $f \ast g \in F$ because $F$ is minor\hyp{}closed.
If $g$ is not a projection, then $g$ is a superposition of $G$, that is, there is a term $t$, say $\ell$\hyp{}ary, in the language of iterative algebras and $h_1, \dots, h_\ell \in G$ such that $t^{\mathcal{O}_A}(h_1, \dots, h_\ell) = g$.
We prove by induction on the structure of the term $t$ that for every $f \in F$ it holds that $f \ast g \in F$.
If $t = x_i$, then $t^{\mathcal{O}_A}(h_1, \dots, h_\ell) = h_i \in G$, and the condition $f \ast h_i \in F$ holds by our assumption.
Consider now the case that $t = \varphi u$, where $\varphi \in \{\zeta, \tau, \Delta, \nabla\}$ and $u$ is a term, and assume that $f \ast u^{\mathcal{O}_A}(h_1, \dots, h_\ell) \in F$ for every $f \in F$.
Then also $f \ast t^{\mathcal{O}_A}(h_1, \dots, h_\ell) = f \ast \varphi u^{\mathcal{O}_A}(h_1, \dots, h_\ell) \in F$ for every $f \in F$, because $F$ is minor\hyp{}closed and the following identities hold for any functions $f$ and $h$ (say $h$ is $n$\hyp{}ary):
\begin{align*}
f \ast \zeta h  &= \pi_{(1 \; 2 \; \cdots \; n)} (f \ast h), &
f \ast \tau h   &= \tau (f \ast h), \\
f \ast \Delta h &= \Delta (f \ast h), &
f \ast \nabla h &= \nabla (f \ast h).
\end{align*}
Finally, consider the remaining case that $t = u \ast v$, and assume that $f \ast u^{\mathcal{O}_A}(h_1, \dots, \linebreak  h_\ell) \in F$ and $f \ast v^{\mathcal{O}_A}(h_1, \dots, h_\ell) \in F$ for every $f \in F$.
Then also
\begin{align*}
f \ast t^{\mathcal{O}_A}(h_1, \dots, h_\ell)
&= f \ast (u^{\mathcal{O}_A}(h_1, \dots, h_\ell) \ast v^{\mathcal{O}_A}(h_1, \dots, h_\ell))
\\ &= (f \ast u^{\mathcal{O}_A}(h_1, \dots, h_\ell)) \ast v^{\mathcal{O}_A}(h_1, \dots, h_\ell) \in F
\end{align*}
for every $f \in F$.

\ref{minorC} $\implies$ \ref{FCsubF}:
Let $f \in F^{(n)}$ and $g_1, \dots, g_n \in C^{(m)}$.
A simple inductive argument shows that, in the construction of $f(g_1, \dots, g_n)$ as a superposition of $\{f, g_1, \dots, g_n\}$ given in the proof of Lemma~\ref{lem:composition-superposition}, the functions $f_i$ are in $F$, because $F$ is minor\hyp{}closed and each $f_i$ is of the form $\zeta \varphi \ast \gamma$ for some $\varphi \in F$ and $\gamma \in G$. Finally, $f(g_1, \dots, g_n) = \theta f_n \in F$, because $F$ is minor\hyp{}closed.
\end{proof}

For left stability, it is enough to consider compositions with generators of the clone.

\begin{lemma}
\label{lem:left-stab-gen}
Let $F \subseteq \mathcal{O}_A$.
Let $C$ be a clone on $A$, and let $G$ be a generating set of $C$.
Then the following conditions are equivalent.
\begin{enumerate}[label=\upshape{(\roman*)}, leftmargin=*, widest=iii]
\item\label{CFsubF} $C F \subseteq F$.
\item\label{gf-C} $g(f_1, \dots, f_n) \in F$ whenever $g \in C^{(n)}$ and $f_1, \dots, f_n \in F^{(m)}$ for some $n, m \in \IN$.
\item\label{gf-G} $g(f_1, \dots, f_n) \in F$ whenever $g \in G^{(n)}$ and $f_1, \dots, f_n \in F^{(m)}$ for some $n, m \in \IN$.
\end{enumerate}
\end{lemma}

\begin{proof}
\ref{CFsubF} $\iff$ \ref{gf-C}:
Holds by the definition of function class composition.

\ref{gf-C} $\implies$ \ref{gf-G}:
This is obvious.

\ref{gf-G} $\implies$ \ref{gf-C}:
Let $g \in C$. Then there is a term $t$ of the language of the algebra $\mathbf{A} = (A; G)$ such that $g = t^\mathbf{A}$.
We prove the claim by induction on the structure of the term $t$.
Let $f_1, \dots, f_n \in F^{(m)}$.
The inductive basis holds, because if $t = x_i$, then $t^\mathbf{A} = \pr_i^{(n)}$, and we have $\pr_i^{(n)}(f_1, \dots, f_n) = f_i \in F$.
Consider now the case when $t = h(t_1, \dots, t_\ell)$ for some $h \in G$ and terms $t_1, \dots, t_\ell$, and assume that for $i \in \{1, \dots, \ell\}$, we have already shown that $t_i^\mathbf{A}(f_1, \dots, f_n) \in F$.
It then follows from superassociativity and our assumptions that
\begin{align*}
t^\mathbf{A}(f_1, \dots, f_n)
&= h^\mathbf{A}(t_1^\mathbf{A}, \dots, t_\ell^\mathbf{A})(f_1, \dots, f_n)
\\ &= h^\mathbf{A}(t_1^\mathbf{A}(f_1, \dots, f_n), \dots, t_\ell^\mathbf{A}(f_1, \dots, f_n))
\in F.
\qedhere
\end{align*}
\end{proof}

Let us record here a simple yet useful observation on the $C$\hyp{}stable class generated by a projection.

\begin{lemma}
For any clone $C$, $\gen[C]{\pr_1^{(1)}} = C$.
\end{lemma}

\begin{proof}
Since $\pr_1^{(1)} \in C$ and $C$ is $C$\hyp{}stable, we clearly have $\gen{\pr_1^{(1)}} \subseteq C$.
By Lemma~\ref{lem:right-stab-gen}\ref{minorC}, we also have $f = \pr_1^{(1)} \ast f \in \gen[C]{\pr_1^{(1)}}$ for every $f \in C$, so $C \subseteq \gen[C]{\pr_1^{(1)}}$.
\end{proof}

%%%%%%%%%%%%%%%%%%%%%%%%%%%%%%%%%%%%%%%%%%%%%%%%%%%

\section{Linear stability over finite fields of prime order}
\label{sec:ff}
\renewcommand{\gendefault}{\clL}

In this section we consider classes of operations on a finite field and their right and left stability under clones of linear functions.
Assume that $A = \GF(q)$, a finite field of order $q = p^m$, with $p$ prime.

\begin{definition}
\label{def:polynomials}
It is well known that every $n$\hyp{}ary operation on $A$ is represented by a unique polynomial over $\GF(q)$ in $n$ variables wherein no variable appears with an exponent greater than $q - 1$.
We call such polynomials \emph{reduced polynomials.}
A reduced polynomial can be written as
\begin{equation}
\label{eq:pol}
\sum_{(a_1, \dots, a_n) \in \{0, \dots, q-1\}^n} \alpha_{(a_1, \dots, a_n)} \prod_{i \in \{1, \dots, n\}} x_i^{a_i},
\end{equation}
where each \emph{coefficient} $\alpha_{(a_1, \dots, a_n)}$ is an element of $\GF(q)$.
We will use the shorthand $\alpha_{\vect{a}} x^\vect{a}$ to designate the \emph{monomial} $\alpha_{(a_1, \dots, a_n)} \prod_{i \in \{1, \dots, n\}} x_i^{a_i}$ with $\vect{a} = (a_1, \dots, a_n)$.
A monomial with coefficient $1$ is called \emph{monic.}
The \emph{degree} of a monomial $\alpha_\vect{a} x^\vect{a}$ is $\sum_{i=1}^n a_i$.
The \emph{degree} of a polynomial $P$, denoted $\deg(P)$, is the maximum of the degrees of its monomials with a nonzero coefficient; we agree that $\deg(0) := 0$.
In general, when we speak of the monomials of a polynomial, we mean the monomials with a nonzero coefficient.
As is usual when writing polynomials, we may omit coefficients equal to $1$, and we may omit monomials with coefficient $0$.
Without any risk of confusion, we will denote functions by reduced polynomials.

The \emph{degree} of an operation $f$, denoted $\deg(f)$, is the degree of the unique reduced polynomial representing $f$.
For $k \in \IN$, denote by $\clD{k}$ the set of all operations on $A$ of degree at most $k$.
Clearly, these sets constitute an infinite ascending chain $\clD{0} \subset \clD{1} \subset \clD{2} \subset \cdots$ whose union is the set $\mathcal{O}_A$ of all operations on $A$.
In particular, $\clD{0}$ is the set of all constant operations, and $\clD{1}$ is the set of all \emph{linear} operations.\footnote{Strictly speaking, operations of degree at most $1$ are \emph{affine} in the sense of linear algebra. We go along with the term \emph{linear} that is common in the context of clone theory and especially in the theory of Boolean functions.}
We shall also use the symbol $\clL$ to denote the set $\clD{1}$.
The set $\clL$ is a clone on $A$; in fact, it is a maximal clone if and only if $q$ is prime (see Rosenberg~\cite{Rosenberg}, Szendrei~\cite[Theorem~3.1]{Szendrei-1980}).
\end{definition}

\begin{proposition}
\label{prop:Dk-L-stable}
For every $k \in \IN$, the set $\clD{k}$ is $\clL$\hyp{}stable.
\end{proposition}

\begin{proof}
Noting that the clone $\clL$ is generated by $\{x_1 + x_2\} \cup \{ \, c x_1 \mid c \in A \, \} \cup \{ \, c \mid c \in A \, \}$, we apply Lemmata~\ref{lem:left-stab-gen} and \ref{lem:right-stab-gen}.
The stability under left composition with $\clL$ follows from the fact that for any $f, g \in \clD{k}$ and any $c \in A$ we have $c(f) = c \in \clD{0} \subseteq \clD{k}$, $c x_1 (f) = c \cdot f \in \clD{k}$, and $(x_1 + x_2)(f, g) = f + g \in \clD{k}$.
As for the right stability, note that $\clD{k}$ is minor\hyp{}closed because the formation of minors does not increase the degree of functions, and that for any $f \in \clD{k}$ and for any $c \in A$, it holds that $f \ast c$, $f \ast c x_1$, and $f \ast (x_1 + x_2)$ are members of $\clD{k}$.
\end{proof}

\begin{proposition}
\label{prop:empty-full-L-stable}
The empty set $\emptyset$ and the set $\mathcal{O}_A$ of all operations on $A$ are $\clL$\hyp{}stable.
\end{proposition}

\begin{proof}
This is obvious.
\end{proof}

\begin{lemma}
\label{lem:constants-L-stable}
Every nonempty $\clL$\hyp{}stable class contains all constant functions.
\end{lemma}

\begin{proof}
Let $K$ be a nonempty $\clL$\hyp{}stable class.
Since $\clL$ contains all projections of any arity, $K \clL$ contains functions of any arity, and so does $K$ because $K \clL \subseteq K$.
Note that for any $g_1, \dots, g_n \in \mathcal{O}_A^{(m)}$, it holds that $\cf{n}{b}(g_1, \dots, g_n) = \cf{m}{b}$.
Since all constant functions are members of $\clL$ and $K$ contains functions of any arity, it follows that $\clL K$ contains all constant functions, and so does $K$ because $\clL K \subseteq K$.
\end{proof}

\begin{lemma}
\label{lem:x1..xk}
For any $k \in \IN$, $\gen{x_1 x_2 \dots x_k} = \clD{k}$.
\end{lemma}

\begin{proof}
Clearly $x_1 x_2 \dots x_k \in \clD{k}$ and the class $\clD{k}$ is $\clL$-stable by Proposition~\ref{prop:Dk-L-stable}, so we have $\gen{x_1 x_2 \dots x_k} \subseteq \clD{k}$.
By identification of variables, permutation of variables, and substitution of constant $1$ for variables, we obtain every monic monomial of degree at most $k$.
By taking the sum of monic monomials of degree at most $k$, with suitable coefficients, we can obtain any polynomial of degree at most $k$, in other words, by composing a suitable linear function with functions represented by monic monomials of degree at most $k$, we obtain any function of degree at most $k$.
Therefore, $\clD{k} \subseteq \gen{x_1 x_2 \dots x_k}$.
\end{proof}

In the remainder of this section, we will assume that $A$ is a finite field $\GF(p)$ of prime order $p$.

\begin{lemma}
\label{lem:Ddegf}
Assume $A = \GF(p)$ with $p$ prime.
If the reduced polynomial of $f \colon A^n \to A$ has degree $k$, then $\gen{f} = \clD{k}$.
\end{lemma}

\begin{proof}
Let $P$ be the reduced polynomial representing $f$ as in \eqref{eq:pol}.
Let 
$\vect{u} = (u_1, \dots,  \linebreak  u_n) \in \{0, \dots, p-1\}^n$
be such that $\alpha_\vect{u} x^\vect{u}$ has degree $k$ and $\alpha_\vect{u} \neq 0$.
We may assume that $\alpha_\vect{u} = 1$, because by composing $f$ from the left by $\alpha_\vect{u}^{-1} x_1$, which belongs to $\clL$, we obtain a function in $\gen{f}$ that has the same monomials as $f$ but with coefficients multiplied by $\alpha_\vect{u}^{-1}$.

Let $U := \{ \, i \in \nset{n} \mid u_i \neq 0 \, \}$.
By substituting $0$ for the variables $x_i$ with $i \in \nset{n} \setminus U$, we obtain a function $f'$ in $\gen{f}$ with reduced polynomial $P'$ such that $P'$ has degree $k$ and contains only variables $x_i$ with $i \in U$, and $\alpha_\vect{u} x^\vect{u}$ is a monomial of degree $k$ in $P'$.
We may consider the function $f'$ in place of $f$ and assume, without loss of generality, that $U = \nset{n}$.

Let $\{B_1, \dots, B_n\}$ be a partition of $\nset{k}$ in $n$ parts such that $\card{B_j} = u_j$ for all $j \in \nset{n}$.
For $j \in \nset{n}$, let $g_j = \sum_{i \in B_j} x_i$. Note that $g_j \in \clL$.
Consider the function $h := f(g_1, \dots, g_n)$, which is in $\gen{f}$.
For every $\vect{a} \in \{0, \dots, p-1\}^n$ with $\sum_{i=1}^n a_i \leq k$, the expansion of the product $\prod_{i=1}^n g_i^{a_i}$ results in a polynomial of degree at most $k$ in which no monomial contains all variables $x_1, \dots, x_k$, with the exception of $\vect{a} = \vect{u}$, for which the expansion yields a polynomial in which one of the monomials is $x_1 \dots x_k$ and the other monomials do not contain all variables $x_1, \dots, x_k$.
Consequently, $h = x_1 \dots x_k + h'$ where $h'$ is a polynomial in variables $x_1, \dots, x_k$ in which no monomial contains all variables $x_1, \dots, x_k$.

Now, let us define a sequence of functions $h_0, \dots, h_k$ recursively as follows: $h_0 := h$.
For $i = 1, \dots, k$, let $h_i := h_{i-1} - h_{i-1}(x_1, \dots, x_{i-1}, 0, x_{i+1}, \dots, x_k)$.
We have $h_i \in \gen{h_{i-1}}$.
It is easy to see that the polynomial of $h_i$ can be obtained from the polynomial of $h_{i-1}$ by removing all monomials in which $x_i$ does not occur.
Consequently, $x_1 \dots x_k = h_k \in \gen{h_{k-1}} \subseteq \gen{h_{k-2}} \subseteq \dots \subseteq \gen{h_0} \subseteq \gen{f}$.
Now it follows from Lemma~\ref{lem:x1..xk} that $\clD{k} = \gen{x_1 \dots x_k} \subseteq \gen{f} \subseteq \clD{k}$.
\end{proof}

\begin{lemma}
\label{lem:gen-L-stable}
Assume $A = \GF(p)$ with $p$ prime.
Let $K \subseteq \mathcal{O}_A$, $K \neq \emptyset$.
If the set $\{ \, \deg(f) \mid f \in K \, \}$ has a maximum $m$, then $\gen{K} = \clD{m}$.
Otherwise $\gen{K} = \mathcal{O}_A$.
\end{lemma}

\begin{proof}
If said maximum $m$ exists, we have $K \subseteq \clD{m}$ and there exists a $g \in K$ with $\deg(g) = m$.
Since $\clD{m}$ is $\clL$\hyp{}stable by Lemma~\ref{prop:Dk-L-stable}, we have $\gen{K} \subseteq \clD{m}$.
Lemma~\ref{lem:Ddegf} implies
\[
\clD{m}
= \clD{\deg(g)}
\subseteq \bigcup_{f \in K} \clD{\deg(f)}
= \bigcup_{f \in K} \gen{f}
\subseteq \gen{K}
\subseteq \clD{m}.
\]

Otherwise there is no finite upper bound on the degrees of the members of $K$.
Then for every $i \in \IN$, there exists an $f_i \in K$ with $\deg(f_i) \geq i$.
Now we have
\[
\mathcal{O}_A
= \bigcup_{i \in \IN} \clD{i}
\subseteq \bigcup_{i \in \IN} \clD{\deg(f_i)}
= \bigcup_{i \in \IN} \gen{f_i}
\subseteq \gen{K}
\subseteq \mathcal{O}_A.
\qedhere
\]
\end{proof}

\begin{theorem}
\label{thm:L-stable:FF}
Assume $A = \GF(p)$ with $p$ prime.
The $\clL$\hyp{}stable classes are $\mathcal{O}_A$, $\clD{k}$ for $k \in \IN$, and $\emptyset$.
\end{theorem}

\begin{proof}
The classes mentioned in the statement are $\clL$\hyp{}stable by Propositions~\ref{prop:Dk-L-stable} and \ref{prop:empty-full-L-stable}.
Lemma~\ref{lem:gen-L-stable} implies that there are no further $\clL$\hyp{}stable classes.
\end{proof}

%%%%%%%%%%%%%%%%%%%%%%%%%%%%%%%%%%%%%%%%%%%%%%%%%%%

\section{Boolean functions}
\label{sec:Bf}

In this section, we define various properties and classes of Boolean functions.
For easy reference, we have collected in Table~\ref{table:notation} the notation used for the function classes.

\begin{definition}
\label{def:all}
Operations on $\{0,1\}$ are called \emph{Boolean functions.}
The class of all Boolean functions is denoted by $\clAll$.
\end{definition}

\begin{definition}
\label{def:Zhegalkin}
By particularizing Definition~\ref{def:polynomials} to the two\hyp{}element field, we obtain that
every Boolean function is represented by a unique \emph{multilinear polynomial} over the two\hyp{}element field, i.e., a polynomial with coefficients in $\GF(2)$ in which no variable appears with an exponent greater than $1$.
Since the coefficients come from the set $\{0,1\}$, every monomial with a nonzero coefficient is monic.
The unique multilinear polynomial representing a Boolean function $f$ is known as the \emph{Zhegalkin polynomial} of $f$, and it can be written as
\begin{equation}
\sum_{S \in \monomials{f}} x_S,
\label{eq:Zhegalkin}
\end{equation}
where $x_S$ is a shorthand for $\prod_{i \in S} x_i$ and $\monomials{f} \subseteq \mathcal{P}(\nset{n})$.
Note that $x_\emptyset = 1$ and $\sum_{S \in \emptyset} x_S = 0$.
The terms $x_S$ with $S \neq \emptyset$ are called \emph{monomials.}
If $\emptyset \in \monomials{f}$, then we say that $f$ has \emph{constant term $1$}; otherwise $f$ has \emph{constant term $0$}.
Without any risk of confusion, we will denote Boolean functions by their Zhegalkin polynomials, and we refer to the set $\monomials{f}$ as the \emph{set of monomials} of $f$.
\end{definition}

\begin{definition}
\label{def:Boolean-functions}
Some well\hyp{}known Boolean functions are defined in Table~\ref{table:Bfs}:
modulo\hyp{}$2$ addition $+$,
conjunction $\wedge$,
disjunction $\vee$,
triple sum $\oplus_3$,
median $\mu$.
Their Zhegalkin polynomials are the following:
\begin{gather*}
x_1 + x_2,
\qquad\qquad
x_1 \wedge x_2 = x_1 x_2,
\qquad\qquad
x_1 \vee x_2 = x_1 x_2 + x_1 + x_2,
\\
\mathord{\oplus_3}(x_1,x_2,x_3) = x_1 + x_2 + x_3,
\qquad
\mu(x_1,x_2,x_3) = x_1 x_2 + x_1 x_3 + x_2 x_3.
\end{gather*}
The triple sum is the Mal'cev operation of the abelian group $(\mathrm{GF}(2),{+})$, and
the triple sum and the median are the minority and the majority operations on $\{0,1\}$, respectively; see Section~\ref{sec:remarks}.
\end{definition}

\begin{table}
\begin{tabular}[t]{ccccc}
\toprule
$x$ & $y$ & $x + y$ & $x \wedge y$ & $x \vee y$ \\
\midrule
0 & 0 & 0 & 0 & 0 \\
0 & 1 & 1 & 0 & 1 \\
1 & 0 & 1 & 0 & 1 \\
1 & 1 & 0 & 1 & 1 \\
\bottomrule
\end{tabular}
\qquad
\begin{tabular}[t]{ccccc}
\toprule
$x$ & $y$ & $z$ & $\oplus_3(x,y,z)$ & $\mu(x,y,z)$ \\
\midrule
0 & 0 & 0 & 0 & 0 \\
0 & 0 & 1 & 1 & 0 \\
0 & 1 & 0 & 1 & 0 \\
0 & 1 & 1 & 0 & 1 \\
1 & 0 & 0 & 1 & 0 \\
1 & 0 & 1 & 0 & 1 \\
1 & 1 & 0 & 0 & 1 \\
1 & 1 & 1 & 1 & 1 \\
\bottomrule
\end{tabular}

\bigskip
\caption{Well\hyp{}known Boolean functions}
\label{table:Bfs}
\end{table}

\begin{definition}
\label{def:ab-eq-neq}
For $a, b \in \{0, 1\}$, let
\begin{align*}
\clVal{a}{} &:= \{ \, f \in \clAll \mid f(0, \dots, 0) = a \, \}, \\
\clVal{}{b} &:= \{ \, f \in \clAll \mid f(1, \dots, 1) = b \, \},
\end{align*}
and let $\clVal{a}{b} := \clVal{a}{} \cap \clVal{}{b}$.
Furthermore, define
\begin{align*}
\clEq &:= \{ \, f \in \clAll \mid f(0, \dots, 0) = f(1, \dots, 1) \, \}, \\
\clNeq &:= \{ \, f \in \clAll \mid f(0, \dots, 0) \neq f(1, \dots, 1) \, \},
\end{align*}
that is, $\clEq = \clVal{0}{0} \cup \clVal{1}{1}$ and $\clNeq = \clVal{0}{1} \cup \clVal{1}{0}$.

Clearly $\clVal{0}{} \cap \clVal{1}{} = \clEmpty$ and $\clVal{0}{} \cup \clVal{1}{} = \clAll$;
similarly, $\clVal{}{0} \cap \clVal{}{1} = \clEmpty$ and $\clVal{}{0} \cup \clVal{}{1} = \clAll$,
and $\clEq \cap \clNeq = \clEmpty$ and $\clEq \cup \clNeq = \clAll$.
It is easy to see that $\clVal{a}{}$ is the class of all Boolean functions with constant term $a$.
\end{definition}

\begin{definition}
\label{def:T0T1Tc}
For $a \in \{0,1\}$, a Boolean function $f$ is \emph{$a$\hyp{}preserving} if $f(a, \dots, a) = a$.
A function is \emph{constant\hyp{}preserving} if it is both $0$- and $1$\hyp{}preserving.
We denote the classes of all $0$\hyp{}preserving, of all $1$\hyp{}preserving, and of all constant\hyp{}preserving functions by $\clTo$, $\clTi$, and $\clTc$, respectively.
Note that $\clTc = \clTo \cap \clTi$.
It follows from the definitions that $\clTo = \clCon{0}$, $\clTi = \clYksi{1}$, and $\clTc = \clVal{0}{1}$.
\end{definition}

\begin{remark}
\label{rem:multiple-notation}
The reason why we have introduced multiple notation for the classes $\clTo = \clCon{0}$ and $\clTi = \clYksi{1}$ is to facilitate writing certain statements in a parameterized form and to make reference, as the case may be, to either the classes $\clCon{a}$ ($a \in \{0,1\}$), $\clYksi{b}$ ($b \in \{0,1\}$), or $\clTa{a}$ ($a \in \{0,1\}$).
\end{remark}

\begin{definition}
\label{def:parity}
The \emph{parity} of a Boolean function $f$, denoted $\parity{f}$, is a number, either $0$ or $1$, which is given by
\[
\parity{f} := \card{ \monomials{f} \setminus \{ \emptyset \} } \bmod{2}.
\]
We call a function \emph{even} or \emph{odd} if its parity is $0$ or $1$, respectively.
Note that $\clEven$ and $\clOdd$ are precisely the classes of even and odd functions, respectively.
\end{definition}

\begin{definition}
\label{def:monotone}
The set $\{0,1\}$ is endowed with the natural order $\leq$, with $0 < 1$, which induces the componentwise order, also denoted by $\leq$, on the Cartesian power $\{0,1\}^n$: for $(a_1, \dots, a_n), (b_1, \dots, b_n) \in \{0,1\}^n$, $(a_1, \dots, a_n) \leq (b_1, \dots, b_n)$ if and only if $a_i \leq b_i$ for all $i \in \nset{n}$.

A Boolean function $f \colon \{0,1\}^n \to \{0,1\}$ is \emph{monotone} if $f(\vect{a}) \leq f(\vect{b})$ whenever $\vect{a} \leq \vect{b}$.
We denote by $\clM$ the class of all monotone Boolean functions.
\end{definition}

\begin{definition}
\label{def:S-R}
For $a \in \{0,1\}$, let $\overline{a}$ denote the \emph{negation} of $a$, that is, $\overline{a} := 1 - a$, and for $\vect{a} = (a_1, \dots, a_n) \in \{0,1\}^n$, let $\overline{\vect{a}} := (\overline{a_1}, \dots, \overline{a_n})$.
For any $f \in \clAll^{(n)}$, denote by $\overline{f}$, $f^\mathrm{n}$, and $f^\mathrm{d}$ the \emph{\textup{(}outer\textup{)} negation,} the \emph{inner negation,} and the \emph{dual} of $f$, that is, the functions $\overline{f}, f^\mathrm{n}, f^\mathrm{d} \in \clAll^{(n)}$ with $\overline{f}(\vect{a}) = \overline{f(\vect{a})}$, $f^\mathrm{n}(\vect{a}) = f(\overline{\vect{a}})$, and $f^\mathrm{d}(\vect{a}) = \overline{f(\overline{\vect{a}})}$ for all $\vect{a} \in \{0,1\}^n$.
For $C \subseteq \clAll$, let $\overline{C} := \{ \, \overline{f} \mid f \in C \, \}$.

A function $f$ is \emph{self\hyp{}dual} if $f = f^\mathrm{d}$, i.e., $f(\vect{a}) = \overline{f(\overline{\vect{a}})}$ for all $\vect{a} \in \{0,1\}^n$.
A function $f$ is \emph{reflexive} (or \emph{self\hyp{}anti\hyp{}dual}) if $f = f^\mathrm{n}$, i.e., $f(\vect{a}) = f(\overline{\vect{a}})$ for all $\vect{a} \in \{0,1\}^n$.
We denote by $\clS$ the class of all self\hyp{}dual functions.
Furthermore, define $\clSc := \clS \cap \clTc$ and $\clSM := \clS \cap \clM$, the classes of constant\hyp{}preserving self\hyp{}dual functions and monotone self\hyp{}dual functions, respectively.
\end{definition}

\begin{lemma}
\label{lem:neg-pol}
Let $f, g \in \clAll^{(n)}$.
\begin{enumerate}[label=\upshape{(\roman*)}, leftmargin=*, widest=iii]
\item\label{lem:neg-pol:outer}
$\overline{f} = f + 1$.

\item\label{lem:neg-pol:inner-sum}
$(f + g)^\mathrm{n} = f^\mathrm{n} + g^\mathrm{n}$.

\item\label{lem:neg-pol:inner}
$f^\mathrm{n} = \sum_{S \in \monomials{f}} \sum_{T \subseteq S} x_T$.

\item\label{lem:neg-pol:monomials}
Consequently, $\monomials{f}$ and $\monomials{f^\mathrm{n}}$ have the same maximal elements with respect to subset inclusion.
\end{enumerate}
\end{lemma}

\begin{proof}
\ref{lem:neg-pol:outer}
This is obvious.

\ref{lem:neg-pol:inner-sum}
It follows from the definition that for all $\vect{a} \in \{0,1\}^n$, we have
\[
(f + g)^\mathrm{n}(\vect{a})
= (f + g)(\overline{\vect{a}})
= f(\overline{\vect{a}}) + g(\overline{\vect{a}})
= f^\mathrm{n}(\vect{a}) + g^\mathrm{n}(\vect{a}).
\]

\ref{lem:neg-pol:inner}
Consider the function $g$ with Zhegalkin polynomial $x_S = \prod_{i \in S} x_i$ for some $S \subseteq \nset{n}$.
The inner negation of $g$ is given by $\prod_{i \in S} (x_i + 1)$.
The expansion of this product yields $\sum_{T \subseteq S} x_T$.
The claim now follows by part \ref{lem:neg-pol:inner-sum}.

\ref{lem:neg-pol:monomials}
This is clear, as the maximal monomials of $f$ will not cancel out in the summation of part \ref{lem:neg-pol:inner}.
\end{proof}

\begin{definition}
\label{def:Dk}
By particularizing the definition of degree (see Definition~\ref{def:polynomials}) to monomials and polynomials over $\GF(2)$, we obtain that the \emph{degree} of a monomial $x_S$ is just $\card{S}$, and
the \emph{degree} of a Boolean function $f$ is the size of the largest monomial in the Zhegalkin polynomial of $f$, i.e., $\deg(f) := \max_{S \in \monomials{f}} \card{S}$ for $f \neq 0$, and we agree that $\deg(0) := 0$.
As before, for $k \in \IN$, we denote by $\clD{k}$ the class of all Boolean functions of degree at most $k$.
Clearly $\clD{k} \subsetneq \clD{k+1}$ for all $k \in \IN$.

A Boolean function $f$ is \emph{linear} if $\deg(f) \leq 1$.
We denote by $\clL$ the class of all linear functions.
Thus $\clL = \clD{1}$.
We also let
\begin{align*}
\clLo &:= \clL \cap \clTo = \clL \cap \clCon{0}, &
\clLi &:= \clL \cap \clTi = \clL \cap \clYksi{1}, \\
\clLS &:= \clL \cap \clS = \clL \cap \clOdd, &
\clLc &:= \clL \cap \clTc = \clL \cap \clBoth{0}{1}.
\end{align*}
The equalities in the above definitions are clear by Remark~\ref{rem:multiple-notation}, except for the equality $\clLS = \clL \cap \clOdd$ which is easy to verify and also follows from Lemma~\ref{lem:Selezneva} below.
\end{definition}

\begin{definition}
\label{def:Xk}
Let $f$ be an $n$\hyp{}ary Boolean function.
The \emph{characteristic} of a set $S \subseteq \nset{n}$ in $f$ is a number, either $0$ or $1$, which is given by
\[
\Char(S,f) := \card{ \{ \, A \in \monomials{f} \mid S \subsetneq A \, \} } \bmod{2}.
\]
The \emph{characteristic rank} of $f$, denoted by $\charrank{f}$, is the smallest integer $m$ such that $\Char(S,f) = 0$ for all subsets $S \subseteq \nset{n}$ with $\card{S} \geq m$.
Clearly $\charrank{f} \leq n$ because $\Char(\nset{n},f) = 0$.

For $k \in \IN$, denote by $\clChar{k}$ the class of all Boolean functions of characteristic rank at most $k$.
For any $k \in \IN$, we have $\clChar{k} \subsetneq \clChar{k+1}$.
The inclusion is proper, as witnessed by the function $x_1 \dots x_{k+1} \in \clChar{k+1} \setminus \clChar{k}$.
Moreover, for any $k \in \IN$, we have $\clD{k} \subseteq \clChar{k}$.
\end{definition}

The characteristic rank of a function has an equivalent description in terms of the degree of a certain derived function.

\begin{lemma}
\label{lem:Xk-description}
A Boolean function $f$ satisfies 
$\charrank{f} = k$
if and only if
the function $\varphi := f + f^\mathrm{n}$ satisfies $\deg(\varphi) = k - 1$.
\textup{(}Here $\deg(0) = -1$.\textup{)}
\end{lemma}

\begin{proof}
Let $f \in \clAll^{(n)}$, and let $\varphi := f + f^\mathrm{n}$.
By Lemma~\ref{lem:neg-pol}\ref{lem:neg-pol:inner}, we have
\[
\varphi = \sum_{S \in \monomials{f}} x_S + \sum_{S \in \monomials{f}} \sum_{T \subseteq S} x_T = \sum_{S \in \monomials{f}} \sum_{T \subsetneq S} x_T.
\]
From this expression, we see that $A \in \monomials{\varphi}$ if and only if
$\card{ \{ \, S \in \monomials{f} \mid A \subsetneq S \, \} } \pmod{2} = 1$,
i.e., $\Char(A,f) = 1$.
Consequently, $\deg(\varphi) = k - 1$ if and only if
$k$ is the smallest integer $m$ such that
for all subsets $A \subseteq \nset{n}$ with $\card{A} \geq m$, we have $\Char(A,f) = 0$, i.e., $\charrank{f} = k$.
\end{proof}

Reflexive and self\hyp{}dual functions have a beautiful characterization in terms of the characteristic rank.

\begin{lemma}[{Selezneva, Bukhman \cite[Lemmata 3.1, 3.5]{SelBuk-2016}}]
\label{lem:Selezneva}
\leavevmode
\begin{enumerate}[label=\upshape{(\roman*)}, leftmargin=*, widest=iii]
\item A Boolean function $f$ is reflexive if and only if $\charrank{f} = 0$.
\item A Boolean function $f$ is self\hyp{}dual if and only if $f + x_1$ is reflexive.
\item A Boolean function $f$ is self\hyp{}dual if and only if $f$ is odd and $\charrank{f} = 1$.
\end{enumerate}
\end{lemma}

In other words, $\clChar{0} = \clChar{1} \cap \clEven$ is the class of all reflexive functions, $\clChar{1} \cap \clOdd$ is the class of all self\hyp{}dual functions, and $\clChar{1}$ is the class of all self\hyp{}dual or reflexive functions.

\begin{definition}
\label{def:conj-disj}
Let $\clLambdac$ and $\clVc$ denote the classes of all conjunctions of arguments and of all disjunctions of arguments, respectively, that is,
\begin{align*}
\clLambdac &:= \{ \, f \in \clAll^{(n)} \mid n \in \IN_{+},\, \emptyset \neq \{i_1, \dots, i_r\} \subseteq \nset{n}, \, f(a_1, \dots, a_n) = a_{i_1} \wedge \dots \wedge a_{i_r} \, \}, \\
\clVc &:= \{ \, f \in \clAll^{(n)} \mid n \in \IN_{+},\, \emptyset \neq \{i_1, \dots, i_r\} \subseteq \nset{n}, \, f(a_1, \dots, a_n) = a_{i_1} \vee \dots \vee a_{i_r} \, \}.
\end{align*}
Let $\clIc$, $\clIo$, $\clIi$, and $\clIstar$ denote the class of all projections, the class of all projections and constant $0$ functions, the class of all projections and constant $1$ functions, and the class of all projections and negated projections, respectively, that is,
\begin{align*}
\clIc &:= \{ \, \pr_i^{(n)} \mid i, n \in \IN_{+},\, 1 \leq i \leq n \, \}, \\
\clIo &:= \clIc \cup \{ \, \cf{n}{0} \mid n \in \IN_{+} \, \}, \\
\clIi &:= \clIc \cup \{ \, \cf{n}{1} \mid n \in \IN_{+} \, \}, \\
\clIstar &:= \clIc \cup \overline{\clIc}.
\end{align*}
\end{definition}

It was shown by Post~\cite{Post} that there are a countably infinite number of clones of Boolean functions.
In this paper, we will only need a handful of them, namely the clones
$\clAll$, $\clTo$, $\clTi$, $\clTc$, $\clM$, $\clS$, $\clSc$, $\clSM$, $\clL$, $\clLo$, $\clLi$, $\clLS$, $\clLc$, $\clLambdac$, $\clVc$, $\clIstar$, $\clIo$, $\clIi$, and $\clIc$ that were defined above.
The lattice of clones of Boolean functions, the so-called \emph{Post's lattice,} is shown in Figure~\ref{fig:Post}, and the above\hyp{}mentioned clones are indicated in the diagram.
In what follows, we will often make use of the following generating sets for some of these clones.
\begin{align*}
\clAll &= \clonegen{x_1 x_2 + 1},
&
\clS &= \clonegen{\mu, \, x_1 + 1},
&
\clSM &= \clonegen{\mu},
&
\clL &= \clonegen{x_1 + x_2, 1},
\\
\clLS &= \clonegen{\mathord{\oplus_3}, \, x_1 + 1},
&
\clLc &= \clonegen{\mathord{\oplus_3}},
&
\clLambdac &= \clonegen{\mathord{\wedge}},
&
\clVc &= \clonegen{\mathord{\vee}},
\\
\clIstar &= \clonegen{x_1 + 1},
&
\clIo &= \clonegen{0},
&
\clIi &= \clonegen{1},
&
\clIc &= \clonegen{\emptyset}.
\end{align*}

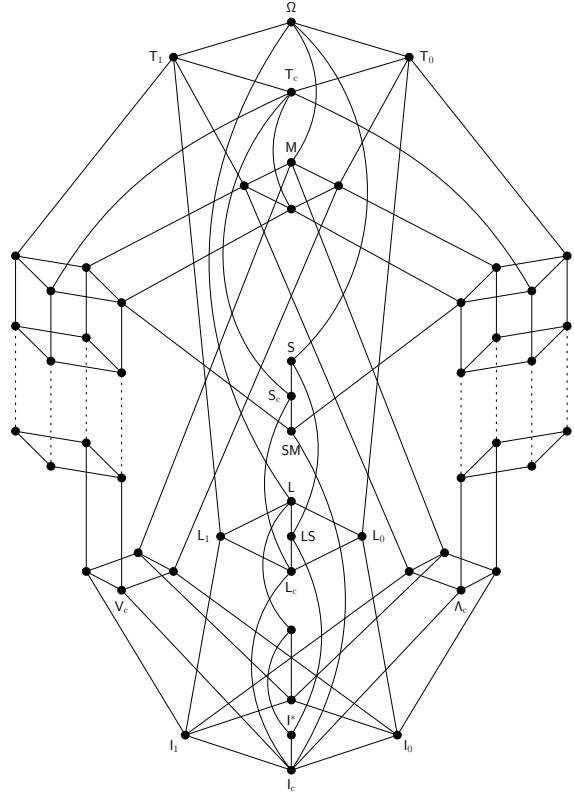
\begin{figure}
\begin{center}
\scalebox{0.31}{%
\tikzstyle{every node}=[circle, draw, fill=black, scale=1, font=\huge]
\begin{tikzpicture}[baseline, scale=1]
   \node [label = below:$\clIc$] (Ic) at (0,-1) {};
   \node [label = above:$\clIstar$] (Istar) at (0,0.5) {};
   \node [label = below right:$\clIo$] (I0) at (4.5,0.5) {};
   \node [label = below left:$\clIi$] (I1) at (-4.5,0.5) {};
   \node (I) at (0,2) {};
   \node (Omega1) at (0,5) {};
   \node [label = below:$\clLc$] (Lc) at (0,7.5) {};
   \node [label = right:$\clLS$] (LS) at (0,9) {};
   \node [label = right:$\clLo$] (L0) at (3,9) {};
   \node [label = left:$\clLi$] (L1) at (-3,9) {};
   \node [label = above:$\clL$] (L) at (0,10.5) {};
   \node [label = below:$\clSM$] (SM) at (0,13.5) {};
   \node [label = left:$\clSc$] (Sc) at (0,15) {};
   \node [label = above:$\clS$] (S) at (0,16.5) {};
   \node (Mc) at (0,23) {};
   \node (M0) at (2,24) {};
   \node (M1) at (-2,24) {};
   \node [label = above:$\clM$] (M) at (0,25) {};
   \node [label = below:$\clLambdac$] (Lamc) at (7.2,6.7) {};
   \node (Lam1) at (5,7.5) {};
   \node (Lam0) at (8.7,7.5) {};
   \node (Lam) at (6.5,8.3) {};
   \node (McUi) at (7.2,11.5) {};
   \node (MUi) at (8.7,13) {};
   \node (TcUi) at (10.2,12) {};
   \node (Ui) at (11.7,13.5) {};
   \node (McU3) at (7.2,16) {};
   \node (MU3) at (8.7,17.5) {};
   \node (TcU3) at (10.2,16.5) {};
   \node (U3) at (11.7,18) {};
   \node (McU2) at (7.2,19) {};
   \node (MU2) at (8.7,20.5) {};
   \node (TcU2) at (10.2,19.5) {};
   \node (U2) at (11.7,21) {};
   \node [label = below:$\clVc$] (Vc) at (-7.2,6.7) {};
   \node (V0) at (-5,7.5) {};
   \node (V1) at (-8.7,7.5) {};
   \node (V) at (-6.5,8.3) {};
   \node (McWi) at (-7.2,11.5) {};
   \node (MWi) at (-8.7,13) {};
   \node (TcWi) at (-10.2,12) {};
   \node (Wi) at (-11.7,13.5) {};
   \node (McW3) at (-7.2,16) {};
   \node (MW3) at (-8.7,17.5) {};
   \node (TcW3) at (-10.2,16.5) {};
   \node (W3) at (-11.7,18) {};
   \node (McW2) at (-7.2,19) {};
   \node (MW2) at (-8.7,20.5) {};
   \node (TcW2) at (-10.2,19.5) {};
   \node (W2) at (-11.7,21) {};
   \node [label = above:$\clTc$] (Tc) at (0,28) {};
   \node [label = right:$\clTo$] (T0) at (5,29.5) {};
   \node [label = left:$\clTi$] (T1) at (-5,29.5) {};
   \node [label = above:$\clAll$] (Omega) at (0,31) {};
   \draw [thick] (Ic) -- (Istar) to[out=135,in=-135] (Omega1);
   \draw [thick] (I) -- (Omega1);
   \draw [thick] (Omega1) to[out=135,in=-135] (L);
   \draw [thick] (Ic) -- (I0) -- (I);
   \draw [thick] (Ic) -- (I1) -- (I);
   \draw [thick] (Ic) to[out=128,in=-134] (Lc);
   \draw [thick] (Ic) to[out=58,in=-58] (SM);
   \draw [thick] (I0) -- (L0);
   \draw [thick] (I1) -- (L1);
   \draw [thick] (Istar) to[out=60,in=-60] (LS);
   \draw [thick] (Ic) -- (Lamc);
   \draw [thick] (I0) -- (Lam0);
   \draw [thick] (I1) -- (Lam1);
   \draw [thick] (I) -- (Lam);
   \draw [thick] (Ic) -- (Vc);
   \draw [thick] (I0) -- (V0);
   \draw [thick] (I1) -- (V1);
   \draw [thick] (I) -- (V);
   \draw [thick] (Lamc) -- (Lam0) -- (Lam);
   \draw [thick] (Lamc) -- (Lam1) -- (Lam);
   \draw [thick] (Lamc) -- (McUi);
   \draw [thick] (Lam0) -- (MUi);
   \draw [thick] (Lam1) -- (M1);
   \draw [thick] (Lam) -- (M);
   \draw [thick] (Vc) -- (V0) -- (V);
   \draw [thick] (Vc) -- (V1) -- (V);
   \draw [thick] (Vc) -- (McWi);
   \draw [thick] (V0) -- (M0);
   \draw [thick] (V1) -- (MWi);
   \draw [thick] (V) -- (M);
   \draw [thick] (McUi) -- (TcUi) -- (Ui);
   \draw [thick] (McUi) -- (MUi) -- (Ui);
   \draw [thick,loosely dashed] (McUi) -- (McU3);
   \draw [thick,loosely dashed] (MUi) -- (MU3);
   \draw [thick,loosely dashed] (TcUi) -- (TcU3);
   \draw [thick,loosely dashed] (Ui) -- (U3);
   \draw [thick] (McU3) -- (TcU3) -- (U3);
   \draw [thick] (McU3) -- (MU3) -- (U3);
   \draw [thick] (McU3) -- (McU2);
   \draw [thick] (MU3) -- (MU2);
   \draw [thick] (TcU3) -- (TcU2);
   \draw [thick] (U3) -- (U2);
   \draw [thick] (McU2) -- (TcU2) -- (U2);
   \draw [thick] (McU2) -- (MU2) -- (U2);
   \draw [thick] (McU2) -- (Mc);
   \draw [thick] (MU2) -- (M0);
   \draw [thick] (TcU2) to[out=120,in=-25] (Tc);
   \draw [thick] (U2) -- (T0);

   \draw [thick] (McWi) -- (TcWi) -- (Wi);
   \draw [thick] (McWi) -- (MWi) -- (Wi);
   \draw [thick,loosely dashed] (McWi) -- (McW3);
   \draw [thick,loosely dashed] (MWi) -- (MW3);
   \draw [thick,loosely dashed] (TcWi) -- (TcW3);
   \draw [thick,loosely dashed] (Wi) -- (W3);
   \draw [thick] (McW3) -- (TcW3) -- (W3);
   \draw [thick] (McW3) -- (MW3) -- (W3);
   \draw [thick] (McW3) -- (McW2);
   \draw [thick] (MW3) -- (MW2);
   \draw [thick] (TcW3) -- (TcW2);
   \draw [thick] (W3) -- (W2);
   \draw [thick] (McW2) -- (TcW2) -- (W2);
   \draw [thick] (McW2) -- (MW2) -- (W2);
   \draw [thick] (McW2) -- (Mc);
   \draw [thick] (MW2) -- (M1);
   \draw [thick] (TcW2) to[out=60,in=-155] (Tc);
   \draw [thick] (W2) -- (T1);

   \draw [thick] (SM) -- (McU2);
   \draw [thick] (SM) -- (McW2);

   \draw [thick] (Lc) -- (LS) -- (L);
   \draw [thick] (Lc) -- (L0) -- (L);
   \draw [thick] (Lc) -- (L1) -- (L);
   \draw [thick] (Lc) to[out=120,in=-120] (Sc);
   \draw [thick] (LS) to[out=60,in=-60] (S);
   \draw [thick] (L0) -- (T0);
   \draw [thick] (L1) -- (T1);
   \draw [thick] (L) to[out=125,in=-125] (Omega);
   \draw [thick] (SM) -- (Sc) -- (S);
   \draw [thick] (Sc) to[out=142,in=-134] (Tc);
   \draw [thick] (S) to[out=42,in=-42] (Omega);
   \draw [thick] (Mc) -- (M0) -- (M);
   \draw [thick] (Mc) -- (M1) -- (M);
   \draw [thick] (Mc) to[out=120,in=-120] (Tc);
   \draw [thick] (M0) -- (T0);
   \draw [thick] (M1) -- (T1);
   \draw [thick] (M) to[out=55,in=-55] (Omega);
   \draw [thick] (Tc) -- (T0) -- (Omega);
   \draw [thick] (Tc) -- (T1) -- (Omega);
\end{tikzpicture}
}
\end{center}
\caption{Post's lattice.}
\label{fig:Post}
\end{figure}

Let us conclude this introductory section with a couple of lemmata that help us express sums and minors of Boolean functions in terms of their sets of monomials.

\begin{lemma}
\label{lem:sum-monomials}
Let $f, g \colon \{0,1\}^n \to \{0,1\}$.
Then $\monomials{f+g} = \monomials{f} \symmdiff \monomials{g}$.
\end{lemma}

\begin{proof}
By adding the polynomials of $f$ and $g$ and by cancelling equal monomials (because we do addition modulo $2$), we obtain
\[
f + g
= \sum_{S \in \monomials{f}} x_S + \sum_{S \in \monomials{g}} x_S
= \sum_{S \in \monomials{f} \symmdiff \monomials{g}} x_S.
\]
Consequently, $\monomials{f + g} = \monomials{f} \symmdiff \monomials{g}$ by the uniqueness of Zhegalkin polynomials.
\end{proof}

\begin{lemma}
\label{lem:minor-monomials}
Let $f \colon \{0,1\}^n \to \{0,1\}$ and $\sigma \colon \nset{n} \to \nset{m}$.
Then
\[
\monomials{f_{\sigma}} =
\bigl\{ \, S \subseteq \nset{m} \bigm| \card{ \{ \, T \in \monomials{f} \mid \sigma(T) = S \, \} } \equiv 1 \pmod{2} \, \bigr\}.
\]
\end{lemma}

\begin{proof}
A straightforward calculation using the definitions of minor and $\monomials{f}$ (Definitions~\ref{def:minor} and \ref{def:Zhegalkin}) shows that for all $a_1, \dots, a_m \in \{0,1\}$,
\[
f_\sigma(a_1, \dots, a_m)
= f(a_{\sigma(1)}, \dots, a_{\sigma(n)})
= \sum_{T \in \monomials{f}} \prod_{i \in T} a_{\sigma(i)}
= \sum_{T \in \monomials{f}} \prod_{i \in \sigma(T)} a_i
.
\]
By cancelling pairs of summands corresponding to indices $T, T' \in \monomials{f}$ such that $\sigma(T) = \sigma(T')$, which are equal for any $a_1, \dots, a_m$, we get
\[
\sum_{T \in \monomials{f}} \prod_{i \in \sigma(T)} a_i
= \sum_{S \in M'} \prod_{i \in S} a_i,
\]
where
\[
M' =
\bigl\{ \, S \subseteq \nset{m} \bigm| \card{ \{ \, T \in \monomials{f} \mid \sigma(T) = S \, \} } \equiv 1 \pmod{2} \, \bigr\}.
\]
Consequently, $\monomials{f_\sigma} = M'$ by the uniqueness of Zhegalkin polynomials.
\end{proof}

\begin{table}
\begin{tabular}{lll}
\toprule
Class & Description or defining condition & Definition \\
\midrule
$\clAll$ & all functions & \ref{def:all} \\
$\clCon{a}$ & $f(0, \dots, 0) = a$ & \ref{def:ab-eq-neq} \\
$\clYksi{b}$ & $f(1, \dots, 1) = b$ & \ref{def:ab-eq-neq} \\
$\clBoth{a}{b}$ & $\clCon{a} \cap \clYksi{b}$ & \ref{def:ab-eq-neq} \\
$\clEven$ & $f(0, \dots, 0) = f(1, \dots, 1)$, even functions & \ref{def:ab-eq-neq}, \ref{def:parity} \\
$\clOdd$ & $f(0, \dots, 0) \neq f(1, \dots, 1)$, odd functions & \ref{def:ab-eq-neq}, \ref{def:parity} \\
$\clTa{a}$ & $f(a, \dots, a) = a$ & \ref{def:T0T1Tc} \\
$\clTc$ & $\clTo \cap \clTi = \clBoth{0}{1}$ & \ref{def:T0T1Tc} \\
$\clM$ & monotone functions & \ref{def:monotone} \\
$\clS$ & self\hyp{}dual functions & \ref{def:S-R} \\
$\clSc$ & $\clS \cap \clTc$ & \ref{def:S-R} \\
$\clSM$ & $\clS \cap \clM$ & \ref{def:S-R} \\
$\clD{k}$ & $\deg(f) \leq k$ (degree bounded by $k$) & \ref{def:Dk} \\
$\clL$ & $\clD{1}$, linear functions & \ref{def:Dk} \\
$\clLo$ & $\clL \cap \clTo = \clL \cap \clCon{0}$ & \ref{def:Dk} \\
$\clLi$ & $\clL \cap \clTi = \clL \cap \clYksi{1}$ & \ref{def:Dk} \\
$\clLS$ & $\clL \cap \clS = \clL \cap \clOdd$ & \ref{def:Dk} \\
$\clLc$ & $\clL \cap \clTc = \clL \cap \clBoth{0}{1}$ & \ref{def:Dk} \\
$\clChar{k}$ & $\charrank{f} \leq k$ (characteristic rank bounded by $k$) & \ref{def:Xk} \\
$\clLambdac$ & conjunctions & \ref{def:conj-disj} \\
$\clVc$ & disjunctions & \ref{def:conj-disj} \\
$\clIo$ & projections and constant $0$ functions & \ref{def:conj-disj} \\
$\clIi$ & projections and constant $1$ functions & \ref{def:conj-disj} \\
$\clIstar$ & projections and negated projections & \ref{def:conj-disj} \\
$\clIc$ & projections & \ref{def:conj-disj} \\
\bottomrule
\end{tabular}

\bigskip
\caption{Notation for classes of Boolean functions. Parameters: $a, b \in \{0, 1\}$, $k \in \IN$.}
\label{table:notation}
\end{table}

%%%%%%%%%%%%%%%%%%%%%%%%%%%%%%%%%%%%%%%%%%%%%%%%%%

\section{$\clLc$\hyp{}stable classes}
\label{sec:Lc-stable}
\renewcommand{\gendefault}{\clLc}

We are now ready to state the main result of this paper, a complete description of the $\clLc$\hyp{}stable classes of Boolean functions.

\begin{theorem}
\label{thm:Lc}
The classes of Boolean functions stable under both left and right compositions with the clone $\clLc = \clonegen{\oplus_3}$ \textup{(}$\clLc$\hyp{}stable classes\textup{)} or, equivalently, the $(\clIc, \clLc)$\hyp{}stable classes are
\begin{align*}
& \clAll, && \clCon{a}, && \clYksi{b}, && \clParity{a}, && \clBoth{a}{b}, \\
& \clD{k}, && \clD{k} \cap \clCon{a}, && \clD{k} \cap \clYksi{b}, && \clD{k} \cap \clParity{a}, && \clD{k} \cap \clBoth{a}{b}, \\
& \clChar{k}, && \clChar{k} \cap \clCon{a}, && \clChar{k} \cap \clYksi{b}, && \clChar{k} \cap \clParity{a}, && \clChar{k} \cap \clBoth{a}{b}, \\
& \clD{i} \cap \clChar{j}, && \clD{i} \cap \clChar{j} \cap \clCon{a}, && \clD{i} \cap \clChar{j} \cap \clYksi{b}, && \clD{i} \cap \clChar{j} \cap \clParity{a}, && \clD{i} \cap \clChar{j} \cap \clBoth{a}{b}, \\
& \clD{0}, && \clD{0} \cap \clCon{a}, && \clEmpty,
\end{align*}
for $a, b \in \{0,1\}$, $\QuantifyParRel$, and $i, j, k \in \IN_{+}$ with $i > j \geq 1$.
\end{theorem}

Several $\clLc$\hyp{}stable classes were known previously,
namely,
the clones $\clAll$, $\clS = \clChar{1} \cap \clOdd$, $\clL = \clD{1}$, $\clTo = \clCon{0}$, $\clTi = \clYksi{1}$,
$\clTc = \clBoth{0}{1}$, $\clSc = \clChar{1} \cap \clBoth{0}{1}$, $\clLo = \clD{1} \cap \clCon{0}$, $\clLi = \clD{1} \cap \clYksi{1}$, $\clLS = \clD{1} \cap \clChar{1} \cap \clOdd$, $\clLc = \clD{1} \cap \clBoth{0}{1}$,
as well as the classes $\clD{k}$ for any $k \in \IN$ and the class $\clChar{0}$ of reflexive (self\hyp{}anti\hyp{}dual) functions \cite[pp.\ 29, 33]{CouFol-2004}.
The classes $\clD{k}$ for $k \in \IN$ were also known to be $\clLo$\hyp{}stable \cite[Example 1, p.\ 111]{CouFol-2009}.

In order to describe the structure of the lattice of $\clLc$\hyp{}stable classes, it is helpful to first look at the poset comprising the eleven classes $\clAll$, $\clEven$, $\clOdd$, $\clCon{0}$, $\clCon{1}$, $\clYksi{0}$, $\clYksi{1}$, $\clBoth{0}{0}$, $\clBoth{0}{1}$, $\clBoth{1}{0}$, $\clBoth{1}{1}$ that is shown in Figure~\ref{fig:11-poset}.
It is noteworthy that the four minimal classes of this poset are pairwise disjoint, and that the six lower covers of $\clAll$ are precisely the unions of the six different pairs of minimal classes.

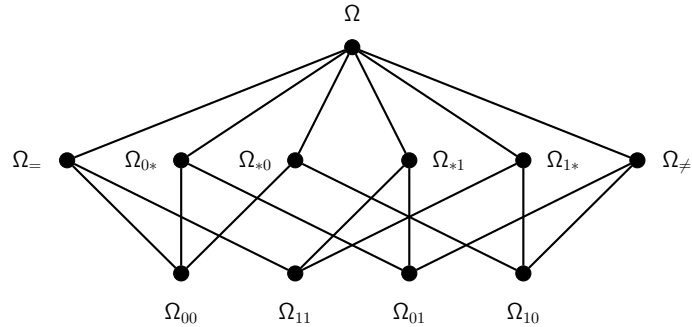
\begin{figure}
\begin{center}
\tikzstyle{every node}=[circle, draw, fill=black, scale=0.6, font=\LARGE]
\begin{tikzpicture}[baseline, scale=0.75]
      \node (S) at (0,2) {};
      \node (P0) at (-5,0) {};
      \node (C0) at (-3,0) {};
      \node (E0) at (-1,0) {};
      \node (E1) at (1,0) {};
      \node (C1) at (3,0) {};
      \node (P1) at (5,0) {};
      \node (P0C0) at (-3,-2) {};
      \node (P0C1) at (-1,-2) {};
      \node (P1C0) at (1,-2) {};
      \node (P1C1) at (3,-2) {};
      \foreach \u/\v in {P0C0/P0, P0C0/C0, P0C0/E0, P0C1/P0, P0C1/C1, P0C1/E1, P1C0/P1, P1C0/C0, P1C0/E1, P1C1/P1, P1C1/C1, P1C1/E0, P0/S, C0/S, E0/S, E1/S, C1/S, P1/S}
      {
         \draw [thick] (\u) -- (\v);
      }
\draw ($(S)+(0,0.6)$) node[draw=none,fill=none]{$\clAll$};
\draw ($(P0)+(-0.7,0)$) node[draw=none,fill=none]{$\clEven$};
\draw ($(C0)+(-0.7,0)$) node[draw=none,fill=none]{$\clCon{0}$};
\draw ($(E0)+(-0.7,0)$) node[draw=none,fill=none]{$\clYksi{0}$};
\draw ($(E1)+(0.7,0)$) node[draw=none,fill=none]{$\clYksi{1}$};
\draw ($(C1)+(0.7,0)$) node[draw=none,fill=none]{$\clCon{1}$};
\draw ($(P1)+(0.7,0)$) node[draw=none,fill=none]{\raisebox{0pt}[0pt][0pt]{$\clOdd$}\makebox[0pt][l]{\phantom{$\clEven$}}};
\draw ($(P0C0)+(0,-0.7)$) node[draw=none,fill=none]{$\clBoth{0}{0}$};
\draw ($(P0C1)+(0,-0.7)$) node[draw=none,fill=none]{$\clBoth{1}{1}$};
\draw ($(P1C0)+(0,-0.7)$) node[draw=none,fill=none]{$\clBoth{0}{1}$};
\draw ($(P1C1)+(0,-0.7)$) node[draw=none,fill=none]{$\clBoth{1}{0}$};
\end{tikzpicture}
\end{center}
\caption{A block of eleven $\clLc$\hyp{}stable classes.}
\label{fig:11-poset}
\end{figure}

The lattice of all $\clLc$\hyp{}stable classes is shown in Figure~\ref{fig:Lc-stable}.
It has rather regular structure;
it is isomorphic to the direct product of the 11\hyp{}element poset of Figure~\ref{fig:11-poset} and the set $\{ \, (i, j) \in (\IN_{+} \cup \{\infty\})^2 \mid i \geq j \geq 1 \, \}$ with the componentwise order, and a few additional elements near the bottom of the lattice.
In order to avoid clutter, we have used some shorthand notation in Figure~\ref{fig:Lc-stable}.
The diagram includes multiple copies of the 11\hyp{}element poset of Figure~\ref{fig:11-poset} (the shaded blocks) connected by thick triple lines.
Each thick triple line between a pair of blocks represents eleven edges, each connecting a vertex of one poset to its corresponding vertex in the other poset.
We have labeled in the diagram the meet\hyp{}irreducible classes, as well as a few other classes of interest; the remaining classes are intersections of the meet\hyp{}irreducible ones.

\begin{figure}
\begin{center}
\newcommand{\M}{3} % number of levels
\newcommand{\hilax}{10} % x distance between grid points
\newcommand{\hilay}{9} % y distance between grid points
\newcommand{\infsep}{1.5} % extra distance for the infinite chains
\newcommand{\infshort}{0.6} % shortening factor for the infinite chains
\scalebox{0.235}{
\tikzstyle{every node}=[circle, draw, fill=black, scale=1, font=\LARGE]
\begin{tikzpicture}[baseline, scale=1]
\foreach \j in {1,...,\M}
{
  \foreach \i in {\j,...,\M}
  {
     \node (keskusD\i{}X\j) at (${\i+\j-2}*(0,\hilay)+{\i-\j}*(\hilax,0)$) {};
  }
  \node (keskusDinfX\j) at (${\j+\M+\infsep-2}*(0,\hilay)+{\M+\infsep-\j}*(\hilax,0)$) {};
}
\node (keskusAll) at (${2*(\M+\infsep)-2}*(0,\hilay)$) {};
\tikzmath{integer \MM; \MM = \M - 1;}
\foreach \j in {1,...,\MM}
{
   \foreach \i in {\j,...,\MM}
   {
      \tikzmath{integer \vv; \vv = \i + 1;}
      \draw [line width=10pt, double distance=30pt] (keskusD\i{}X\j) -- (keskusD\vv{}X\j);
      \draw [line width=10pt] (keskusD\i{}X\j) -- (keskusD\vv{}X\j);
   }
}
\foreach \i in {2,...,\M}
{
   \foreach \j in {2,...,\i}
   {
      \tikzmath{integer \vv; \vv = \j - 1;}
      \draw [line width=10pt, double distance=30pt] (keskusD\i{}X\vv) -- (keskusD\i{}X\j);
      \draw [line width=10pt] (keskusD\i{}X\vv) -- (keskusD\i{}X\j);
   }
}
\foreach \j in {2,...,\M}
{
   \tikzmath{integer \vv; \vv = \j - 1;}
   \draw [line width=10pt, double distance=30pt] (keskusDinfX\vv) -- (keskusDinfX\j);
   \draw [line width=10pt] (keskusDinfX\vv) -- (keskusDinfX\j);
}
\tikzmath{real \s; \s = \infshort*\infsep*sqrt(\hilax*\hilax+\hilay*\hilay);}
\foreach \j in {1,...,\M}
{
   \draw [line width=10pt, double distance=30pt, shorten >=\s cm] (keskusD\M{}X\j) -- (keskusDinfX\j);
   \draw [line width=10pt, shorten >=\s cm] (keskusD\M{}X\j) -- (keskusDinfX\j);
   \draw [line width=10pt, line cap=round, dash pattern=on 0pt off 40pt, dash phase=30pt] (keskusD\M{}X\j) -- (keskusDinfX\j);
}
\draw [line width=10pt, double distance=30pt, shorten >=\s cm] (keskusDinfX\M) -- (keskusAll);
\draw [line width=10pt, shorten >=\s cm] (keskusDinfX\M) -- (keskusAll);
\draw [line width=10pt, line cap=round, dash pattern=on 0pt off 40pt, dash phase=30pt] (keskusDinfX\M) -- (keskusAll);
\tikzmath{real \t; \t = 0.25*2*\infsep*\hilay;}
\draw [line width=10pt, line cap=round, dash pattern=on 0pt off 40pt, shorten >=\t cm, shorten <=\t cm] (keskusD\M{}X\M) -- (keskusAll);
\tikzmath{integer \MM; \MM = \M + 1;}
\foreach \i in {1,...,\MM}
{
   \foreach \j in {1,...,\i}
   {
      \ifthenelse{\j=\MM}{\def\nimi{All} \def\lappu{\clAll}}
                         {\ifthenelse{\i=\MM}{\def\nimi{DinfX\j} \def\lappu{\clChar{\j}}}
                                             {\def\nimi{D\i{}X\j} \ifthenelse{\i=\j}{\def\lappu{\clD{\i}}}{\def\lappu{\clD{\i}\cap\clChar{\j}}}}};
      \node (blob\nimi) at (keskus\nimi) [rectangle,rounded corners,draw=green!60!black,fill=green!20,minimum width=13cm,minimum height=6.5cm] {};
      \node (S\nimi) at ($(keskus\nimi) + (0,2)$) {};
      \draw ($(S\nimi)+(0,0.6)$) node[draw=none,fill=none]{$\lappu$};
      \node (P0\nimi) at ($(keskus\nimi) + (-5,0)$) {};
      \node (C0\nimi) at ($(keskus\nimi) + (-3,0)$) {};
      \node (E0\nimi) at ($(keskus\nimi) + (-1,0)$) {};
      \node (E1\nimi) at ($(keskus\nimi) + (1,0)$) {};
      \node (C1\nimi) at ($(keskus\nimi) + (3,0)$) {};
      \node (P1\nimi) at ($(keskus\nimi) + (5,0)$) {};
      \node (P0C0\nimi) at ($(keskus\nimi) + (-3,-2)$) {};
      \node (P0C1\nimi) at ($(keskus\nimi) + (-1,-2)$) {};
      \node (P1C0\nimi) at ($(keskus\nimi) + (1,-2)$) {};
      \node (P1C1\nimi) at ($(keskus\nimi) + (3,-2)$) {};
      \foreach \u/\v in {P0C0/P0, P0C0/C0, P0C0/E0, P0C1/P0, P0C1/C1, P0C1/E1, P1C0/P1, P1C0/C0, P1C0/E1, P1C1/P1, P1C1/C1, P1C1/E0, P0/S, C0/S, E0/S, E1/S, C1/S, P1/S}
      {
         \draw [thick] (\u\nimi) -- (\v\nimi);
      }
   }
}
\draw ($(P0All)+(-0.7,0)$) node[draw=none,fill=none]{$\clEven$};
\draw ($(C0All)+(-0.7,0)$) node[draw=none,fill=none]{$\clCon{0}$};
\draw ($(E0All)+(-0.7,0)$) node[draw=none,fill=none]{$\clYksi{0}$};
\draw ($(E1All)+(0.7,0)$) node[draw=none,fill=none]{$\clYksi{1}$};
\draw ($(C1All)+(0.7,0)$) node[draw=none,fill=none]{$\clCon{1}$};
\draw ($(P1All)+(0.7,0)$) node[draw=none,fill=none]{\raisebox{0pt}[0pt][0pt]{$\clOdd$}\makebox[0pt][l]{\phantom{$\clEven$}}};
\draw ($(P1C0All)+(0,-0.7)$) node[draw=none,fill=none]{$\clTc$};
\draw ($(P0DinfX1)+(-0.7,0)$) node[draw=none,fill=none]{$\clChar{0}$};
\draw ($(P1DinfX1)+(0.7,0)$) node[draw=none,fill=none]{$\clS$};
\draw ($(P1C0DinfX1)+(0,-0.7)$) node[draw=none,fill=none]{$\clSc$};
\draw ($(C0D1{}X1)+(-0.6,0)$) node[draw=none,fill=none]{$\clLo$};
\draw ($(E1D1{}X1)+(0.6,0)$) node[draw=none,fill=none]{$\clLi$};
\draw ($(P1D1{}X1)+(0.6,0)$) node[draw=none,fill=none]{$\clLS$};
\draw ($(P1C0D1{}X1)+(0.6,-0.6)$) node[draw=none,fill=none]{$\clLc$};
\node (D0) at ($(P0D1{}X1) + (0,-6)$) {};
\draw ($(D0)+(-0.7,0)$) node[draw=none,fill=none]{$\clD{0}$};
\node (C0D0) at ($(D0) + (2,-2)$) {};
\node (C1D0) at ($(C0D0) + (2,0)$) {};
\node (empty) at ($(C1D0) + (2,-2)$) {};
\draw ($(empty)+(0,-0.7)$) node[draw=none,fill=none]{$\clEmpty$};
\foreach \u/\v in {empty/C0D0, empty/C1D0, empty/P1C0D1{}X1, empty/P1C1D1{}X1, C0D0/D0, C0D0/P0C0D1{}X1, C1D0/D0, C1D0/P0C1D1{}X1, D0/P0D1{}X1}
{
   \draw [thick] (\u) -- (\v);
}
\end{tikzpicture}
}
\end{center}
\caption{$\clLc$\hyp{}stable classes.}
\label{fig:Lc-stable}
\end{figure}

The remainder of this section is devoted to the proof of Theorem~\ref{thm:Lc}.
The proof has two parts.
First we observe that the classes listed in Theorem~\ref{thm:Lc} are $\clLc$\hyp{}stable.
Second, we need to show that there are no other $\clLc$\hyp{}stable classes.

To this end, we start by verifying that the classes of Theorem~\ref{thm:Lc} are $\clLc$\hyp{}stable.
Since intersections of $\clLc$\hyp{}stable classes are $\clLc$\hyp{}stable, it suffices to verify this for the meet\hyp{}irreducible classes.
With the help of the following lemma, we can further simplify the task of checking the stability under left and right composition with clones containing the triple sum.
In fact $\clLc$\hyp{}stability is equivalent to $(\clIc, \clLc)$\hyp{}stability.

\begin{lemma}
\label{lem:Lc-simplify}
\leavevmode
\begin{enumerate}[label=\upshape{(\roman*)}, leftmargin=*, widest=iii]
\item\label{lem:Lc-simplify:triplesum}
For any $f \in \clAll^{(n)}$, we have $f \ast \mathord{\oplus_3} = \mathord{\oplus_3}(f_{\sigma_1}, f_{\sigma_2}, f_{\sigma_3})$, where, for $i \in \nset{3}$, $\sigma_i \colon \nset{n} \to \nset{n+2}$, $1 \mapsto i$, $j \mapsto j + 2$ for $2 \leq j \leq n$.

\item\label{lem:Lc-simplify:general}
Let $G \subseteq \clAll$, let $C_1 := \clonegen{G \cup \{ \mathord{\oplus_3} \}}$, $C'_1 := \clonegen{G}$, and let $C_2$ be a clone containing $\mathord{\oplus_3}$.
Then a class $F \subseteq \clAll$ is $(C_1, C_2)$\hyp{}stable if and only if it is $(C'_1, C_2)$\hyp{}stable.

\item\label{lem:Lc-simplify:Lc}
The following are equivalent for a class $F \subseteq \clAll$.
\begin{enumerate}[label=\upshape{(\alph*)}]
\item $F$ is $\clLc$\hyp{}stable.
\item $F$ is $(\clIc, \clLc)$\hyp{}stable.
\item $F$ is minor\hyp{}closed and $f + g + h \in F$ whenever $f, g, h \in F$.
\end{enumerate}
\end{enumerate}
\end{lemma}

\begin{proof}
\ref{lem:Lc-simplify:triplesum}
Let $\vect{a} = (a_1, \dots, a_{n+2}) \in \{0,1\}^{n+2}$.
By the pigeonhole principle, there exist $i$, $j$, $k$ such that $\{i, j, k\} = \{1, 2, 3\}$ and $a_i = a_j$.
Using the fact that $\oplus_3$ is the minority operation on $\{0,1\}$, we obtain
\begin{align*}
& \mathord{\oplus_3}(f_{\sigma_1}, f_{\sigma_2}, f_{\sigma_3})(\vect{a})
\\
& = \mathord{\oplus_3}(f(a_1, a_4, \dots, a_{n+2}), f(a_2, a_4, \dots, a_{n+2}), f(a_3, a_4, \dots, a_{n+2}))
\\
& = f(a_k, a_4, \dots, a_{n+2})
= f(\mathord{\oplus_3}(a_1, a_2, a_3), a_4, \dots, a_{n+2})
= (f \ast \mathord{\oplus_3})(\vect{a}),
\end{align*}
and we conclude that $f \ast \mathord{\oplus_3} = \mathord{\oplus_3}(f_{\sigma_1}, f_{\sigma_2}, f_{\sigma_3})$.

\ref{lem:Lc-simplify:general}
Since $C'_1 \subseteq C_1$, stability under right composition with $C_1$ implies stability under right composition with $C'_1$.
Assume now that $F$ is $(C'_1, C_2)$\hyp{}stable.
By Lemma~\ref{lem:right-stab-gen}, $F$ is minor\hyp{}closed and $f \ast g \in F$ whenever $f \in F$ and $g \in G$.
Moreover, $f \ast \mathord{\oplus_3} = \mathord{\oplus_3}(f_{\sigma_1},f_{\sigma_2},f_{\sigma_3})$, where $f_{\sigma_1}$, $f_{\sigma_2}$, $f_{\sigma_3}$ are the minors of $f$ specified in part~\ref{lem:Lc-simplify:triplesum}.
Since $F$ is minor\hyp{}closed, we have $f_{\sigma_1}, f_{\sigma_2}, f_{\sigma_3} \in F$.
By our assumption, $\mathord{\oplus_3} \in C_2$, and since $F$ is stable under left composition with $C_2$, it follows that $\mathord{\oplus_3}(f_{\sigma_1},f_{\sigma_2},f_{\sigma_3}) \in F$.
It follows from Lemma~\ref{lem:right-stab-gen} that $F$ is stable under right composition with $C_1$.

\ref{lem:Lc-simplify:Lc}
Since $\clLc = \clonegen{\mathord{\oplus_3}}$, this is a consequence of part~\ref{lem:Lc-simplify:general} and Lemma~\ref{lem:left-stab-gen}.
\end{proof}

In view of Lemma~\ref{lem:Lc-simplify}\ref{lem:Lc-simplify:Lc}, our task is reduced to verifying that each one of the meet\hyp{}irreducible classes shown in Figure~\ref{fig:Lc-stable}, namely $\clAll$, $\clCon{0}$, $\clCon{1}$, $\clYksi{0}$, $\clYksi{1}$, $\clEven$, $\clOdd$, $\clD{k}$, and $\clChar{k}$ for $k \in \IN$, is minor\hyp{}closed and closed under triple sums of its members.

\begin{lemma}
\label{lem:All-closed}
$\clAll$ is minor\hyp{}closed and closed under triple sums of its members.
\end{lemma}

\begin{proof}
This is obvious.
\end{proof}

\begin{lemma}
\label{lem:CaEa-closed}
Let $a, b \in \{0,1\}$.
\begin{enumerate}[label=\upshape{(\roman*)}, leftmargin=*, widest=ii]
\item\label{lem:CaEa-closed:Ca} $\clCon{a}$ is minor\hyp{}closed and closed under triple sums of its members.
\item\label{lem:CaEa-closed:Ea} $\clYksi{b}$ is minor\hyp{}closed and closed under triple sums of its members.
\end{enumerate}
\end{lemma}

\begin{proof}
\ref{lem:CaEa-closed:Ca}
Let $f \in \clCon{a}^{(n)}$, and let $\sigma \colon \nset{n} \to \nset{m}$.
We have $f_\sigma(0, \dots, 0) = f(0, \dots, 0) = a$, so $f_\sigma \in \clCon{a}$; thus $\clCon{a}$ is minor\hyp{}closed.
Let now $f, g, h \in \clCon{a}^{(n)}$.
We have $(f + g + h)(0, \dots, 0) = f(0, \dots, 0) + g(0, \dots, 0) + h(0, \dots, 0) = a + a + a = a$; thus $f + g + h \in \clCon{a}$.

\ref{lem:CaEa-closed:Ea}
The proof is similar to that of part \ref{lem:CaEa-closed:Ca}.
\end{proof}

\begin{lemma}
\label{lem:Pa-closed}
For $\QuantifyParRel$, $\clParity{a}$ is minor\hyp{}closed and closed under triple sums of its members.
\end{lemma}

\begin{proof}
We show first that $\clParity{a}$ is minor\hyp{}closed.
Let $f \in \clParity{a}^{(n)}$ and $\sigma \colon \nset{n} \to \nset{m}$.
We have $f_\sigma(0, \dots, 0) = f(0, \dots, 0) \ParRel f(1, \dots, 1) = f_\sigma(1, \dots, 1)$, so $f_\sigma \in \clParity{a}$.

We now show that $\clParity{a}$ is closed under triple sums of its members.
For $f, g, h \in \clEven^{(n)}$, we have
$(f + g + h)(\vect{0})
= f(\vect{0}) + g(\vect{0}) + h(\vect{0})
= f(\vect{1}) + g(\vect{1}) + h(\vect{1})
= (f + g + h)(\vect{1})$,
where $\vect{0} := (0, \dots, 0)$ and $\vect{1} := (1, \dots, 1)$;
therefore $f + g + h \in \clEven$.
For $f, g, h \in \clOdd^{(n)}$, we have
$(f + g + h)(\vect{0})
= f(\vect{0}) + g(\vect{0}) + h(\vect{0})
= f(\vect{1}) + 1 + g(\vect{1}) + 1+ h(\vect{1}) + 1
= (f + g + h)(\vect{1}) + 1$;
therefore $f + g + h \in \clOdd$.
\end{proof}

\begin{lemma}
\label{lem:Dk-closed}
For $k \in \IN$, $\clD{k}$ is minor\hyp{}closed and closed under sums of its members.
\end{lemma}

\begin{proof}
It is clear that the degree of functions cannot increase by formation of minors nor by composition with linear functions.
\end{proof}

\begin{lemma}
\label{lem:Xk-closed}
For $k \in \IN$, $\clChar{k}$ is minor\hyp{}closed and closed under sums of its members.
\end{lemma}

\begin{proof}
Let $f, g \in \clChar{k}^{(n)}$, and let $\varphi := f + f^\mathrm{n}$ and $\gamma := g + g^\mathrm{n}$.
By Lemma~\ref{lem:Xk-description}, $\deg(\varphi) \leq k-1$ and $\deg(\gamma) \leq k-1$.
Let $\sigma : \nset{m} \to \nset{n}$.
By Lemma~\ref{lem:comp-minors}, we have
$f_\sigma + (f_\sigma)^\mathrm{n} = f_\sigma + (f^\mathrm{n})_\sigma = (f + f^\mathrm{n})_\sigma = \varphi_\sigma$.
Since the formation of minors does not increase the degree, we have $\deg(\varphi_\sigma) \leq k - 1$, so $f_\sigma \in \clChar{k}$.
Furthermore, $(f + g) + (f + g)^\mathrm{n} = (f + f^\mathrm{n}) + (g + g^\mathrm{n}) = \varphi + \gamma$, which has degree at most $k - 1$, so $f + g \in \clChar{k}$.
\end{proof}

\begin{proposition}
\label{prop:Lc-sufficiency}
The classes listed in Theorem~\ref{thm:Lc} are $\clLc$\hyp{}stable.
\end{proposition}

\begin{proof}
According to Lemmata~\ref{lem:All-closed}, \ref{lem:CaEa-closed}, \ref{lem:Pa-closed}, \ref{lem:Dk-closed}, and \ref{lem:Xk-closed}, each of the classes $\clAll$, $\clCon{0}$, $\clCon{1}$, $\clYksi{0}$, $\clYksi{1}$, $\clEven$, $\clOdd$, $\clD{k}$, and $\clChar{k}$ for $k \in \IN$ is minor\hyp{}closed and closed under triple sums of its members,
so by Lemma~\ref{lem:Lc-simplify}\ref{lem:Lc-simplify:Lc}, each is $\clLc$\hyp{}stable.
It follows that the remaining classes listed in Theorem~\ref{thm:Lc}, being intersections of the above classes, are also $\clLc$\hyp{}stable.
\end{proof}

It remains to show that the classes listed in Theorem~\ref{thm:Lc} are the only $\clLc$\hyp{}stable classes.
To this end, we are going to verify that any set of Boolean functions generates exactly what is suggested by Figure~\ref{fig:Lc-stable}.
More precisely, we prove that each class $K$ is generated by any subset of $K$ that is not contained in any proper subclass fo $K$, i.e., the subset contains for each proper subclass $C$ of $K$ an element in $K \setminus C$.
If each proper subclass is contained in a lower cover of $K$, then it suffices to consider the lower covers of $K$.
We begin with some helpful lemmata.

\begin{lemma}
\label{lem:Lc-sum-of-minors}
For any $F \subseteq \clAll$, we have $f \in \gen{F}$ if and only if $f$ is the sum of an odd number of minors of members of $F$, i.e., $f = \sum_{i=1}^{2k+1} (g_i)_{\sigma_i}$ for some $k \in \IN$, $g_i \in F$, $\sigma_i \colon \nset{n_i} \to \nset{n}$, where $n_i := \arity{g_i}$ and $n := \arity{f}$ \textup{(}$1 \leq i \leq 2k + 1$\textup{)}.
\end{lemma}

\begin{proof}
``$\Leftarrow$'':
Clear because $\gen{F}$ is closed under minors and triple sums and hence under any odd sums of its members by Lemma~\ref{lem:Lc-simplify}\ref{lem:Lc-simplify:Lc}.

``$\Rightarrow$'':
By Lemma~\ref{lem:Lc-simplify}\ref{lem:Lc-simplify:Lc}, $\gen{F}$ is the set obtained by a finite number of the following construction steps:
\begin{enumerate}[label=(\arabic*), leftmargin=*, widest=3]
\item Every $f \in F$ is a member of $\gen{F}$.
\item If $f \in \gen{F}$, $\arity{f} = n$, and $\sigma \colon \nset{n} \to \nset{m}$ for some $m \in \IN_{+}$, then $f_\sigma \in \gen{F}$.
\item If $f, g, h \in \gen{F}$, all of arity $n \in \IN_{+}$, then $f + g + h \in \gen{F}$.
\end{enumerate}
We will show by induction on the construction that every $f \in \gen{F}$ is an odd sum of minors of members of $F$.
This obviously holds for every $f \in F$: $f = \sum_{i=1}^1 f_{\id}$.
Assume $f = \sum_{i=1}^{2k+1} (g_i)_{\sigma_i}$ for some $g_i \in F$ and $\sigma_i \colon \nset{n_i} \to \nset{n}$ ($1 \leq i \leq 2k+1$).
Then for any $\tau \colon \nset{n} \to \nset{m}$, we have
\[
f_\tau
= \bigl( \sum_{i=1}^{2k+1} (g_i)_{\sigma_i} \bigr)_\tau
= \sum_{i=1}^{2k+1} ((g_i)_{\sigma_i})_\tau
= \sum_{i=1}^{2k+1} (g_i)_{\tau \circ \sigma_i},
\]
where the second and the third equalities hold by Lemmata~\ref{lem:comp-minors} and \ref{lem:minor-composition}, respectively.
Finally, assume that $f = \sum_{i=1}^{2k+1} (f_i)_{\sigma_i}$, $g = \sum_{i=1}^{2 \ell + 1} (g_i)_{\tau_i}$, $h = \sum_{i=1}^{2m+1} (h_i)_{\rho_i}$ for some $f_i, g_i, h_i \in F$, $\sigma_i \colon \nset{\arity{f_i}} \to \nset{n}$, $\tau_i \colon \nset{\arity{g_i}} \to \nset{n}$, $\rho_i \colon \nset{\arity{h_i}} \to \nset{n}$.
Then
\[
f + g + h
= \sum_{i=1}^{2k+1} (f_i)_{\sigma_i} + \sum_{i=1}^{2 \ell + 1} (g_i)_{\tau_i} + \sum_{i=1}^{2m+1} (h_i)_{\rho_i},
\]
which is an odd sum of minors of members of $F$.
\end{proof}

\begin{lemma}
\label{lem:gen-C+1}
Assume that $C$ is an $\clLc$\hyp{}stable class and $\gen{F} = C$.
Then $\overline{C}$ is $\clLc$\hyp{}stable and $\gen{\overline{F}} = \overline{C}$.
\end{lemma}

\begin{proof}
Assume that $\gen{F} = C$.
Then $\overline{C}$ is $\clLc$\hyp{}stable because for all $n$\hyp{}ary $f + 1, g + 1, h + 1 \in \overline{C}$, we have $f, g, h \in C$ and hence
$(f + 1) + (g + 1) + (h + 1) = (f + g + h) + 1 \in \overline{C}$, and
for any $\sigma \colon \nset{n} \to \nset{m}$, we have, by Lemma~\ref{lem:comp-minors}, $(f + 1)_\sigma = f_\sigma + 1_\sigma = f_\sigma + 1 \in \overline{C}$.

In order to show that $\overline{C}$ is generated by $\overline{F}$, let $f + 1 \in \overline{C}$.
Then $f \in C$, and by Lemma~\ref{lem:Lc-sum-of-minors}, $f = \sum_{i=1}^{2k+1} (g_i)_{\sigma_i}$ for some $g_i \in F$ and some minor formation map $\sigma_i$ ($1 \leq i \leq 2k+1$).
Consequently,
$f + 1
= \sum_{i=1}^{2k+1} ((g_i)_{\sigma_i} + 1)
= \sum_{i=1}^{2k+1} ((g_i)_{\sigma_i} + 1_{\sigma_i})
= \sum_{i=1}^{2k+1} (g_i + 1)_{\sigma_i}$
by Lemma~\ref{lem:comp-minors}.
Since each $g_i + 1$ is in $\overline{F}$, Lemma~\ref{lem:Lc-sum-of-minors} implies that $f \in \gen{\overline{F}}$.
\end{proof}

\begin{proposition}
\label{prop:gen:D0}
\leavevmode
\begin{enumerate}[label=\upshape{(\roman*)}, leftmargin=*, widest=iii]
\item\label{gen:empty}
$\gen{\clEmpty} = \clEmpty$.
\item\label{gen:D0C0}
For any $f \in \clD{0} \cap \clCon{0}$, we have $\gen{f} = \clD{0} \cap \clCon{0}$.
\item\label{gen:D0C1}
For any $f \in \clD{0} \cap \clCon{1}$, we have $\gen{f} = \clD{0} \cap \clCon{1}$.
\item\label{gen:D0}
For any $f, g \in \clD{0}$ such that $f \notin \clCon{0}$, $g \notin \clCon{1}$, we have $\gen{f,g} = \clD{0}$.
\end{enumerate}
\end{proposition}

\begin{proof}
\ref{gen:empty}
This is obvious.

\ref{gen:D0C0}
The function $f$ is a constant $0$ function of some arity.
We obtain any constant $0$ function by identifying arguments or introducing fictitious arguments.
Therefore $\clD{0} \cap \clCon{0} \subseteq \gen{f} \subseteq \clD{0} \cap \clCon{0}$.

\ref{gen:D0C1}
Follows from part \ref{gen:D0C0} by Lemma~\ref{lem:gen-C+1} because $\overline{\clD{0} \cap \clCon{0}} = \clD{0} \cap \clCon{1}$.

\ref{gen:D0}
Since $\clCon{0}$ and $\clCon{1}$ partition $\clAll$, it follows that $f \in \clD{0} \cap \clCon{1}$ and $g \in \clD{0} \cap \clCon{0}$.
By parts \ref{gen:D0C0} and \ref{gen:D0C1},
$\clD{0} = (\clD{0} \cap \clCon{0}) \cup (\clD{0} \cap \clCon{1}) = \gen{g} \cup \gen{f} \subseteq \gen{f,g} \subseteq \clD{0}$.
\end{proof}

\begin{lemma}
\label{lem:arity-degree}
Let $f \in \clAll$ with $n := \arity{f}$.
Let $k \in \IN$.
\begin{enumerate}[label=\upshape{(\roman*)}, leftmargin=*, widest=iii]
\item\label{lem:arity-degree:degree}
If $n > \deg(f)$, then $f$ has a minor of degree $\deg(f)$ and arity $\deg(f) + 1$.
\item\label{lem:arity-degree:atleast}
If $f \in \clChar{k}$ and $\deg(f) > k$, then $n > \deg(f)$.
\item\label{lem:arity-degree:smallest}
If $f \in \clChar{k}$ and $n - 1 = \deg(f) > k$, then $\monomials{f}$ contains all subsets of $\nset{n}$ of cardinality $n - 1$.
\item\label{lem:arity-degree:minor1}
If $f \in \clChar{k} \setminus \clChar{k-1}$, then $f$ has a $k$\hyp{}ary minor $g$ such that $\nset{k} \in \monomials{g}$ and $g \in \clD{k} \setminus \clChar{k-1}$.
\item\label{lem:arity-degree:minor2}
If $f \in \clChar{k}$ and there is an $S \in \monomials{f}$ with $\ell := \card{S} > k$, then $f$ has an $(\ell + 1)$\hyp{}ary minor $g$ such that $\monomials{g}$ contains all subsets of $\nset{\ell + 1}$ of cardinality $\ell$ but $\nset{\ell + 1} \notin \monomials{g}$.
Moreover, if $\ell > k + 1$, then $\monomials{g}$ contains also a subset of cardinality $\ell - 1$.
\item\label{lem:arity-degree:minor3}
If $\deg{f} = n$, then $f$ has a minor of arity $n-1$ and degree $n-1$.
\end{enumerate}
\end{lemma}

\begin{proof}
\ref{lem:arity-degree:degree}
Let $m := \deg(f)$.
There exists an $S \in \monomials{f}$ with $\card{S} = m$.
Let us identify all arguments not in $S$, i.e., we form the minor $f_\sigma$ with a minor formation map $\sigma \colon \nset{n} \to \nset{m+1}$ that maps $S$ onto $\nset{m}$ and every element of $\nset{n} \setminus S$ to $m+1$.
Then $f_\sigma$ has arity $m+1$.
Clearly every monomial of $f_\sigma$ has degree at most $m$, and $\nset{m} \in \monomials{f_\sigma}$; hence $\deg(f_\sigma) = m$.

\ref{lem:arity-degree:atleast}
Clearly $n = \arity{f} \geq \deg(f)$.
Assume that $n > k$, and suppose, to the contrary, that $n = \deg(f)$.
But then $\Char(\nset{n - 1},f) = 1$ and $\card{\nset{n - 1}} \geq k$, contradicting $f \in \clChar{k}$.

\ref{lem:arity-degree:smallest}
Assume that $n - 1 = \deg(f) > k$.
Then there exists an $S \in \monomials{f}$ with $\card{S} = n - 1$.
Let $A \subseteq S$ with $\card{A} = n - 2$.
Since $n - 2 \geq k$ and $f \in \clChar{k}$, there must be an even number of proper supersets of $A$ in $\monomials{f}$.
We already have $S \in \monomials{f}$, so there must be another one.
In fact there is only one other possibility, namely $A \cup \{i\}$, where $i$ is the unique element of $\nset{n} \setminus S$.
By letting $A$ range over all $(n-2)$\hyp{}element subsets of $S$, we conclude that $\monomials{f}$ indeed contains all subsets of $\nset{n}$ of cardinality $n-1$.

\ref{lem:arity-degree:minor1}
Since $f \notin \clChar{k-1}$, there exists a subset $A \subseteq \nset{n}$ with $\card{A} = k-1$ such that $\Char(A,f) = 1$.
Let us identify all arguments not in $A$, i.e., we form the minor $f_\sigma$ with a minor formation map $\sigma \colon \nset{n} \to \nset{k}$ that maps $A$ onto $\nset{k-1}$ and every element of $\nset{n} \setminus A$ to $k$.
Then $f_\sigma$ has arity $k$.
Since those subsets of $\nset{n}$ whose image under $\sigma$ equals $\nset{k}$ are precisely all proper supersets of $A$, and since $\Char(A,f) = 1$, there are an odd number of sets $T \in \monomials{f}$ such that $\sigma(T) = \nset{k}$.
By Lemma~\ref{lem:minor-monomials}, $\nset{k} \in \monomials{f_\sigma}$.
Then clearly $f_\sigma \in \clD{k} \setminus \clChar{k-1}$.

\ref{lem:arity-degree:minor2}
By part \ref{lem:arity-degree:atleast}, we must have $n > \deg(f) \geq \ell$.
By identifying all arguments that are not in $S$, we obtain a minor $g$ of $f$ that has arity $\ell + 1$ and contains a monomial of degree $\ell$.
Since $\clChar{k}$ is minor\hyp{}closed, $g \in \clChar{k}$, so by part \ref{lem:arity-degree:smallest}, $\nset{\ell + 1} \notin \monomials{g}$; hence $\deg(g) = \ell$.
By part \ref{lem:arity-degree:smallest}, $\monomials{g}$ contains all subsets of $\nset{\ell + 1}$ of cardinality $\ell$.
If $\ell > k + 1$, then $\monomials{g}$ must also contain a subset of cardinality $\ell - 1$.
For, consider a subset $A \subseteq \nset{\ell + 1}$ with $\card{A} = \ell - 2$.
Since $\ell - 2 \geq k$ and $g \in \clChar{k}$, we have $\Char(A,g) = 0$, so there must be an even number of sets $S \in \monomials{g}$ with $A \subsetneq S$.
There are exactly three such sets $S$ of cardinality $\ell$, namely $\nset{\ell + 1} \setminus \{i\}$ for each $i \in \nset{\ell + 1} \setminus A$; therefore there must also be a set of cardinality $\ell - 1$ in $\monomials{g}$.

\ref{lem:arity-degree:minor3}
If $f$ has no monomial of degree $n - 1$, then for any $i, j \in \nset{n}$ with $i < j$, the $(n - 1)$\hyp{}ary minor $f_{ij}$ has degree $n - 1$.
If $f$ has exactly one monomial of degree $n - 1$, say $S \in \monomials{f}$, $\card{S} = n - 1$, then for any $i, j \in S$ with $i < j$, the minor $f_{ij}$ has degree $n - 1$.
If $f$ has at least two monomials of degree $n - 1$, say $S, T \in \monomials{f}$, $S \neq T$, $\card{S} = \card{T} = n - 1$, then for $\{i, j\} := S \symmdiff T$ with $i < j$, the minor $f_{ij}$ has degree $n - 1$.
\end{proof}

In what follows, we are going to make use of a family of special Boolean functions $\monster{k}$ that was inspired
by the ``unitrades'' and the proof methods presented by Potapov~\cite[Section~4]{Potapov}.
There is a minor difference in the definition, though.
While Potapov's unitrade $\monster{k}$ is composed of all subsets of cardinality $k$, we nevertheless include all nonempty proper subsets of $\nset{k+1}$ in the set of monomials of $\monster{k}$, as this will serve better our needs.

\begin{definition}
\label{def:monster}
For $k \in \IN$, let $\monster{k} \colon \{0,1\}^{k+1} \to \{0,1\}$ be the function satisfying $\monomials{\monster{k}} = \{ \, S \subseteq \nset{k+1} \mid 0 < \card{S} < k+1 \, \}$.
Equivalently, $\monster{k}$ satisfies $\monster{k}(\vect{a}) = 1$ if and only if $\vect{a} \notin \{(0, \dots, 0), (1, \dots, 1)\}$.
For $n \geq k$ and $B \subseteq \nset{n}$ with $\card{B} = k$, denote by $\monster{k}^B$ the minor $(\monster{k})_\sigma$ where $\sigma \colon \nset{k} \to \nset{n}$ is an injective map with range $B$ (since $\monster{k}$ is totally symmetric, any such map $\sigma$ produces the same minor).
In other words, $\monster{k}^B$ is obtained from $\monster{k}$ by introducing $n - k$ fictitious arguments and then permuting arguments so that the essential arguments are the ones indexed by the elements of $B$.
While the arity of $\monster{k}^B$ is not explicit in the notation, it will be clear from the context.
\end{definition}

\begin{lemma}
\label{lem:monster}
\leavevmode
\begin{enumerate}[label=\upshape{(\roman*)}, leftmargin=*, widest=ii]
\item\label{lem:monster:parameters}
For any $k \in \IN$, we have
$\monster{k} \in \clD{k} \cap \clChar{0} \cap \clBoth{0}{0}$.

\item\label{lem:monster:minors}
For any $k, \ell \in \IN$ with $k \leq \ell$, $\monster{k}$ is a minor of $\monster{\ell}$.
\end{enumerate}
\end{lemma}

\begin{proof}
\ref{lem:monster:parameters}
It is clear from the definition that $\deg(\monster{k}) = k$, $\monster{k}$ is reflexive, and $\monster{k} \in \clBoth{0}{0}$.
Therefore, $\monster{k} \in \clD{k} \cap \clChar{0} \cap \clBoth{0}{0}$.

\ref{lem:monster:minors}
By the transitivity of the minor relation, it suffices to show that $\monster{k}$ is a minor of $\monster{k+1}$ for any $k \in \IN$.
By identifying the $(k+1)$\hyp{}st and $(k+2)$\hyp{}nd arguments, i.e., by taking $\sigma$ to be the identification map $\sigma_{k+1,k+2}$, we obtain, by Lemma~\ref{lem:minor-monomials},
\[
\monomials{(\monster{k+1})_\sigma}
= \bigl\{ \, S \subseteq \nset{k+1} \bigm| \card{ \{ \, T \in \monomials{\monster{k+1}} \mid \sigma(T) = S \, \}} \equiv 1 \pmod{2} \, \bigr\} =: M.
\]
We now determine which subsets of $\nset{k+1}$ belong to the set $M$ on the right side of the above equality.
Recall that $\monomials{\monster{k+1}} = \{ \, T \subseteq \nset{k+2} \mid 0 < \card{T} < k + 2 \, \}$.
For any $S \subseteq \nset{k}$, the only subset $S'$ of $\nset{k+2}$ such that $\sigma(S) = \sigma(S')$ is $S$ itself; hence $S \in M$ for all $\emptyset \neq S \subseteq \nset{k}$.
For any set of the form $S \cup \{k+1\}$ with $S \subseteq \nset{k}$, there are exactly three subsets $S'$ of $\nset{n+2}$ such that $\sigma(S') = S \cup \{k+1\}$, namely the sets $S \cup \{k+1\}$, $S \cup \{k+2\}$, and $S \cup \{k+1, k+2\}$.
If $S \neq \nset{k}$, then all three sets belong to $\monomials{\monster{k+1}}$.
If $S = \nset{k}$, then only the first two belong to $\monomials{\monster{k+1}}$.
Hence $S \cup \{k\} \in M$ for all $S \subsetneq \nset{k}$.
We conclude that $M = \{ \, S \subseteq \nset{k+1} \mid 0 < \card{S} < k+1 \, \} = \monomials{\monster{k}}$, that is $(\monster{k+1})_\sigma = \monster{k}$.
\end{proof}

Here is another functional construction that we will use in what follows.

\begin{definition}
For any function $f \colon \{0,1\}^n \to \{0,1\}$ and any $i \in \nset{n}$, let $f'_i \colon \{0,1\}^n \to \{0,1\}$ be the function with
$\monomials{f'_i} := \{ \, S \setminus \{i\} \mid S \in \monomials{f}, \, i \in S \, \}$.
\end{definition}

The effect of negating an argument in a function $f$ can be expressed in a convenient way with the help of $f'_i$.

\begin{lemma}
\label{lem:f'i-negation}
Let $f \colon \{0,1\}^n \to \{0,1\}$, $i \in \nset{n}$, and $g := f(x_1, \dots, x_{i-1}, x_i + 1, x_{i+1}, \dots, x_n)$.
Then $g = f + f'_i$.
\end{lemma}

\begin{proof}
Given $f = \sum_{S \in \monomials{f}} x_S$,
we have
\begin{align*}
g
&
= f(x_1, \dots, x_{i-1}, x_i + 1, x_{i+1}, \dots, x_n)	
= \sum_{\substack{S \in \monomials{f} \\ i \notin S}} x_S + \sum_{\substack{S \in \monomials{f} \\ i \in S}} (x_i + 1) x_{S \setminus \{i\}}
\\ &
= \sum_{\substack{S \in \monomials{f} \\ i \notin S}} x_S + \sum_{\substack{S \in \monomials{f} \\ i \in S}} (x_S + x_{S \setminus \{i\}})
= \sum_{S \in \monomials{f}} x_S + \sum_{\substack{S \in \monomials{f} \\ i \in S}} x_{S \setminus \{i\}}
= f + f'_i.
\qedhere
\end{align*}
\end{proof}

\begin{lemma}
\label{lem:f'i}
Let $f \colon \{0,1\}^n \to \{0,1\}$.
\begin{enumerate}[label=\upshape{(\roman*)}, leftmargin=*, widest=ii]
\item\label{lem:f'i:deg}
If $f \neq 0$, then $\deg(f'_i) < \deg(f)$. \textup{(}Here $\deg(0) = -1$.\textup{)}
\item\label{lem:f'i:char}
If $f \in \clChar{k}$ for some $k > 0$, then $f'_i \in \clChar{k-1}$.
\end{enumerate}
\end{lemma}

\begin{proof}
\ref{lem:f'i:deg}
This is obvious from the construction of $f'_i$.

\ref{lem:f'i:char}
Assume $f \in \clChar{k}$ with $k > 0$.
Then for $\varphi := f + f^\mathrm{n}$ we have $\deg(\varphi) \leq k - 1$.
Our goal is to show that $\deg(\theta) \leq k-2$ for $\theta := f'_i + (f'_i)^\mathrm{n}$.
Writing $\vect{z} := (x_1, \dots, x_{i-1}, x_i + 1, x_{i+1}, \dots, x_n)$, we have
\begin{align*}
\theta
&= f'_i + (f'_i)^\mathrm{n}
= (f + f(\vect{z})) + (f + f(\vect{z}))^\mathrm{n}
= f + f^\mathrm{n} + f(\vect{z}) + (f(\vect{z}))^\mathrm{n}
\\ &
= (f + f^\mathrm{n}) + (f + f^\mathrm{n})(\vect{z})
= \varphi + \varphi(\vect{z})
= \varphi'_i
,
\end{align*}
where the second and the last equalities hold by Lemma~\ref{lem:f'i-negation}
and the third equality holds by Lemma~\ref{lem:neg-pol}\ref{lem:neg-pol:inner}.
If $\varphi \neq 0$, then $\deg(\theta) = \deg(\varphi'_i) < \deg(\varphi) < k - 1$ by part \ref{lem:f'i:deg}, so $f'_i \in \clChar{k-1}$.
If $\varphi = 0$, then $\theta = 0$, so $f'_i \in \clChar{0} \subseteq \clChar{k-1}$.
\end{proof}

\begin{lemma}
\label{lem:Char0-sum-Wk}
For any $k \in \IN$, every function in $\clD{k} \cap \clChar{1} \cap \clBoth{0}{0}$ is a sum of minors of $\monster{k}$.
Consequently, $\gen{\monster{k}} = \clD{k} \cap \clChar{1} \cap \clBoth{0}{0}$.
\end{lemma}

\begin{proof}
We follow the proof technique of Potapov \cite[Proposition 11]{Potapov}.
Note that $\clBoth{0}{0} \subseteq \clEven$, so every function in $\clD{k} \cap \clChar{1} \cap \clBoth{0}{0}$ is even.
We proceed by induction on $k$.
The claim is obvious for $k = 0$, since $\clD{0} \cap \clChar{1} \cap \clBoth{0}{0} = \clD{0} \cap \clCon{0}$, and every constant $0$ function (of any arity) can be obtained from $\monster{0}$, the unary constant $0$ function, by introducing fictitious arguments.
The claim is also clear for $k = 1$, since $\clD{1} \cap \clChar{1} \cap \clBoth{0}{0} = \clD{1} \cap \clEven \cap \clCon{0}$, and any even function of degree $1$ with constant term $0$ can be obtained by adding together suitable minors of $\monster{1} = x_1 + x_2$ obtained by introducing fictitious arguments and permuting arguments.

Assume now that the claim holds for $k = \ell$ for some $\ell \geq 1$.
Every function of degree less than $\ell + 1$ in the class $\clD{\ell + 1} \cap \clChar{1} \cap \clBoth{0}{0}$ is a sum of minors of $\monster{\ell}$ by the induction hypothesis and is therefore a sum of minors of $\monster{\ell + 1}$ because $\monster{\ell} \leq \monster{\ell + 1}$ by Lemma~\ref{lem:monster}\ref{lem:monster:minors}.
We only need to consider functions of degree exactly $\ell + 1$.
We proceed by induction on the arity of functions.
By Lemma~\ref{lem:arity-degree}\ref{lem:arity-degree:atleast}, for any $f \in \clChar{0}$ with $\deg(f) = \ell + 1$, we must have $\arity{f} \geq \ell + 2$.
Therefore, in order to establish the basis of induction, we need to consider an arbitrary function $f \in \clD{\ell + 1} \cap \clChar{1} \cap \clBoth{0}{0}$ with $\arity{f} = \ell + 2$.
By Lemma~\ref{lem:arity-degree}\ref{lem:arity-degree:smallest}, $\monomials{f}$ contains all subsets of $\nset{\ell + 2}$ of cardinality $\ell + 1$.
Then $g := f + \monster{\ell + 1} = f + \monster{\ell + 1} + 0 \in \clD{\ell} \cap \clChar{1} \cap \clBoth{0}{0}$ because $f$, $\monster{\ell + 1}$, and $0$ belong to $\clChar{1} \cap \clBoth{0}{0}$, which is $\clLc$\hyp{}stable by Proposition~\ref{prop:Lc-sufficiency}, and $\deg(g) \leq \ell$ because all monomials of degree $\ell + 1$ are cancelled in the sum $f + \monster{\ell + 1}$.
By the inductive hypothesis, $g$ is a sum of minors of $\monster{\ell}$; hence $f = g + \monster{\ell + 1}$ is a sum of minors of $\monster{\ell + 1}$.

For the inductive step, assume that every $m$\hyp{}ary function in $\clD{\ell + 1} \cap \clChar{1} \cap \clBoth{0}{0}$ of degree $\ell + 1$ is a sum of minors of $\monster{\ell + 1}$.
Let $f \in \clD{\ell + 1} \cap \clChar{1} \cap \clBoth{0}{0}$ be $(m+1)$\hyp{}ary and of degree $\ell + 1$.
If $f$ does not depend on the $(m+1)$\hyp{}st argument, then $f$ is obtained from an $m$\hyp{}ary function $f^* \in \clD{\ell + 1} \cap \clChar{1} \cap \clBoth{0}{0}$ by introducing a fictitious argument; then $f^*$ is a sum of minors of $\monster{\ell + 1}$, and by introducing a fictitious argument to the summands we obtain $f$ as a sum of minors of $\monster{\ell + 1}$.
From now on, assume that $f$ depends on the $(m+1)$\hyp{}st argument.
Let $g := f'_{m+1}$, and let $c$ be the constant term ($0$ or $1$) of $g$.
By Lemma~\ref{lem:f'i} we have $g \in \clD{\ell} \cap \clChar{0}$; furthermore, $g + c \in \clD{\ell} \cap \clChar{0} \cap \clCon{0} = \clD{\ell} \cap \clChar{1} \cap \clBoth{0}{0}$.
By the inductive hypothesis, $g + c$ is a sum of minors of $\monster{\ell}$, say $g + c = \sum_{i=1}^p \monster{k_i}^{S_i}$, with $k_i \leq \ell$ for each $i$.
Now let $h := \sum_{i=1}^p \monster{k_i + 1}^{S_i \cup \{m+1\}} + c^*$, where
\[
c^* :=
\begin{cases}
0, & \text{if $c = 0$ and $p$ is even or $c = 1$ and $p$ is odd,} \\
\monster{1}^{\{m,m+1\}} = x_m + x_{m+1}, & \text{if $c = 0$ and $p$ is odd or $c = 1$ and $p$ is even,}
\end{cases}
\]
and let $f^* := f + h$.
We have $f^* \in \clD{\ell+1} \cap \clChar{1} \cap \clBoth{0}{0}$ because $f$, $h$, and $0$ belong to the class $\clD{\ell+1} \cap \clChar{1} \cap \clBoth{0}{0}$, which is $\clLc$\hyp{}stable.

We claim that $f^*$ does not depend on the $(m+1)$\hyp{}st argument.
This will follow if we show that for every $S \subseteq \nset{m}$, $S \cup \{m+1\} \in \monomials{f}$ if and only if $S \cup \{m+1\} \in \monomials{h}$, as this implies that no monomial of $f^* = f + h$ contains $m+1$.
So, let $S \subseteq \nset{m}$.
By the definition of $g = f'_{m+1}$, we have $S \cup \{m+1\} \in \monomials{f}$ if and only if $S \in \monomials{g}$.
Consider first the case that $S \neq \emptyset$.
We have $S \in \monomials{g}$ if and only if $S \in \monomials{g + c}$, which is equivalent to the condition that $S \subseteq S_i$ for an odd number of the sets $S_i$.
This is equivalent to the condition that $S \cup \{m+1\} \subseteq S_i \cup \{m+1\}$ for an odd number of the sets $S_i$, which in turn is equivalent to $S \cup \{m+1\} \in \monomials{h}$.
As for the case $S = \emptyset$,
the definition of $c^*$ guarantees that $\{m+1\} \in \monomials{h}$ if and only if $c = 1$, that is, $\emptyset \in \monomials{g}$, or, equivalently, $\{m+1\} \in \monomials{f}$.

Let now $f^{**}$ be the $m$\hyp{}ary function obtained from $f^*$ by removing the fictitious $(m+1)$\hyp{}st argument; then $f^*$ and $f^{**}$ are minors of each other.
By the induction hypothesis, $f^{**}$ is a sum of minors of $\monster{\ell+1}$, and consequently so is $f^*$ and hence also $f^* + h = f$.

As for the last claim about $\gen{\monster{k}}$,
since $0 = \monster{0}$ is a minor of $\monster{k}$, it follows that every sum of minors of $\monster{k}$ (not just every odd sum) is in $\gen{\monster{k}}$.
Therefore, by what we have shown above,
$\clD{k} \cap \clChar{1} \cap \clBoth{0}{0} \subseteq \gen{\monster{k}} \subseteq \clD{k} \cap \clChar{1} \cap \clBoth{0}{0}$.
\end{proof}

\begin{lemma}
\label{lem:f=g+h}
A Boolean function $f$ belongs to $\clChar{k}$ if and only if $f = g + h$ for some $g \in \clChar{0}$ and $h \in \clD{k}$.
\end{lemma}

\begin{proof}
``$\Leftarrow$'':
Clear because $\clChar{0} \subseteq \clChar{k}$, $\clD{k} \subseteq \clChar{k}$, and $\clChar{k}$ is closed under sums by Lemma~\ref{lem:Xk-closed}.

``$\Rightarrow$'':
Assume $f \in \clChar{k}^{(n)}$, and let $h := x_1 \cdot (f + f^\mathrm{n})$ and $g := f + h$.
We have $f = g + h$ by definition.
Since $\deg(f + f^\mathrm{n}) \leq k-1$ by Lemma~\ref{lem:Xk-description}, 
we have $\deg(h) \leq k$, so $h \in \clD{k}$.
Furthermore, for any $\vect{a} = (a_1, \dots, a_n) \in \{0,1\}^n$, we have
\begin{align*}
g(\overline{\vect{a}})
& = f(\overline{\vect{a}}) + (a_1 + 1)(f(\overline{\vect{a}}) + f^\mathrm{n}(\overline{\vect{a}}))
\\ & = f(\overline{\vect{a}}) + a_1 \cdot (f(\overline{\vect{a}}) + f^\mathrm{n}(\overline{\vect{a}})) + (f(\overline{\vect{a}}) + f^\mathrm{n}(\overline{\vect{a}}))
\\ & = f(\vect{a}) + a_1 \cdot (f(\vect{a}) + f^\mathrm{n}(\vect{a}))
= g(\vect{a}),
\end{align*}
and thus $g$ is reflexive, i.e., $g \in \clChar{0}$.
\end{proof}

\begin{lemma}
\label{lem:sum-of-two}
For any $k \geq 2$, $\gen{x_1 \dots x_k + x_1} = \clD{k} \cap \clBoth{0}{0}$.
\end{lemma}

\begin{proof}
Let $f := x_1 \dots x_k + x_1$.
We have $f \in \clD{k} \cap \clBoth{0}{0}$, so $\gen{f} \subseteq \clD{k} \cap \clBoth{0}{0}$.
By permuting arguments we get $g := x_1 x_{k+1} x_{k+2} \dots x_{2k-1} + x_1 \in \gen{f}$, and by identifying all arguments we get $0 \in \gen{f}$; hence also $h := f + g + 0 = x_1 \dots x_k + x_1 x_{k+1} x_{k+2} \dots x_{2k-1} \in \gen{f}$.
Again by permuting the arguments of $h$ we get $h' := x_1 x_{k+1} x_{k+2} \dots x_{2k-1} + x_{2k-1} x_{2k} \dots x_{3k-2} \in \gen{f}$; hence also $h'' := h + h' + 0 = x_1 \dots x_k + x_{2k-1} x_{2k} \dots x_{3k-2} \in \gen{f}$.
It is clear that any even sum of monomials of degree at most $k$ can be obtained by adding (an odd number of) minors of $h''$.
Therefore $\clD{k} \cap \clBoth{0}{0} = \clD{k} \cap \clEven \cap \clCon{0} \subseteq \gen{h''} \subseteq \gen{f}$.
\end{proof}

\begin{proposition}
\label{prop:gen:DiXjC0E0}
Let $a \in \{0,1\}$.
\begin{enumerate}[label=\upshape{(\roman*)}, leftmargin=*, widest=iii]
\item\label{gen:DkX1C0E0}
Let $k \in \IN_{+}$. For any $f \in (\clD{k} \cap \clChar{1} \cap \clBoth{a}{a}) \setminus \clD{k-1}$, we have $\gen{f} = \clD{k} \cap \clChar{1} \cap \clBoth{a}{a}$.
\item\label{gen:DkC0E0}
Let $k \in \IN_{+}$ with $k \geq 2$. For any $g \in (\clD{k} \cap \clBoth{a}{a}) \setminus \clChar{k-1}$, we have $\gen{g} = \clD{k} \cap \clBoth{a}{a}$.
\item\label{gen:DiXjC0E0}
Let $i, j \in \IN$ with $i > j \geq 2$. For any $f, g \in \clD{i} \cap \clChar{j} \cap \clBoth{a}{a}$ such that $f \notin \clD{i-1}$ and $g \notin \clChar{j-1}$, we have $\gen{f,g} = \clD{i} \cap \clChar{j} \cap \clBoth{a}{a}$.
\end{enumerate}
\end{proposition}

\begin{proof}
It suffices to prove the statements for $a = 0$.
The statements for $a = 1$ follow by Lemma~\ref{lem:gen-C+1} because $\overline{\clD{i} \cap \clChar{j} \cap \clBoth{0}{0}} = \clD{i} \cap \clChar{j} \cap \clBoth{1}{1}$, $\overline{\clD{i-1}} = \clD{i-1}$, and $\overline{\clChar{j-1}} = \clChar{j-1}$.
Note that $\clBoth{0}{0} \subseteq \clEven$.

\ref{gen:DkX1C0E0}
We proceed by induction on $k$.
For $k = 1$, take any $f \in (\clD{1} \cap \clChar{1} \cap \clBoth{0}{0}) \setminus \clD{0} = (\clD{1} \cap \clEven \cap \clCon{0}) \setminus \clD{0}$.
The function $f$ is a sum of an even nonzero number of arguments, so by identification of arguments we get $0, x_1 + x_2 \in \gen{f}$, and with these we can generate every even sum:
$\clD{1} \cap \clChar{1} \cap \clBoth{0}{0} = \clD{1} \cap \clEven \cap \clCon{0} \subseteq \gen{f} \subseteq \clD{1} \cap \clChar{1} \cap \clBoth{0}{0}$.

Assume that the claim holds for $k = \ell$ for some $\ell \geq 1$.
Let $f \in (\clD{\ell+1} \cap \clChar{1} \cap \clBoth{0}{0}) \setminus \clD{\ell}$.
Since $\clChar{1} \cap \clBoth{0}{0} \subseteq \clChar{1} \cap \clEven = \clChar{0}$, Lemma~\ref{lem:arity-degree}\ref{lem:arity-degree:minor2} implies that $f$ has an $(\ell + 2)$\hyp{}ary minor $\varphi$ such that $\deg(\varphi) = \ell + 1$ and $\monomials{\varphi}$ contains all $(\ell + 1)$\hyp{}element subsets of $\nset{\ell + 2}$ and a subset $S$ of cardinality $\ell$.
By identifying the two arguments not in $S$, we obtain a minor $\varphi'$ of $\varphi$ such that $\varphi' \in \clChar{0}$, $\arity{\varphi'} = \ell + 1$, and $\deg(\varphi') \geq \ell > 0$, so by Lemma~\ref{lem:arity-degree}\ref{lem:arity-degree:atleast} we must have $\deg(\varphi') = \ell$.
Since $\varphi' \in (\clD{\ell} \cap \clChar{1} \cap \clBoth{0}{0}) \setminus \clD{\ell-1}$, it holds that $\gen{\varphi'} = \clD{\ell} \cap \clChar{1} \cap \clBoth{0}{0}$ by the induction hypothesis.
All monomials of degree $\ell + 1$ are cancelled in the sum $\varphi + \monster{\ell + 1}$, so we have $\varphi + \monster{\ell + 1} \in \clD{\ell} \cap \clChar{1} \cap \clBoth{0}{0} = \gen{\varphi'} \subseteq \gen{f}$.
Since also $\varphi, 0 \in \gen{f}$, we get $\monster{\ell + 1} = (\varphi + \monster{\ell + 1}) + \varphi + 0 \in \gen{f}$.
By Lemma~\ref{lem:Char0-sum-Wk},
$\clD{\ell+1} \cap \clChar{1} \cap \clBoth{0}{0} = \gen{\monster{\ell + 1}} \subseteq \gen{f} \subseteq \clD{\ell+1} \cap \clChar{1} \cap \clBoth{0}{0}$.

\ref{gen:DkC0E0}
We proceed by induction on $k$.
For $k = 2$, let $g \in (\clD{2} \cap \clBoth{0}{0}) \setminus \clChar{1}$.
Since $\clD{2} \subseteq \clChar{2}$, $g$ has a binary minor $\gamma$ with $\nset{2} \in \monomials{\gamma}$ by Lemma~\ref{lem:arity-degree}\ref{lem:arity-degree:minor1}.
Since $\gamma \in \clBoth{0}{0} \subseteq \clEven$, we have $\gamma \equiv x_1 x_2 + x_1$.
It follows from Lemma~\ref{lem:sum-of-two} that
$\clD{2} \cap \clBoth{0}{0} = \gen{x_1 x_2 + x_1} \subseteq \gen{g} \subseteq \clD{2} \cap \clBoth{0}{0}$.

Assume that the claim holds for $k = \ell$ for some $\ell \geq 2$.
Let $g \in (\clD{\ell + 1} \cap \clBoth{0}{0}) \setminus \clChar{\ell}$.
Since $\clD{\ell + 1} \subseteq \clChar{\ell + 1}$, $g$ has an $(\ell+1)$\hyp{}ary minor $\gamma$ with $\nset{\ell+1} \in \monomials{\gamma}$ by Lemma~\ref{lem:arity-degree}\ref{lem:arity-degree:minor1}.
By Lemma~\ref{lem:arity-degree}\ref{lem:arity-degree:minor3}, $\gamma$ has an $\ell$\hyp{}ary minor $\gamma_{ij} \in (\clD{\ell} \cap \clBoth{0}{0}) \setminus \clChar{\ell - 1}$.
By the inductive hypothesis, $\clD{\ell} \cap \clBoth{0}{0} = \gen{\gamma_{ij}} \subseteq \gen{g}$.
The functions $\gamma' := \gamma + (x_1 \dots x_{\ell+1} + x_1)$ and $0$ are members of $\clD{\ell} \cap \clBoth{0}{0} \subseteq \gen{g}$,
so also $x_1 \dots x_{\ell + 1} + x_1 = \gamma' + \gamma + 0 \in \gen{g}$.
By Lemma~\ref{lem:sum-of-two}, $\clD{\ell + 1} \cap \clBoth{0}{0} = \gen{x_1 \dots x_{\ell + 1} + x_1} \subseteq \gen{g} \subseteq \clD{\ell + 1} \cap \clBoth{0}{0}$.

\ref{gen:DiXjC0E0}
Let $f, g \in \clD{i} \cap \clChar{j} \cap \clBoth{0}{0}$ such that $f \notin \clD{i-1}$ and $g \notin \clChar{j-1}$.
By Lemma~\ref{lem:arity-degree}\ref{lem:arity-degree:minor1}, $g$ has a $j$\hyp{}ary minor $g' \in \clD{j} \setminus \clChar{j-1}$.
Since $\clD{i} \cap \clChar{j} \cap \clBoth{0}{0}$ is minor\hyp{}closed, we have 
$g' \in (\clD{j} \cap \clBoth{0}{0}) \setminus \clChar{j-1}$,
and by part \ref{gen:DkC0E0}, $\gen{g'} = \clD{j} \cap \clBoth{0}{0}$.

By Lemma~\ref{lem:f=g+h}, $f = f_1 + f_2$ for some $f_1 \in \clChar{0}$ and $f_2 \in \clD{j}$.
Since $\clChar{0} \subseteq \clEven$ and $f \in \clBoth{0}{0} \subseteq \clEven$, we must also have $f_2 \in \clEven$.
Since $f \in \clCon{0}$, it is clear that by changing the constant terms in $f_1$ and $f_2$ if necessary, we can assume that both $f_1$ and $f_2$ are in the class $\clEven \cap \clCon{0} = \clBoth{0}{0}$.
Thus $f_2 \in \clD{j} \cap \clBoth{0}{0} = \gen{g'} \subseteq \gen{g}$,
so $f_1 = f + f_2 + 0 \in \gen{f,g}$.
Since $f_1 \in (\clD{i} \cap \clChar{0} \cap \clCon{0}) \setminus \clD{i-1} = (\clD{i} \cap \clChar{1} \cap \clBoth{0}{0}) \setminus \clD{i-1}$, we have $\gen{f_1} = \clD{i} \cap \clChar{1} \cap \clBoth{0}{0}$ by part \ref{gen:DkX1C0E0}.
It follows from Lemma~\ref{lem:f=g+h} that
\begin{align*}
\clD{i} \cap \clChar{j} \cap \clBoth{0}{0}
&
= \{ \, \alpha + \beta \mid \alpha \in \clD{i} \cap \clChar{0} \cap \clBoth{0}{0}, \, \beta \in \clD{j} \cap \clBoth{0}{0} \, \}
\\ &
= \{ \, \alpha + \beta + 0 \mid \alpha \in \clD{i} \cap \clChar{1} \cap \clBoth{0}{0}, \, \beta \in \clD{j} \cap \clBoth{0}{0} \, \}
\\ &
\subseteq
\gen{f_1, g'}
\subseteq
\gen{f,g}
\subseteq
\clD{i} \cap \clChar{j} \cap \clBoth{0}{0}.
\qedhere
\end{align*}
\end{proof}

\begin{lemma}
\label{lem:x1xk}
For any $k \in \IN_{+}$, $\gen{x_1 \dots x_k} = \clD{k} \cap \clBoth{0}{1}$.
\end{lemma}

\begin{proof}
It is clear that any monomial of degree at most $k$ can be obtained as a minor of $x_1 \dots x_k$.
Any function in $\clD{k} \cap \clBoth{0}{1} = \clD{k} \cap \clOdd \cap \clCon{0}$ is an odd sum of monomials of degree at most $k$.
Therefore $\clD{k} \cap \clBoth{0}{1} \subseteq \gen{x_1 \dots x_k} \subseteq \clD{k} \cap \clBoth{0}{1}$.
\end{proof}

\begin{proposition}
\label{prop:gen:DiXjC0E1}
Let $a \in \{0,1\}$.
\begin{enumerate}[label=\upshape{(\roman*)}, leftmargin=*, widest=iii]
\item\label{gen:D1C0E1}
For any $f \in \clD{1} \cap \clBoth{a}{\overline{a}}$, we have $\gen{f} = \clD{1} \cap \clBoth{a}{\overline{a}}$.
\item\label{gen:DkX1C0E1}
Let $k \in \IN_{+}$ with $k \geq 2$. For any $f \in (\clD{k} \cap \clChar{1} \cap \clBoth{a}{\overline{a}}) \setminus \clD{k-1}$, we have $\gen{f} = \clD{k} \cap \clChar{1} \cap \clBoth{a}{\overline{a}}$.
\item\label{gen:DkC0E1}
Let $k \in \IN_{+}$ with $k \geq 2$. For any $g \in (\clD{k} \cap \clBoth{a}{\overline{a}}) \setminus \clChar{k-1}$, we have $\gen{g} = \clD{k} \cap \clBoth{a}{\overline{a}}$.
\item\label{gen:DiXjC0E1}
Let $i, j \in \IN$ with $i > j \geq 2$. For any $f, g \in \clD{i} \cap \clChar{j} \cap \clBoth{a}{\overline{a}}$ such that $f \notin \clD{i-1}$ and $g \notin \clChar{j-1}$, we have $\gen{f,g} = \clD{i} \cap \clChar{j} \cap \clBoth{a}{\overline{a}}$.
\end{enumerate}
\end{proposition}

\begin{proof}
It suffices to prove the statements for $a = 0$.
The statements for $a = 1$ follow by Lemma~\ref{lem:gen-C+1} because $\overline{\clD{i} \cap \clChar{j} \cap \clBoth{0}{1}} = \clD{i} \cap \clChar{j} \cap \clBoth{1}{0}$, $\overline{\clD{i-1}} = \clD{i-1}$, and $\overline{\clChar{j-1}} = \clChar{j-1}$.

\ref{gen:D1C0E1}
If $f \in \clD{1} \cap \clBoth{0}{1}$, then by identifying all arguments we get $x_1 \in \gen{f}$.
By Lemma~\ref{lem:x1xk}, we have $\clD{1} \cap \clBoth{0}{1} = \clLc = \gen{x_1} \subseteq \gen{f} \subseteq \clD{1} \cap \clBoth{0}{1}$.

\ref{gen:DkX1C0E1}
We show by induction on $k$ that the claim holds for any $k \geq 1$ (not just for $k \geq 2$).
The basis of the induction, the case when $k = 1$, is, in fact, statement \ref{gen:D1C0E1} that we have already established; note that $(\clD{1} \cap \clChar{1} \cap \clBoth{0}{1}) \setminus \clD{0} = \clD{1} \cap \clBoth{0}{1}$.
For the induction step, assume that the claim holds for $k = \ell$ for some $\ell \geq 1$.
Let $f \in (\clD{\ell + 1} \cap \clChar{1} \cap \clBoth{0}{1}) \setminus \clD{\ell}$.
By Lemma~\ref{lem:arity-degree}\ref{lem:arity-degree:minor2}, $f$ has an $(\ell + 2)$\hyp{}ary minor $\varphi \in (\clD{\ell + 1} \cap \clChar{1}\cap \clBoth{0}{1}) \setminus \clD{\ell}$ such that $\monomials{\varphi}$ contains all subsets of $\nset{\ell + 2}$ of cardinality $\ell + 1$.
If $\ell \geq 2$, then $\monomials{\varphi}$ furthermore contains a subset of cardinality $\ell$; then, again by Lemma~\ref{lem:arity-degree}\ref{lem:arity-degree:minor2}, $\varphi$ has an $(\ell + 1)$\hyp{}ary minor $\varphi' \in (\clD{\ell} \cap \clChar{1} \cap \clBoth{0}{1}) \setminus \clD{\ell-1}$, and by the inductive hypothesis, $\clD{\ell} \cap \clChar{1} \cap \clBoth{0}{1} = \gen{\varphi'} \subseteq \gen{f}$.
If $\ell = 1$, then $\clD{\ell} \cap \clChar{1} \cap \clBoth{0}{1} = \clLc = \gen{x_1} \subseteq \gen{f}$ because $x_1$ is a minor of $f$ (identify all arguments).
In either case, let $\lambda := \monster{\ell + 1} + \varphi$.
We have $\monster{\ell + 1} \in \clD{\ell + 1} \cap \clChar{1} \cap \clBoth{0}{0}$ by Lemma~\ref{lem:monster} and $\varphi \in \clD{\ell + 1} \cap \clChar{1} \cap \clBoth{0}{1}$.
Consequently $\lambda \in \clD{\ell} \cap \clChar{1} \cap \clBoth{0}{1} \subseteq \gen{f}$
because all monomials of degree $\ell + 1$ are cancelled in the sum $\monster{\ell + 1} + \varphi$, $\clChar{1}$ is closed under sums by Lemma~\ref{lem:Xk-closed}, and
\begin{align*}
& \lambda(0, \dots, 0) = \monster{\ell + 1}(0, \dots, 0) + \varphi(0, \dots, 0) = 0 + 0 = 0, \\
& \lambda(1, \dots, 1) = \monster{\ell + 1}(1, \dots, 1) + \varphi(1, \dots, 1) = 0 + 1 = 1.
\end{align*}

Let now $h \in (\clD{\ell + 1} \cap \clChar{1} \cap \clBoth{0}{1}) \setminus \clD{\ell}$ be arbitrary.
Then $h + x_1 \in \clD{\ell + 1} \cap \clChar{1} \cap \clBoth{0}{0}$, so by Lemma~\ref{lem:Char0-sum-Wk}, $h + x_1 \in \gen{\monster{\ell + 1}}$, that is, $h + x_1 = \sum_{i=1}^{2m+1} \monster{k_i}^{S_i}$ with $k_i \leq \ell + 1$.
We can write $\monster{k_i}^{S_i} = (\monster{\ell+1})_{\sigma_i}$ for a suitable minor formation map $\sigma_i$.
Consequently,
\begin{align*}
h
&
= \bigl( \sum_{i=1}^{2m+1} \monster{k_i}^{S_i} \bigr) + x_1
= \bigl( \sum_{i=1}^{2m+1} (\monster{\ell+1})_{\sigma_i} \bigr) + x_1
= \bigl( \sum_{i=1}^{2m+1} (\monster{\ell+1} + \varphi + \varphi)_{\sigma_i} \bigr) + x_1
\\ &
= \bigl( \sum_{i=1}^{2m+1} \bigl( \underbrace{(\monster{\ell+1} + \varphi)_{\sigma_i}}_{\in \gen{f}} + \underbrace{\varphi_{\sigma_i}}_{\in \gen{f}} \bigr) \bigr) + \underbrace{x_1,}_{\in \gen{f}}
\end{align*}
where the last equality holds by Lemma~\ref{lem:comp-minors}.
Since the last expression is an odd sum of elements of $\gen{f}$, it follows that $h \in \gen{f}$.
We conclude that
$\clD{\ell + 1} \cap \clChar{1} \cap \clBoth{0}{1} \subseteq \gen{f} \subseteq \clD{\ell + 1} \cap \clChar{1} \cap \clBoth{0}{1}$.

\ref{gen:DkC0E1}
We show by induction on $k$ that the claim holds for any $k \geq 1$.
The basis of the induction, the case when $k = 1$, is, in fact, statement \ref{gen:D1C0E1} that we have already established, because $(\clD{1} \cap \clBoth{0}{1}) \setminus \clChar{0} = \clD{1} \cap \clBoth{0}{1}$.
For the induction step, assume that the claim holds for $k = \ell$ for some $\ell \geq 1$.
Let $g \in (\clD{\ell+1} \cap \clBoth{0}{1}) \setminus \clChar{\ell}$.
By Lemma~\ref{lem:arity-degree}\ref{lem:arity-degree:minor1}, $g$ has an $(\ell+1)$\hyp{}ary minor $\gamma$ such that $\nset{\ell+1} \in \monomials{\gamma}$ and $\gamma \in (\clD{\ell+1} \cap \clBoth{0}{1}) \setminus \clChar{\ell}$.
By Lemma~\ref{lem:arity-degree}\ref{lem:arity-degree:minor3}, $\gamma$ has an $\ell$\hyp{}ary minor $\gamma_{ij} \in (\clD{\ell} \cap \clBoth{0}{1}) \setminus \clChar{\ell-1}$.
By the inductive hypothesis, $\clD{\ell} \cap \clBoth{0}{1} = \gen{\gamma_{ij}} \subseteq \gen{g}$.
We have $\gamma' := \gamma + x_1 \dots x_{\ell+1} + x_1 \in \clD{\ell} \cap \clBoth{0}{1} \subseteq \gen{g}$ and clearly $x_1 \in \gen{g}$, so also $x_1 \dots x_{\ell + 1} = \gamma' + \gamma + x_1 \in \gen{g}$.
By Lemma~\ref{lem:x1xk}, we have $\clD{\ell+1} \cap \clBoth{0}{1} = \gen{x_1 \dots x_{\ell+1}} \subseteq \gen{g} \subseteq \clD{\ell+1} \cap \clBoth{0}{1}$.

\ref{gen:DiXjC0E1}
Let $f, g \in \clD{i} \cap \clChar{j} \cap \clBoth{0}{1}$ such that $f \notin \clD{i-1}$ and $g \notin \clChar{j-1}$.
By Lemma~\ref{lem:arity-degree}\ref{lem:arity-degree:minor1}, $g$ has a $j$\hyp{}ary minor $g' \in \clD{j} \setminus \clChar{j-1}$.
Since $\clD{i} \cap \clChar{j} \cap \clBoth{0}{1}$ is minor\hyp{}closed, we have $g' \in (\clD{j} \cap \clBoth{0}{1}) \setminus \clChar{j-1}$, and by part \ref{gen:DkC0E1}, $\gen{g'} = \clD{j} \cap \clBoth{0}{1}$.

By Lemma~\ref{lem:f=g+h}, $f = f_1 + f_2$ for some $f_1 \in \clChar{0}$ and $f_2 \in \clD{j}$.
Since $\clChar{0} \subseteq \clEven$ and $f \in \clOdd$, we must also have $f_2 \in \clOdd$.
Since $f \in \clCon{0}$, it is clear that by changing the constant terms if necessary, we may assume that both $f_1$ and $f_2$ are in $\clCon{0}$.
Thus $f_2 \in \clD{j} \cap \clOdd \cap \clCon{0} = \clD{j} \cap \clBoth{0}{1} = \gen{g'} \subseteq \gen{g}$,
so $f_1 + x_1 = f + f_2 + x_1 \in \gen{f,g}$.
Since $f_1 \in (\clD{i} \cap \clChar{0} \cap \clCon{0}) \setminus \clD{i-1} = (\clD{i} \cap \clChar{1} \cap \clEven \cap \clCon{0}) \setminus \clD{i-1}$,
we have $f_1 + x_1 \in (\clD{i} \cap \clChar{1} \cap \clOdd \cap \clCon{0}) \setminus \clD{i-1} = (\clD{i} \cap \clChar{1} \cap \clBoth{0}{1}) \setminus \clD{i-1}$,
so $\gen{f_1 + x_1} = \clD{i} \cap \clChar{1} \cap \clBoth{0}{1}$ by part \ref{gen:DkX1C0E1}.

Now, with the help of Lemma~\ref{lem:f=g+h}, we can see that for any $h \in \clD{i} \cap \clChar{j} \cap \clBoth{0}{1}$, we have $h = h_1 + h_2$ for some $h_1 \in \clD{i} \cap \clChar{0} \cap \clCon{0} = \clD{i} \cap \clChar{1} \cap \clEven \cap \clCon{0}$ and $h_2 \in \clD{j} \cap \clOdd \cap \clCon{0} = \clD{j} \cap \clBoth{0}{1}$, and hence
$h_1 + x_1 \in \clD{i} \cap \clChar{1} \cap \clBoth{0}{1} \subseteq \gen{f,g}$ and
$h_2 \in \gen{g}$.
Since $x_1 \in \gen{f}$ as well, we have $h = (h_1 + x_1) + h_2 + x_1 \in \gen{f,g}$.
We conclude that $\clD{i} \cap \clChar{j} \cap \clBoth{0}{1} \subseteq \gen{f,g} \subseteq \clD{i} \cap \clChar{j} \cap \clBoth{0}{1}$.
\end{proof}

\begin{proposition}
\label{prop:gen:CaEb}
Let $a, b \in \{0,1\}$.
\begin{enumerate}[label=\upshape{(\roman*)}, leftmargin=*, widest=iii]
\item\label{gen:X1CaEb}
For any $f_i \in (\clChar{1} \cap \clBoth{a}{b}) \setminus \clD{i}$ \textup{(}$i \in \IN_{+}$\textup{)}, we have $\gen{\{ \, f_i \mid i \in \IN_{+} \, \}} = \clChar{1} \cap \clBoth{a}{b}$.
\item\label{gen:XjCaEb}
Let $k \in \IN_{+}$ with $k \geq 2$. For any $f_i \in (\clChar{k} \cap \clBoth{a}{b}) \setminus \clD{i}$ \textup{(}$i \in \IN_{+}$\textup{)} and $g \in (\clChar{k} \cap \clBoth{a}{b}) \setminus \clChar{k-1}$, we have $\gen{\{ \, f_i \mid i \in \IN_{+} \, \} \cup \{g\}} = \clChar{k} \cap \clBoth{a}{b}$.
\item\label{gen:CaEb}
For any $g_i \in (\clBoth{a}{b}) \setminus \clChar{i}$ \textup{(}$i \in \IN_{+}$\textup{)}, we have $\gen{\{\, g_i \mid i \in \IN_{+} \,\}} = \clBoth{a}{b}$.
\end{enumerate}
\end{proposition}

\begin{proof}
\ref{gen:X1CaEb}
For $i \in \IN_{+}$, let $f_i \in (\clChar{1} \cap \clBoth{a}{b}) \setminus \clD{i}$,
and let $n_i := \deg(f_i)$; we have $n_i > i$.
Then $f_i \in (\clD{n_i} \cap \clChar{1} \cap \clBoth{a}{b}) \setminus \clD{n_i - 1}$, so by Proposition~\ref{prop:gen:DiXjC0E0}\ref{gen:DkX1C0E0} and Proposition~\ref{prop:gen:DiXjC0E1}\ref{gen:DkX1C0E1}, $\gen{f_i} = \clD{n_i} \cap \clChar{1} \cap \clBoth{a}{b}$.
Therefore
\begin{align*}
\clChar{1} \cap \clBoth{a}{b}
&
= \bigcup_{i \in \IN_{+}} (\clD{i} \cap \clChar{1} \cap \clBoth{a}{b})
\subseteq \bigcup_{i \in \IN_{+}} (\clD{n_i} \cap \clChar{1} \cap \clBoth{a}{b})
\\ &
= \bigcup_{i \in \IN_{+}} \gen{f_i}
\subseteq \gen{\{ \, f_i \mid i \in \IN_{+} \, \}}
\subseteq \clChar{1} \cap \clBoth{a}{b}.
\end{align*}

\ref{gen:XjCaEb}
For $i \in \IN_{+}$, let $f_i \in (\clChar{k} \cap \clBoth{a}{b}) \setminus \clD{i}$, and let $g \in (\clChar{k} \cap \clBoth{a}{b}) \setminus \clChar{k-1}$,
and let $n_i := \deg(f_i)$; we have $n_i > i$.
By Lemma~\ref{lem:arity-degree}\ref{lem:arity-degree:minor1}, $g$ has a $k$\hyp{}ary minor $\gamma$ of degree $k$ such that $\gamma \in (\clD{k} \cap \clBoth{a}{b}) \setminus \clChar{k - 1}$.
By Proposition~\ref{prop:gen:DiXjC0E0}\ref{gen:DiXjC0E0} and Proposition~\ref{prop:gen:DiXjC0E1}\ref{gen:DiXjC0E1}, it holds for $i \geq k$ that $\gen{f_i, g} = \clD{n_i} \cap \clChar{k} \cap \clBoth{a}{b}$.
Therefore
\begin{align*}
\clChar{k} \cap \clBoth{a}{b}
&
= \bigcup_{i \in \IN_{+}} (\clD{i} \cap \clChar{k} \cap \clBoth{a}{b})
= \bigcup_{i \geq k} (\clD{i} \cap \clChar{k} \cap \clBoth{a}{b})
\subseteq \bigcup_{i \geq k} (\clD{n_i} \cap \clChar{k} \cap \clBoth{a}{b})
\\ &
= \bigcup_{i \geq k} \gen{f_i, g}
\subseteq \gen{\{ \, f_i \mid i \in \IN_{+} \, \} \cup \{g\}}
\subseteq \clChar{k} \cap \clBoth{a}{b}.
\end{align*}

\ref{gen:CaEb}
For $i \in \IN_{+}$, let $g_i \in (\clBoth{a}{b}) \setminus \clChar{i}$,
and let $k_i := \charrank{g_i}$.
Then $g_i \in (\clChar{k_i} \cap \clBoth{a}{b}) \setminus \clChar{k_i - 1}$.
By Lemma~\ref{lem:arity-degree}\ref{lem:arity-degree:minor1}, $g_i$ has a $k_i$\hyp{}ary minor $\gamma_i$ of degree $k_i$ such that $\gamma_i \in (\clD{k_i} \cap \clBoth{a}{b}) \setminus \clChar{k_i - 1}$.
By Proposition~\ref{prop:gen:DiXjC0E0}\ref{gen:DkC0E0} and Proposition~\ref{prop:gen:DiXjC0E1}\ref{gen:DkC0E1}, $\gen{\gamma_i} = \clD{k_i} \cap \clBoth{a}{b}$.
Therefore
\begin{align*}
\clBoth{a}{b}
&
= \bigcup_{i \in \IN_{+}} (\clD{i} \cap \clBoth{a}{b})
\subseteq \bigcup_{i \in \IN_{+}} (\clD{k_i} \cap \clBoth{a}{b})
= \bigcup_{i \in \IN_{+}} \gen{\gamma_i}
\\ &
\subseteq \bigcup_{i \in \IN_{+}} \gen{g_i}
\subseteq \gen{\{ \, g_i \mid i \in \IN_{+} \, \}}
\subseteq \clBoth{a}{b}.
\qedhere
\end{align*}
\end{proof}

\begin{proposition}
\label{prop:gen:DiXjCa}
Let $a \in \{0,1\}$.
\begin{enumerate}[label=\upshape{(\roman*)}, leftmargin=*, widest=iii]
\item\label{gen:DkX1Ca}
Let $k \in \IN_{+}$. For any $f, h, h' \in \clD{k} \cap \clChar{1} \cap \clCon{a}$ with $f \notin \clD{k-1}$, $h \notin \clYksi{a}$, $h' \notin \clYksi{\overline{a}}$, we have $\gen{f, h, h'} = \clD{k} \cap \clChar{1} \cap \clCon{a}$.
\item\label{gen:DkCa}
Let $k \in \IN_{+}$ with $k \geq 2$. For any $g, h, h' \in \clD{k} \cap \clCon{a}$ with $g \notin \clChar{k-1}$, $h \notin \clYksi{a}$, $h' \notin \clYksi{\overline{a}}$, we have $\gen{g, h, h'} = \clD{k} \cap \clCon{a}$.
\item\label{gen:DiXjCa}
Let $i, j \in \IN$ with $i > j \geq 2$. For any $f, g, h, h' \in \clD{i} \cap \clChar{j} \cap \clCon{a}$ such that $f \notin \clD{i-1}$, $g \notin \clChar{j-1}$, $h \notin \clYksi{a}$, $h' \notin \clYksi{\overline{a}}$, we have $\gen{f, g, h, h'} = \clD{i} \cap \clChar{j} \cap \clCon{a}$.
\end{enumerate}
\end{proposition}

\begin{proof}
It suffices to prove the statements for $a = 0$.
The statements for $a = 1$ follow by Lemma~\ref{lem:gen-C+1} because $\overline{\clD{i} \cap \clChar{j} \cap \clCon{0}} = \clD{i} \cap \clChar{j} \cap \clCon{1}$, $\overline{\clD{i-1}} = \clD{i-1}$, $\overline{\clChar{j-1}} = \clChar{j-1}$, $\overline{\clYksi{0}} = \clYksi{1}$, and $\overline{\clYksi{1}} = \clYksi{0}$.
We consider only statement \ref{gen:DiXjCa}.
The proofs of statements \ref{gen:DkX1Ca} and \ref{gen:DkCa} are analogous; we just need to omit the parts of the proof that deal with the function $f$ or $g$, as the case may be, that does not appear in the statement.

Since $\{ \clYksi{0}, \clYksi{1} \}$ is a partition of $\clAll$, we have that $h \in \clYksi{1}$ and $h' \in \clYksi{0}$.
By identifying all arguments, we get $x_1 \in \gen{h}$ and $0 \in \gen{h'}$,
so we have $f + x_1 = f + x_1 + 0 \in \gen{f, h, h'}$ and $g + x_1 = g + x_1 + 0 \in \gen{g, h, h'}$.
One of the functions $f$ and $f + x_1$ belongs to the class $(\clD{i} \cap \clChar{j} \cap \clBoth{0}{0}) \setminus \clD{i-1}$ and the other to $(\clD{i} \cap \clChar{j} \cap \clBoth{0}{1}) \setminus \clD{i-1}$,
and, similarly,
one of $g$ and $g + x_1$ belongs to $(\clD{i} \cap \clChar{j} \cap \clBoth{0}{0}) \setminus \clChar{j-1}$ and the other to $(\clD{i} \cap \clChar{j} \cap \clBoth{0}{1}) \setminus \clChar{j-1}$.
Propositions~\ref{prop:gen:DiXjC0E0}\ref{gen:DiXjC0E0} and \ref{prop:gen:DiXjC0E1}\ref{gen:DiXjC0E1} imply that $\gen{f, g, h, h'}$ contains a generating set for both $\clD{i} \cap \clChar{j} \cap \clBoth{0}{0}$ and $\clD{i} \cap \clChar{j} \cap \clBoth{0}{1}$.
Therefore
\[
\clD{i} \cap \clChar{j} \cap \clCon{0}
= (\clD{i} \cap \clChar{j} \cap \clBoth{0}{0}) \cup (\clD{i} \cap \clChar{j} \cap \clBoth{0}{1})
\subseteq \gen{f, g, h, h'}
\subseteq \clD{i} \cap \clChar{j} \cap \clCon{0}.
\qedhere
\]
\end{proof}

\begin{proposition}
\label{prop:gen:DiXjEa}
Let $a \in \{0,1\}$.
\begin{enumerate}[label=\upshape{(\roman*)}, leftmargin=*, widest=iii]
\item\label{gen:DkX1Ea}
Let $k \in \IN_{+}$. For any $f, h, h' \in \clD{k} \cap \clChar{1} \cap \clYksi{a}$ with $f \notin \clD{k-1}$, $h \notin \clCon{a}$, $h' \notin \clCon{\overline{a}}$, we have $\gen{f, h, h'} = \clD{k} \cap \clChar{1} \cap \clYksi{a}$.
\item\label{gen:DkEa}
Let $k \in \IN_{+}$ with $k \geq 2$. For any $g, h, h' \in \clD{k} \cap \clYksi{a}$ with $g \notin \clChar{k-1}$, $h \notin \clCon{a}$, $h' \notin \clCon{\overline{a}}$, we have $\gen{g, h, h'} = \clD{k} \cap \clYksi{a}$.
\item\label{gen:DiXjEa}
Let $i, j \in \IN$ with $i > j \geq 2$. For any $f, g, h, h' \in \clD{i} \cap \clChar{j} \cap \clYksi{a}$ such that $f \notin \clD{i-1}$, $g \notin \clChar{j-1}$, $h \notin \clCon{a}$, $h' \notin \clCon{\overline{a}}$, we have $\gen{f, g, h, h'} = \clD{i} \cap \clChar{j} \cap \clYksi{a}$.
\end{enumerate}
\end{proposition}

\begin{proof}
It suffices to prove the statements for $a = 1$.
The statements for $a = 0$ follow by Lemma~\ref{lem:gen-C+1} because
$\overline{\clD{i} \cap \clChar{j} \cap \clYksi{1}} = \clD{i} \cap \clChar{j} \cap \clYksi{0}$, $\overline{\clD{i-1}} = \clD{i-1}$, $\overline{\clChar{j-1}} = \clChar{j-1}$, $\overline{\clCon{1}} = \clCon{0}$, and $\overline{\clCon{0}} = \clCon{1}$.
We consider only statement \ref{gen:DiXjEa}.
The proofs of statements \ref{gen:DkX1Ea} and \ref{gen:DkEa} are analogous; we just need to omit the parts of the proof that deal with the function $f$ or $g$, as the case may be, that does not appear in the statement.

Since $\{ \clCon{0}, \clCon{1} \}$ is a partition of $\clAll$, we have that $h \in \clCon{0}$ and $h' \in \clCon{1}$.
By identifying all arguments, we get $x_1 \in \gen{h}$ and $1 \in \gen{h'}$,
so we have $f + x_1 + 1 \in \gen{f, h, h'}$ and $g + x_1 + 1 \in \gen{g, h, h'}$.
One of the functions $f$ and $f + x_1 + 1$ belongs to the class $(\clD{i} \cap \clChar{j} \cap \clBoth{0}{1}) \setminus \clD{i-1}$ and the other to $(\clD{i} \cap \clChar{j} \cap \clBoth{1}{1}) \setminus \clD{i-1}$,
and, similarly,
one of $g$ and $g + x_1 + 1$ belongs to $(\clD{i} \cap \clChar{j} \cap \clBoth{0}{1}) \setminus \clChar{j-1}$ and the other to $(\clD{i} \cap \clChar{j} \cap \clBoth{1}{1}) \setminus \clChar{j-1}$.
Propositions~\ref{prop:gen:DiXjC0E0}\ref{gen:DiXjC0E0} and \ref{prop:gen:DiXjC0E1}\ref{gen:DiXjC0E1} imply that $\gen{f, g, h, h'}$ contains a generating set for both $\clD{i} \cap \clChar{j} \cap \clBoth{0}{1}$ and $\clD{i} \cap \clChar{j} \cap \clBoth{1}{1}$.
Therefore
\[
\clD{i} \cap \clChar{j} \cap \clYksi{1}
= (\clD{i} \cap \clChar{j} \cap \clBoth{0}{1}) \cup (\clD{i} \cap \clChar{j} \cap \clBoth{1}{1})
\subseteq \gen{f, g, h, h'}
\subseteq \clD{i} \cap \clChar{j} \cap \clYksi{1}.
\qedhere
\]
\end{proof}

\begin{proposition}
\label{prop:gen:DiXjPa}
Let $\QuantifyParRel$.
\begin{enumerate}[label=\upshape{(\roman*)}, leftmargin=*, widest=iii]
\item\label{gen:DkX1Pa}
Let $k \in \IN_{+}$. For any $f, h, h' \in \clD{k} \cap \clChar{1} \cap \clParity{a}$ with $f \notin \clD{k-1}$, $h \notin \clCon{0}$, $h' \notin \clCon{1}$, we have $\gen{f, h, h'} = \clD{k} \cap \clChar{1} \cap \clParity{a}$.
\item\label{gen:DkPa}
Let $k \in \IN_{+}$ with $k \geq 2$. For any $g, h, h' \in \clD{k} \cap \clParity{a}$ with $g \notin \clChar{k-1}$, $h \notin \clCon{0}$, $h' \notin \clCon{1}$, we have $\gen{g, h, h'} = \clD{k} \cap \clParity{a}$.
\item\label{gen:DiXjPa}
Let $i, j \in \IN$ with $i > j \geq 2$. For any $f, g, h, h' \in \clD{i} \cap \clChar{j} \cap \clParity{a}$ such that $f \notin \clD{i-1}$, $g \notin \clChar{j-1}$, $h \notin \clCon{0}$, $h' \notin \clCon{1}$, we have $\gen{f, g, h, h'} = \clD{i} \cap \clChar{j} \cap \clParity{a}$.
\end{enumerate}
\end{proposition}

\begin{proof}
We consider only statement \ref{gen:DiXjPa}.
The proofs of statements \ref{gen:DkX1Pa} and \ref{gen:DkPa} are analogous; we just need to omit the parts of the proof that deal with the function $f$ or $g$, as the case may be, that does not appear in the statement.

Since $\{ \clCon{0}, \clCon{1} \}$ is a partition of $\clAll$, we have that $h \in \clCon{1}$ and $h' \in \clCon{0}$.
By identifying all arguments, we get $1 \in \gen{h}$ and $0 \in \gen{h'}$ if $\mathord{\ParRel}$ is $\mathord{=}$; or $x_1 + 1 \in \gen{h}$ and $x_1 \in \gen{h'}$ if $\mathord{\ParRel}$ is $\mathord{\neq}$.
With the triple sum and these two minors of $h$ and $h'$ we are able to negate functions ($\varphi + 1 = \varphi + 1 + 0$ and $\varphi + 1 = \varphi + (x_1 + 1) + x_1$); hence $f + 1 \in \gen{f, h, h'}$ and $g + 1 \in \gen{g, h, h'}$.
One of the functions $f$ and $f + 1$ belongs to $(\clD{i} \cap \clChar{j} \cap \clCon{0} \cap \clParity{a}) \setminus \clD{i-1}$ and the other to $(\clD{i} \cap \clChar{j} \cap \clCon{1} \cap \clParity{a}) \setminus \clD{i-1}$,
and, similarly,
one of the functions $g$ and $g + 1$ belongs to $(\clD{i} \cap \clChar{j} \cap \clCon{0} \cap \clParity{a}) \setminus \clChar{j-1}$ and the other to $(\clD{i} \cap \clChar{j} \cap \clCon{1} \cap \clParity{a}) \setminus \clChar{j-1}$,
Propositions~\ref{prop:gen:DiXjC0E0}\ref{gen:DiXjC0E0} and \ref{prop:gen:DiXjC0E1}\ref{gen:DiXjC0E1} imply that $\gen{f, g, h, h'}$ contains a generating set for both $\clD{i} \cap \clChar{j} \cap \clCon{0} \cap \clParity{a}$ and $\clD{i} \cap \clChar{j} \cap \clCon{1} \cap \clParity{a}$.
Therefore
\begin{align*}
\clD{i} \cap \clChar{j} \cap \clParity{a}
&
= (\clD{i} \cap \clChar{j} \cap \clCon{0} \cap \clParity{a}) \cup (\clD{i} \cap \clChar{j} \cap \clCon{1} \cap \clParity{a})
\\ &
\subseteq \gen{f, g, h, h'}
\subseteq \clD{i} \cap \clChar{j} \cap \clParity{a}.
\qedhere
\end{align*}
\end{proof}

\begin{proposition}
\label{prop:gen:C}
Let $C \in \{\clCon{0}, \clCon{1}, \clYksi{0}, \clYksi{1}, \clEven, \clOdd\}$, and let
\[
(K_1, K_2) :=
\begin{cases}
(\clYksi{0}, \clYksi{1}), & \text{if $C \in \{\clCon{0}, \clCon{1}\}$,} \\
(\clCon{0}, \clCon{1}), & \text{if $C \in \{\clYksi{0}, \clYksi{1}, \clEven, \clOdd\}$.}
\end{cases}
\]
\begin{enumerate}[label=\upshape{(\roman*)}, leftmargin=*, widest=iii]
\item\label{gen:X1C}
For any $f_i \in (\clChar{1} \cap C) \setminus \clD{i}$ \textup{(}$i \in \IN_{+}$\textup{)}, $h_1 \in (\clChar{1} \cap C) \setminus K_1$, $h_2 \in (\clChar{1} \cap C) \setminus K_2$, we have $\gen{\{ \, f_i \mid i \in \IN_{+} \, \} \cup \{h_1, h_2\}} = \clChar{1} \cap C$.
\item\label{gen:XkC}
Let $k \in \IN_{+}$ with $k \geq 2$. For any $f_i \in (\clChar{k} \cap C) \setminus \clD{i}$ \textup{(}$i \in \IN_{+}$\textup{)}, $g \in (\clChar{k} \cap C) \setminus \clChar{k-1}$, $h_1 \in (\clChar{1} \cap C) \setminus K_1$, $h_2 \in (\clChar{1} \cap C) \setminus K_2$, we have $\gen{\{ \, f_i \mid i \in \IN_{+} \, \} \cup \{g, h_1, h_2\}} = \clChar{k} \cap C$.
\item\label{gen:C}
For $g_i \in C \setminus \clChar{i}$ \textup{(}$i \in \IN_{+}$\textup{)}, $h_1 \in C \setminus K_1$, $h_2 \in C \setminus K_2$, we have $\langle \{\, g_i \mid i \in \IN_{+} \,\} \cup \linebreak \{h_1, h_2\} \rangle_{\clLc} = C$.
\end{enumerate}
\end{proposition}

\begin{proof}
\ref{gen:X1C}
For $i \in \IN_{+}$, let $f_i \in (\clChar{1} \cap C) \setminus \clD{i}$, $h_1 \in (\clChar{1} \cap C) \setminus K_1$, $h_2 \in (\clChar{1} \cap C) \setminus K_2$
and let $n_i := \deg(f_i)$; we have $n_i > i$.
By identifying all arguments of $h_1$ and $h_2$, we get minors $\eta_1 \in (\clD{1} \cap \clChar{1} \cap C) \setminus K_1$, $\eta_2 \in (\clD{1} \cap \clChar{1} \cap C) \setminus K_2$.
Since $f_i \in (\clD{n_i} \cap \clChar{1} \cap C) \setminus \clD{n_i - 1}$, it follows from Propositions~\ref{prop:gen:DiXjCa}\ref{gen:DkX1Ca}, \ref{prop:gen:DiXjEa}\ref{gen:DkX1Ea}, and \ref{prop:gen:DiXjPa}\ref{gen:DkX1Pa} that $\gen{f_i, \eta_1, \eta_2} = \clD{n_i} \cap \clChar{1} \cap C$ for any $i \in \IN_{+}$.
Therefore
\begin{align*}
\clChar{1} \cap C
&
= \bigcup_{i \in \IN_{+}} (\clD{i} \cap \clChar{1} \cap C)
\subseteq \bigcup_{i \in \IN_{+}} (\clD{n_i} \cap \clChar{1} \cap C)
\\ &
= \bigcup_{i \in \IN_{+}} \gen{f_i, \eta_1, \eta_2}
\subseteq \gen{\{ \, f_i \mid i \in \IN_{+} \, \} \cup \{h_1, h_2\}}
\subseteq \clChar{1} \cap C.
\end{align*}

\ref{gen:XkC}
For $i \in \IN_{+}$, let $f_i \in (\clChar{k} \cap C) \setminus \clD{i}$, $g \in (\clChar{k} \cap C) \setminus \clChar{k-1}$, $h_1 \in (\clChar{1} \cap C) \setminus K_1$, $h_2 \in (\clChar{1} \cap C) \setminus K_2$,
and let $n_i := \deg(f_i)$; we have $n_i > i$.
By Lemma~\ref{lem:arity-degree}\ref{lem:arity-degree:minor1}, $g$ has a $k$\hyp{}ary minor $\gamma$ of degree $k$ such that $\gamma \in (\clD{k} \cap \clChar{k} \cap C) \setminus \clChar{k - 1}$.
By identifying all arguments of $h_1$ and $h_2$, we get minors $\eta_1 \in (\clD{1} \cap \clChar{k} \cap C) \setminus K_1$, $\eta_2 \in (\clD{1} \cap \clChar{k} \cap C) \setminus K_2$.
By Propositions~\ref{prop:gen:DiXjCa}\ref{gen:DiXjCa}, \ref{prop:gen:DiXjEa}\ref{gen:DiXjEa}, and \ref{prop:gen:DiXjPa}\ref{gen:DiXjPa} it holds that $\gen{f_i, g, \eta_1, \eta_2} = \clD{n_i} \cap \clChar{k} \cap C$ whenever $n_i \geq k$ (this certainly holds whenever $i \geq k$).
Therefore
\begin{align*}
\clChar{k} \cap C
&
= \bigcup_{i \in \IN_{+}} (\clD{i} \cap \clChar{k} \cap C)
= \bigcup_{i \geq k} (\clD{i} \cap \clChar{k} \cap C)
\subseteq \bigcup_{i \geq k} (\clD{n_i} \cap \clChar{k} \cap C)
\\ &
= \bigcup_{i \geq k} \gen{f_i, g, \eta_1, \eta_2}
\subseteq \gen{\{ \, f_i \mid i \in \IN_{+} \, \} \cup \{g, h_1, h_2\}}
\subseteq \clChar{k} \cap C.
\end{align*}

\ref{gen:C}
For $i \in \IN_{+}$, let $g_i \in C \setminus \clChar{i}$, $h_1 \in C \setminus K_1$, $h_2 \in C \setminus K_2$,
and let $k_i := \charrank{g_i}$.
Then $g_i \in (\clChar{k_i} \cap C) \setminus \clChar{k_i - 1}$.
By Lemma~\ref{lem:arity-degree}\ref{lem:arity-degree:minor1}, $g_i$ has a $k_i$\hyp{}ary minor $\gamma_i$ of degree $k_i$ such that $\gamma_i \in (\clD{k_i} \cap C) \setminus \clChar{k_i - 1}$.
By identifying all arguments of $h_1$ and $h_2$, we get the minors $\eta_1 \in (\clD{1} \cap C) \setminus K_1$, $\eta_2 \in (\clD{1} \cap C) \setminus K_2$.
By Propositions~\ref{prop:gen:DiXjCa}\ref{gen:DkCa}, \ref{prop:gen:DiXjEa}\ref{gen:DkEa}, and \ref{prop:gen:DiXjPa}\ref{gen:DkPa} it holds that $\gen{\gamma_i, \eta_1, \eta_2} = \clD{k_i} \cap C$ for any $i \in \IN_{+}$.
Therefore
\begin{align*}
C
&
= \bigcup_{i \in \IN_{+}} (\clD{i} \cap C)
\subseteq \bigcup_{i \in \IN_{+}} (\clD{k_i} \cap C)
= \bigcup_{i \in \IN_{+}} \gen{\gamma_i, \eta_1, \eta_2}
\\ &
\subseteq \bigcup_{i \in \IN_{+}} \gen{g_i, h_1, h_2}
\subseteq \gen{\{ \, g_i \mid i \in \IN_{+} \, \} \cup \{h_1, h_2\}}
\subseteq C.
\qedhere
\end{align*}
\end{proof}

\begin{lemma}
\label{lem:h1...h6}
For any $h_1 \in \clCon{0}$, $h_2 \in \clCon{1}$, $h_3 \in \clYksi{0}$, $h_4 \in \clYksi{1}$, $h_5 \in \clEven$, $h_6 \in \clOdd$, 
we have $\clD{1} \subseteq \gen{h_1, h_2, h_3, h_4, h_5, h_6}$.
\end{lemma}

\begin{proof}
By identifying all arguments, we see that
\begin{equation}
\begin{array}{@{\text{either} \quad}c@{\quad \text{or} \quad}c@{\quad \text{is in} \quad}c}
0 & x_1 & \gen{h_1}, \\
1 & x_1 + 1 & \gen{h_2}, \\
0 & x_1 + 1 & \gen{h_3}, \\
1 & x_1 & \gen{h_4}, \\
0 & 1 & \gen{h_5}, \\
x_1 & x_1 + 1 & \gen{h_6}. \\
\end{array}
\label{eq:h1...h6}
\end{equation}
Let $G := \{ 0, 1, x_1, x_1 + 1 \}$.
Clearly $\gen{G} = \clD{1}$.
Any three\hyp{}element subset of $G$ also generates $\clD{1}$ because each element of $G$ is the sum of the other three elements.
Any choice of functions from the six pairs in \eqref{eq:h1...h6} includes at least three different elements of $G$, so we conclude that $\clD{1} \subseteq \gen{h_1, h_2, h_3, h_4, h_5, h_6}$.
\end{proof}

\begin{proposition}
\label{prop:gen:DiXj}
\leavevmode
\begin{enumerate}[label=\upshape{(\roman*)}, leftmargin=*, widest=iii]
\item\label{gen:DkX1}
Let $k \in \IN_{+}$. For any $f, h_1, h_2, h_3, h_4, h_5, h_6 \in \clD{k} \cap \clChar{1}$ with $f \notin \clD{k-1}$, $h_1 \notin \clCon{0}$, $h_2 \notin \clCon{1}$, $h_3 \notin \clYksi{0}$, $h_4 \notin \clYksi{1}$, $h_5 \notin \clEven$, $h_6 \notin \clOdd$, we have $\gen{f, h_1, h_2, h_3, h_4, h_5, h_6} = \clD{k} \cap \clChar{1}$.
\item\label{gen:Dk}
Let $k \in \IN_{+}$ with $k \geq 2$. For any $g, h_1, h_2, h_3, h_4, h_5, h_6 \in \clD{k}$ with $g \notin \clChar{k-1}$, $h_1 \notin \clCon{0}$, $h_2 \notin \clCon{1}$, $h_3 \notin \clYksi{0}$, $h_4 \notin \clYksi{1}$, $h_5 \notin \clEven$, $h_6 \notin \clOdd$, we have $\gen{g, h_1, h_2, h_3, h_4, h_5, h_6} = \clD{k}$.
\item\label{gen:DiXj}
Let $i, j \in \IN$ with $i > j \geq 2$. For any functions $f, g, h_1, h_2, h_3, h_4, h_5, h_6 \in \clD{i} \cap \clChar{j}$ with $f \notin \clD{i-1}$, $g \notin \clChar{j-1}$, $h_1 \notin \clCon{0}$, $h_2 \notin \clCon{1}$, $h_3 \notin \clYksi{0}$, $h_4 \notin \clYksi{1}$, $h_5 \notin \clEven$, $h_6 \notin \clOdd$, we have $\gen{f, g, h_1, h_2, h_3, h_4, h_5, h_6} = \clD{i} \cap \clChar{j}$.
\end{enumerate}
\end{proposition}

\begin{proof}
We consider only statement \ref{gen:DiXj}.
The proofs of statements \ref{gen:DkX1} and \ref{gen:Dk} are analogous; we just need to omit the parts of the proof that deal with the function $f$ or $g$, as the case may be, that does not appear in the statement.

Since $\{ \clEven, \clOdd \}$, $\{ \clCon{0}, \clCon{1} \}$, and $\{ \clYksi{0}, \clYksi{1} \}$ are partitions of $\clAll$, we have that $h_1 \in \clCon{1}$, $h_2 \in \clCon{0}$, $h_3 \in \clYksi{1}$, $h_4 \in \clYksi{0}$, $h_5 \in \clOdd$, and $h_6 \in \clEven$.
By Lemma~\ref{lem:h1...h6}, we have $\clD{1} \subseteq \gen{h_1, h_2, h_3, h_4, h_5, h_6}$.
Hence $f + x_1 + 1 \in \gen{f, h_1, h_2, h_3, h_4, h_5, h_6}$ and $g + x_1 + 1 \in \gen{g, h_1, h_2, h_3, h_4, h_5, h_6}$.
One of $f$ and $f + x_1 + 1$ belongs to $(\clD{i} \cap \clChar{j} \cap \clEven) \setminus \clD{i-1}$ and the other to $(\clD{i} \cap \clChar{j} \cap \clOdd) \setminus \clD{i-1}$,
and, similarly,
one of $g$ and $g + x_1 + 1$ belongs to $(\clD{i} \cap \clChar{j} \cap \clEven) \setminus \clChar{j-1}$ and the other to $(\clD{i} \cap \clChar{j} \cap \clOdd) \setminus \clChar{j-1}$.
Proposition~\ref{prop:gen:DiXjPa}\ref{gen:DiXjPa} implies that $\gen{f, g, h_1, h_2, h_3, h_4, h_5, h_6}$ contains a generating set for both $\clD{i} \cap \clChar{j} \cap \clEven$ and $\clD{i} \cap \clChar{j} \cap \clOdd$.
Therefore
\[
\clD{i} \cap \clChar{j}
= (\clD{i} \cap \clChar{j} \cap \clEven) \cup (\clD{i} \cap \clChar{j} \cap \clOdd)
\subseteq \gen{f, g, h_1, h_2, h_3, h_4, h_5, h_6}
\subseteq \clD{i} \cap \clChar{j}.
\qedhere
\]
\end{proof}

\begin{proposition}
\label{prop:gen:Xk}
\leavevmode
\begin{enumerate}[label=\upshape{(\roman*)}, leftmargin=*, widest=iii]
\item\label{gen:X1}
For any $f_i \in \clChar{1} \setminus \clD{i}$ \textup{(}$i \in \IN_{+}$\textup{)} and $h_1, h_2, h_3, h_4, h_5, h_6 \in \clChar{1}$ with $h_1 \notin \clCon{0}$, $h_2 \notin \clCon{1}$, $h_3 \notin \clYksi{0}$, $h_4 \notin \clYksi{1}$, $h_5 \notin \clEven$, $h_6 \notin \clOdd$, we have $\langle \{ \, f_i \mid i \in \IN_{+} \, \} \cup \{h_1, h_2, \linebreak h_3, h_4, h_5, h_6\} \rangle_{\clLc} = \clChar{1}$.
\item\label{gen:Xk}
Let $k \in \IN_{+}$, $k \geq 2$. For any $f_i \in \clChar{k} \setminus \clD{i}$ \textup{(}$i \in \IN_{+}$\textup{)}, $g \in \clChar{k} \setminus \clChar{k-1}$, and $h_1, h_2, h_3, h_4, h_5, h_6 \in \clChar{k}$ such that $h_1 \notin \clCon{0}$, $h_2 \notin \clCon{1}$, $h_3 \notin \clYksi{0}$, $h_4 \notin \clYksi{1}$, $h_5 \notin \clEven$, $h_6 \notin \clOdd$, we have $\gen{\{ \, f_i \mid i \in \IN_{+} \, \} \cup \{g, h_1, h_2, h_3, h_4, h_5, h_6\}} = \clChar{k}$.
\item\label{gen:All}
For any $g_i \in \clAll \setminus \clChar{i}$ \textup{(}$i \in \IN_{+}$\textup{)} and $h_1, h_2, h_3, h_4, h_5, h_6 \in \clAll$ with $h_1 \notin \clCon{0}$, $h_2 \notin \clCon{1}$, $h_3 \notin \clYksi{0}$, $h_4 \notin \clYksi{1}$, $h_5 \notin \clEven$, $h_6 \notin \clOdd$, we have $\langle \{\, g_i \mid i \in \IN_{+} \,\} \cup \{h_1, h_2, \linebreak h_3, h_4, h_5, h_6\} \rangle = \clAll$.
\end{enumerate}
\end{proposition}

\begin{proof}
Observe first that $0, 1, x_1 \in \clD{1} \subseteq \clChar{k}$ for any $k \in \IN_{+}$ and $1 \notin \clCon{0}$, $0 \notin \clCon{1}$, $1 \notin \clYksi{0}$, $0 \notin \clYksi{1}$, $x_1 \notin \clEven$, $0 \notin \clOdd$.

\ref{gen:X1}
For $i \in \IN_{+}$, let $f_i \in \clChar{1} \setminus \clD{i}$ and $h_1, h_2, h_3, h_4, h_5, h_6 \in \clChar{1}$ be such that $h_1 \notin \clCon{0}$, $h_2 \notin \clCon{1}$, $h_3 \notin \clYksi{0}$, $h_4 \notin \clYksi{1}$, $h_5 \notin \clEven$, $h_6 \notin \clOdd$,
and let $n_i := \deg(f_i)$; we have $n_i > i$.
Since $f_i \in (\clD{n_i} \cap \clChar{1}) \setminus \clD{n_i - 1}$, Proposition~\ref{prop:gen:DiXj}\ref{gen:DkX1} implies $\gen{f_i, 0, 1, x_1} = \clD{n_i} \cap \clChar{1}$ for any $i \in \IN_{+}$.
We have $\{ 0, 1, x_1 \} \subseteq \clD{1} \subseteq \gen{h_1, h_2, h_3, h_4, h_5, h_6}$ by Lemma~\ref{lem:h1...h6}.
Therefore
\begin{align*}
\clChar{1}
&
= \bigcup_{i \in \IN_{+}} (\clD{i} \cap \clChar{1})
\subseteq \bigcup_{i \in \IN_{+}} (\clD{n_i} \cap \clChar{1})
= \bigcup_{i \in \IN_{+}} \gen{f_i, 0, 1, x_1}
\\ &
\subseteq \gen{\{ \, f_i \mid i \in \IN_{+} \, \} \cup \{h_1, h_2, h_3, h_4, h_5, h_6\}}
\subseteq \clChar{1}.
\end{align*}

\ref{gen:Xk}
For $i \in \IN_{+}$, let $f_i \in \clChar{k} \setminus \clD{i}$, $g \in \clChar{k} \setminus \clChar{k-1}$, and $h_1, h_2, h_3, h_4, h_5, h_6 \in \clChar{k}$ such that $h_1 \notin \clCon{0}$, $h_2 \notin \clCon{1}$, $h_3 \notin \clYksi{0}$, $h_4 \notin \clYksi{1}$, $h_5 \notin \clEven$, $h_6 \notin \clOdd$,
and let $n_i := \deg(f_i)$; we have $n_i > i$.
By Lemma~\ref{lem:arity-degree}\ref{lem:arity-degree:minor1}, $g$ has a $k$\hyp{}ary minor $\gamma$ of degree $k$ such that $\gamma \in \clChar{k} \setminus \clChar{k - 1}$; hence $\gamma \in \clD{k} \setminus \clChar{k - 1}$.
By Proposition~\ref{prop:gen:DiXj}\ref{gen:DiXj}, it holds that $\gen{f_i, \gamma, 0, 1, x_1} = \clD{n_i} \cap \clChar{k}$ whenever $n_i \geq k$.
We have $\{ 0, 1, x_1 \} \subseteq \clD{1} \subseteq \gen{h_1, h_2, h_3, h_4, h_5, h_6}$ by Lemma~\ref{lem:h1...h6}.
Therefore
\begin{align*}
\clChar{k}
&
= \bigcup_{i \in \IN_{+}} (\clD{i} \cap \clChar{k})
= \bigcup_{i \geq k} (\clD{i} \cap \clChar{k})
\subseteq \bigcup_{i \geq k} (\clD{n_i} \cap \clChar{k})
= \bigcup_{i \geq k} \gen{f_i, \gamma, 0, 1, x_1}
\\ &
\subseteq \gen{\{ \, f_i \mid i \in \IN_{+} \, \} \cup \{g, h_1, h_2, h_3, h_4, h_5, h_6\}}
\subseteq \clChar{k}.
\end{align*}

\ref{gen:All}
For $i \in \IN_{+}$, let $g_i \in \clAll \setminus \clChar{i}$, and $h_1, h_2, h_3, h_4, h_5, h_6 \in \clChar{k}$ such that $h_1 \notin \clCon{0}$, $h_2 \notin \clCon{1}$, $h_3 \notin \clYksi{0}$, $h_4 \notin \clYksi{1}$, $h_5 \notin \clEven$, $h_6 \notin \clOdd$,
and let $k_i := \charrank{g_i}$.
Then $g_i \in \clChar{k_i} \setminus \clChar{k_i - 1}$.
By Lemma~\ref{lem:arity-degree}\ref{lem:arity-degree:minor1}, $g_i$ has a $k_i$\hyp{}ary minor $\gamma_i$ of degree $k_i$ such that $\gamma_i \in \clD{k_i} \setminus \clChar{k_i - 1}$.
By Proposition~\ref{prop:gen:DiXj}\ref{gen:Dk}, it holds that $\gen{\gamma_i, 0, 1, x_1} = \clD{k_i}$ for any $i \in \IN_{+}$.
We have $\{ 0, 1, x_1 \} \subseteq \clD{1} \subseteq \gen{h_1, h_2, h_3, h_4, h_5, h_6}$ by Lemma~\ref{lem:h1...h6}.
Therefore
\begin{align*}
\clAll
&
= \bigcup_{i \in \IN_{+}} \clD{i}
\subseteq \bigcup_{i \in \IN_{+}} \clD{k_i}
= \bigcup_{i \in \IN_{+}} \gen{\gamma_i, 0, 1, x_1}
\\ &
\subseteq \gen{\{ \, g_i \mid i \in \IN_{+} \, \} \cup \{h_1, h_2, h_3, h_4, h_5, h_6\}}
\subseteq \clAll.
\qedhere
\end{align*}
\end{proof}

\begin{proof}[Proof of Theorem~\ref{thm:Lc}]
By Lemma~\ref{lem:Lc-simplify}\ref{lem:Lc-simplify:Lc}, $\clLc$\hyp{}stability is equivalent to $(\clIc, \clLc)$\hyp{}stability.
The given classes are $\clLc$\hyp{}stable by Proposition~\ref{prop:Lc-sufficiency}.
The fact that there are no further $\clLc$\hyp{}stable classes distinct from these follows from
Propositions~\ref{prop:gen:D0}, \ref{prop:gen:DiXjC0E0}, \ref{prop:gen:DiXjC0E1}, \ref{prop:gen:CaEb}, \ref{prop:gen:DiXjCa}, \ref{prop:gen:DiXjEa}, \ref{prop:gen:DiXjPa}, \ref{prop:gen:C}, \ref{prop:gen:DiXj}, \ref{prop:gen:Xk}, in which we have shown that any set of Boolean functions generates one of the classes listed in the statement -- more precisely, that for each class $C$ and for any set $F \subseteq C$ that is not included in any proper subclass of $C$ it holds that $\gen{F} = C$.
\end{proof}

%%%%%%%%%%%%%%%%%%%%%%%%%%%%%%%%%%%%%%%%%%%%%%%%%%

\section{$(C_1,C_2)$\hyp{}stable classes for $\clLc \subseteq C_2$}
\label{sec:C1C2}

Theorem~\ref{thm:Lc} allows us to describe also all $(C_1,C_2)$\hyp{}stable classes of Boolean functions for clones $C_1$ and $C_2$ such that $C_1$ is arbitrary and $\clLc \subseteq C_2$.
Namely, by Lemma~\ref{lem:Lc-simplify}, $\clLc$\hyp{}stability is equivalent to $(\clIc, \clLc)$\hyp{}stability.
Since $(C_1,C_2)$\hyp{}stability implies $(\clIc, \clLc)$\hyp{}stability whenever $\clLc \subseteq C_2$, it suffices to search for $(C_1,C_2)$\hyp{}stable classes among the $\clLc$\hyp{}stable ones.
To this end, we determine, for each $(\clIc, \clLc)$\hyp{}stable class $K$, the clones $C_1$ and $C_2$ for which it holds that $K C_1 \subseteq K$ and $C_2 K \subseteq K$.
The results are summarized in the following theorem which refers to Table~\ref{table:stability}.

\begin{table}
\begin{tabular}{clccl}
\toprule
    & & $K C \subseteq K$ & $C K \subseteq K$ & \\
$K$ & & if and only if & if and only if & \multicolumn{1}{c}{result} \\
    & & $C \subseteq \ldots$ & $C \subseteq \ldots$ & \\
\midrule
$\clAll$ & & $\clAll$ & $\clAll$ & Proposition~\ref{prop:all-empty} \\
$\clCon{a}$ & & $\clTo$ & $\clTa{a}$ & Proposition~\ref{prop:Ca} \\
$\clYksi{a}$ & & $\clTi$ & $\clTa{a}$ & Proposition~\ref{prop:Ca} \\
$\clEven$ & & $\clTc$ & $\clAll$ & Proposition~\ref{prop:odd-even} \\
$\clOdd$ & & $\clTc$ & $\clS$ & Proposition~\ref{prop:odd-even} \\
$\clBoth{a}{b}$ & & $\clTc$ & $\clTa{a} \cap \clTa{b}$ & Proposition~\ref{prop:CaEb} \\
\midrule
$\clChar{k}$ & $k \geq 2$ & $\clLS$ & $\clL$ & Proposition~\ref{prop:Xk} \\
& $k = 1$ & $\clS$ & $\clL$ & \\
$\clChar{k} \cap \clCon{a}$ & $k \geq 2$ & $\clLc$ & $\clLa{a}$ & Proposition~\ref{prop:XkCa} \\
& $k = 1$ & $\clSc$ & $\clLa{a}$ & \\
$\clChar{k} \cap \clYksi{a}$ & $k \geq 2$ & $\clLc$ & $\clLa{a}$ & Proposition~\ref{prop:XkEa} \\
& $k = 1$ & $\clSc$ & $\clLa{a}$ & \\
$\clChar{k} \cap \clEven$ & $k \geq 2$ & $\clLc$ & $\clL$ & Proposition~\ref{prop:XkP0} \\
& $k = 1$ & $\clS$ & $\clAll$ & \\
$\clChar{k} \cap \clOdd$ & $k \geq 2$ & $\clLc$ & $\clLS$ & Proposition~\ref{prop:XkP1} \\
& $k = 1$ & $\clS$ & $\clS$ & \\
$\clChar{k} \cap \clBoth{a}{b}$ & $k \geq 2$ & $\clLc$ & $\clLa{a} \cap \clLa{b}$ & Proposition~\ref{prop:XkCaEb} \\
& $k = 1$, $a = b$ & $\clSc$ & $\clTa{a}$ & \\
& $k = 1$, $a \neq b$ & $\clSc$ & $\clSc$ & \\
\midrule
$\clD{k}$ & & $\clL$ & $\clL$ & Proposition~\ref{prop:Dk} \\
$\clD{k} \cap \clCon{a}$ & & $\clLo$ & $\clLa{a}$ & Proposition~\ref{prop:DkCa} \\
$\clD{k} \cap \clYksi{a}$ & & $\clLi$ & $\clLa{a}$ & Proposition~\ref{prop:DkCa} \\
$\clD{k} \cap \clEven$ & $k \geq 2$ & $\clLc$ & $\clL$ & Proposition~\ref{prop:DkP0} \\
& $k = 1$ & $\clLS$ & $\clL$ & \\
$\clD{k} \cap \clOdd$ & $k \geq 2$ & $\clLc$ & $\clLS$ & Proposition~\ref{prop:DkP1} \\
& $k = 1$ & $\clLS$ & $\clLS$ & \\
$\clD{k} \cap \clBoth{a}{b}$ & & $\clLc$ & $\clLa{a} \cap \clLa{b}$ & Proposition~\ref{prop:DkCaEb} \\
\midrule
$\clD{i} \cap \clChar{j}$ & & $\clLS$ & $\clL$ & Proposition~\ref{prop:DiXj} \\
$\clD{i} \cap \clChar{j} \cap \clCon{a}$ & & $\clLc$ & $\clLa{a}$ & Proposition~\ref{prop:DiXjCa} \\
$\clD{i} \cap \clChar{j} \cap \clYksi{a}$ & & $\clLc$ & $\clLa{a}$ & Proposition~\ref{prop:DiXjCa} \\
$\clD{i} \cap \clChar{j} \cap \clEven$ & $j \geq 2$ & $\clLc$ & $\clL$ & Proposition~\ref{prop:DiXjP0} \\
& $j = 1$ & $\clLS$ & $\clL$ & \\
$\clD{i} \cap \clChar{j} \cap \clOdd$ & $j \geq 2$ & $\clLc$ & $\clLS$ & Proposition~\ref{prop:DiXjP1} \\
& $j = 1$ & $\clLS$ & $\clLS$ & \\
$\clD{i} \cap \clChar{j} \cap \clBoth{a}{b}$ & & $\clLc$ & $\clLa{a} \cap \clLa{b}$ & Proposition~\ref{prop:DiXjCaEb} \\
\midrule
$\clD{0}$ & & $\clAll$ & $\clAll$ & Proposition~\ref{prop:constants} \\
$\clD{0} \cap \clCon{a}$ & & $\clAll$ & $\clTa{a}$ & Proposition~\ref{prop:constants} \\
$\clEmpty$ & & $\clAll$ & $\clAll$ & Proposition~\ref{prop:all-empty} \\
\bottomrule
\end{tabular}

\bigskip
\caption{The $\clLc$\hyp{}stable classes $K$ and their stability under right and left compositions with clones $C$. Parameters: $a, b \in \{0,1\}$, $i, j, k \in \IN$ with $k \geq 1$, $i > j \geq 1$.}
\label{table:stability}
\end{table}

\begin{theorem}
\label{thm:C1C2-stability}
For each $\clLc$\hyp{}stable class $K$, as determined in Theorem~\ref{thm:Lc}, there exist clones $C_1^K$ and $C_2^K$, as prescribed in Table~\ref{table:stability}, such that for every clone $C$, it holds that $K C \subseteq K$ if and only if $C \subseteq C_1^K$, and $C K \subseteq K$ if and only if $C \subseteq C_2^K$.
\end{theorem}

The proof of Theorem~\ref{thm:C1C2-stability} will be developed in the remainder of this section.
The following two lemmata will be useful.
The first one (Lemma~\ref{lem:suff-int}) provides sufficient conditions for right and left stability for classes that are intersections of classes for which we already know sufficient conditions for right and left stability.
The second one (Lemma~\ref{lem:non}) provides necessary conditions.
These will be applied in the subsequent propositions in which necessary and sufficient stability conditions are established for each $\clLc$\hyp{}stable class.

\begin{lemma}
\label{lem:suff-int}
Let $K_1, K_2, C_1, C_2 \subseteq \clAll$.
Then the following statements hold.
\begin{enumerate}[label=\upshape{(\roman*)}, leftmargin=*, widest=ii]
\item\label{lem:suff-int:right}
Assume $K_1 C \subseteq K_1$ whenever $C \subseteq C_1$ and $K_2 C \subseteq K_2$ whenever $C \subseteq C_2$.
Then $(K_1 \cap K_2) C \subseteq K_1 \cap K_2$ whenever $C \subseteq C_1 \cap C_2$.
\item\label{lem:suff-int:left}
Assume $C K_1 \subseteq K_1$ whenever $C \subseteq C_1$ and $C K_2 \subseteq K_2$ whenever $C \subseteq C_2$.
Then $C (K_1 \cap K_2) \subseteq K_1 \cap K_2$ whenever $C \subseteq C_1 \cap C_2$.
\end{enumerate}
\end{lemma}

\begin{proof}
\ref{lem:suff-int:right}
If $C \subseteq C_1 \cap C_2$, then
$(K_1 \cap K_2) C \subseteq K_1 C_1 \subseteq K_1$
and
$(K_1 \cap K_2) C \subseteq K_2 C_2 \subseteq K_2$
by the monotonicity of function class composition and the stability of $K_1$ and $K_2$ under right composition with $C_1$ and $C_2$, respectively.
Therefore $(K_1 \cap K_2) C \subseteq K_1 \cap K_2$.

\ref{lem:suff-int:left}
The proof is analogous to that of part \ref{lem:suff-int:right}.
\end{proof}

\begin{lemma}
\label{lem:non}
Let $a, b \in \{0, 1\}$, $\QuantifyParRel$, $i, j \in \IN_{+}$ with $i \geq j \geq 1$.
\begin{enumerate}[label={\upshape{(\roman*)}}, leftmargin=*, widest=iii]
\item\label{lem:non:any}
For any $\clEmpty \neq K \subseteq \clAll$, the following statements hold.
\begin{enumerate}[label={\upshape{(\alph*)}}, leftmargin=*, widest=m]
\item\label{lem:non:Inota-K-CaEb}
$\clIa{\overline{a}} K \nsubseteq \clCon{a} \cup \clYksi{a}$.
\item\label{lem:non:Ia-K-P1}
$\clIa{a} K \nsubseteq \clOdd$.
\item\label{lem:non:Ia-K-CaEb}
If $a \neq b$, then
$\clIo K \nsubseteq \clBoth{a}{b}$,
$\clIi K \nsubseteq \clBoth{a}{b}$.
\end{enumerate}

\item\label{lem:non:DiXjCaEb}
For $K := \clD{i} \cap \clChar{j} \cap \clBoth{a}{b}$, the following statements hold.
\begin{enumerate}[label={\upshape{(\alph*)}},resume, leftmargin=*, widest=m]
\item\label{lem:non:Istar-K-CaEb}
$\clIstar K \nsubseteq \clCon{a} \cup \clYksi{b}$.
\item\label{lem:non:LV-K-Xj}
$\clLambdac K \nsubseteq \clD{i}$,
$\clVc K \nsubseteq \clD{i}$.
If $j \geq 2$ or $a \neq b$, then
$\clLambdac K \nsubseteq \clChar{j}$,
$\clVc K \nsubseteq \clChar{j}$.
\item\label{lem:non:SM-K}
$\clSM K \nsubseteq \clD{i}$.
If $j \geq 2$, then
$\clSM K \nsubseteq \clChar{j}$.
\item\label{lem:non:K-Ia-CaEb}
$K \clIo \nsubseteq \clYksi{b}$,
$K \clIi \nsubseteq \clCon{a}$,
$K \clIstar \nsubseteq \clCon{a} \cup \clYksi{b}$.
\item\label{lem:non:K-Ia-Xj}
If $i > j$, then
$K \clIo \nsubseteq \clChar{j}$,
$K \clIi \nsubseteq \clChar{j}$.
\item\label{lem:non:K-LV}
$K \clLambdac \nsubseteq \clD{i} \cup \clChar{j}$,
$K \clVc \nsubseteq \clD{i} \cup \clChar{j}$.
\item\label{lem:non:K-SM}
$K \clSM \nsubseteq \clD{i}$.
If $j \geq 2$, then
$K \clSM \nsubseteq \clChar{j}$.
\end{enumerate}

\item\label{lem:non:X1Ca-X1Ea}
\begin{enumerate}[label={\upshape{(\alph*)}},resume, leftmargin=*, widest=m]
\item\label{lem:non:SM-X1Ca} $\clSM (\clChar{1} \cap \clCon{a}) \nsubseteq \clChar{1}$, $\clSM (\clChar{1} \cap \clYksi{a}) \nsubseteq \clChar{1}$.
\end{enumerate}

\item\label{lem:non:DiXjPa}
For $K := \clD{i} \cap \clChar{j} \cap \clParity{a}$, the following statements hold.
\begin{enumerate}[label={\upshape{(\alph*)}},resume, leftmargin=*, widest=m]
\item\label{lem:non:K-Ia-Pa}
$K \clIo \nsubseteq \clParity{a}$,
$K \clIi \nsubseteq \clParity{a}$.
\item\label{lem:non:K-Istar-Pa}
If $j \geq 2$, then
$K \clIstar \nsubseteq \clParity{a}$.
\item\label{lem:non:LV-K-Pa}
If $\mathord{\ParRel} = \mathord{\neq}$, then
$\clLambdac K \nsubseteq \clParity{a}$,
$\clVc K \nsubseteq \clParity{a}$.
\end{enumerate}
\end{enumerate}
\end{lemma}

\begin{proof}
Throughout the proof, we will use Lemmata~\ref{lem:left-stab-gen} and \ref{lem:right-stab-gen} together with the fact that $\clIo = \clonegen{0}$, $\clIi = \clonegen{1}$, $\clIstar = \clonegen{x_1 + 1}$, $\clLambdac = \clonegen{\mathord{\wedge}}$, $\clVc = \clonegen{\mathord{\vee}}$, $\clSM = \clonegen{\mu}$.

\ref{lem:non:any}
\ref{lem:non:Inota-K-CaEb}
For any $\varphi \in \clAll$, we have $\overline{a}(\varphi) = \overline{a} \notin \clCon{a} \cup \clYksi{a}$.
Therefore $\clIa{\overline{a}} K \nsubseteq \clCon{a} \cup \clYksi{a}$.

\ref{lem:non:Ia-K-P1}
For any $\varphi \in \clAll$, we have $a(\varphi) = a \notin \clOdd$.
Therefore $\clIa{a} K \nsubseteq \clOdd$.

\ref{lem:non:Ia-K-CaEb}
If $a \neq b$, then, by \ref{lem:non:Inota-K-CaEb}, we have $\clIo K \nsubseteq \clCon{1} \cup \clYksi{1}$ and $\clIi K \nsubseteq \clCon{0} \cup \clYksi{0}$.
Since $\clYksi{a}{b}$ is a subset of both $\clCon{0} \cup \clYksi{0}$ and $\clCon{1} \cup \clYksi{1}$, it follows that $\clIa{i} K \nsubseteq \clBoth{a}{b}$ for $i \in \{0,1\}$.

\ref{lem:non:DiXjCaEb}
Let
\begin{align*}
f_0 &:= x_1 + x_2 + a, & g_0 &:= \monster{i} + a, & h_0 &:= x_1 \dots x_j + x_{j+1} + a, \\
f_1 &:= x_1 + a, & g_1 &:= \monster{i} + x_{i+1} + a, & h_1 &:= x_1 \dots x_j + a,
\end{align*}
and note that $f_0, g_0, h_0 \in \clD{i} \cap \clChar{j} \cap \clBoth{a}{a}$ and $f_1, g_1, h_1 \in \clD{i} \cap \clChar{j} \cap \clBoth{a}{\overline{a}}$.

\ref{lem:non:Istar-K-CaEb}
For any $a, b \in \{0,1\}$ and $f \in \clCon{a} \cup \clYksi{b}$ we have $(x_1 + 1)(f) = f + 1 \notin \clCon{a} \cup \clYksi{b}$.
For any $a, b \in \{0,1\}$ there exists a function in $\clD{i} \cap \clChar{j} \cap \clBoth{a}{b}$; consider the functions $f_0$ and $f_1$ defined above.
It follows that $\clIstar K \nsubseteq \clCon{a} \cup \clYksi{b}$.

\ref{lem:non:LV-K-Xj}
The reduced polynomial of each of the functions
\begin{align*}
& \mathord{\wedge}(\monster{i} + a, x_{i+1} + x_{i+2} + a), & & \mathord{\wedge}(\monster{i} + x_{i+1} + a, x_{i+1} + a), \\
& \mathord{\vee}(\monster{i} + a, x_{i+1} + x_{i+2} + a), & & \mathord{\vee}(\monster{i} + x_{i+1} + a, x_{i+1} + a)
\end{align*}
contains the monomial $x_1 x_2 \dots x_{i+1}$ and hence has degree at least $i + 1$;
therefore none of them is an element of $\clD{i}$.
Note that the inner functions of the two compositions on the left (right, resp.)\ are minors of $f_0$ and $g_0$ ($f_1$ and $g_1$, resp.)\ and hence belong to $K$ if $a = b$ (if $a \neq b$, resp.).
This shows that $\clLambdac K \nsubseteq \clD{i}$, $\clVc K \nsubseteq \clD{i}$.

If $j \geq 2$, then
\begin{align*}
& \mathord{\wedge}(h_0, x_{j+1} + x_{j+2} + a) = x_1 \dots x_j x_{j+1} + x_1 \dots x_j x_{j+2} + \dots, \\
& \mathord{\vee}(h_0, x_{j+1} + x_{j+2} + a) = x_1 \dots x_j x_{j+1} + x_1 \dots x_j x_{j+2} + \dots, \\
& \mathord{\wedge}(h_1, x_{j+1} + a) = x_1 \dots x_j x_{j+1} + \dots, \\
& \mathord{\vee}(h_1, x_{j+1} + a) = x_1 \dots x_j x_{j+1} + \dots,
\end{align*}
where the terms that have not been written out have degree at most $j$.
The $j$\hyp{}element set $\{2, \dots, j+1\}$ has characteristic $1$ in each, so these functions are not in $\clChar{j}$.
Note that the inner functions of the first (last, resp.)\ two compositions are minors of $h_0$ and $f_0$ ($h_1$ and $f_1$, resp.)\ and hence belong to $K$ if $a = b$ (if $a \neq b$, resp.).
This shows that $\clLambdac K \nsubseteq \clChar{j}$, $\clVc K \nsubseteq \clChar{j}$ if $j \geq 2$.

If $j = 1$ and $a \neq b$, then
\begin{align*}
& \mathord{\wedge}(x_1 + a, x_2 + a) = x_1 x_2 + a x_1 + a x_2 + a \notin \clChar{1}, \\
& \mathord{\vee}(x_1 + a, x_2 + a) = x_1 x_2 + (a + 1) x_1 + (a + 1) x_2 + a \notin \clChar{1}.
\end{align*}
Note that the inner functions are minors of $f_1$ and hence belong to $K$.
This shows that $\clLambdac K \nsubseteq \clChar{j}$, $\clVc K \nsubseteq \clChar{j}$ also in this case.

\ref{lem:non:SM-K}
The reduced polynomial of each of the functions
\begin{align*}
& \mu(\monster{i} + a, x_{i+1} + x_{i+2} + a, a), & & \mu(\monster{i} + x_{i+1} + a, x_{i+1} + a, x_{i+2} + a)
\end{align*}
contains the monomial $x_1 x_2 \dots x_{i+1}$ and hence has degree at least $i + 1$;
therefore none of them is an element of $\clD{i}$.
Note that the inner functions of the first (second, resp.)\ composition are minors of $f_0$ and $g_0$ ($f_1$ and $g_1$, resp.)\ and hence belong to $K$ if $a = b$ (if $a \neq b$, resp.).
This shows that $\clSM K \nsubseteq \clD{i}$.

If $j \geq 2$, then
\begin{align*}
& \mu(h_0, x_{j+1} + x_{j+2} + a, a)
= x_1 \dots x_j x_{j+1} + x_1 \dots x_j x_{j+2} + x_{j+1} + x_{j+1} x_{j+2} + a, \\
& \mu(h_1, x_{j+1} + a, x_{j+2} + a)
= x_1 \dots x_j x_{j+1} + x_1 \dots x_j x_{j+2} + x_{j+1} x_{j+2} + a.
\end{align*}
Neither of these functions is in $\clChar{j}$, which can be seen by considering the characteristic of the $j$\hyp{}element set $\{2, \dots, j+1\}$.
Note that the inner functions of the first (second, resp.)\ composition are minors of $h_0$ and $f_0$ ($h_1$ and $f_1$, resp.)\ and hence belong to $K$ if $a = b$ (if $a \neq b$, resp.).
This shows that $\clSM K \nsubseteq \clChar{j}$ if $j \geq 2$.

\ref{lem:non:K-Ia-CaEb}
If $a = b$, then
\begin{gather*}
f_0(x_1, 0) = x_1 + a \notin \clYksi{b},
\qquad
f_0(x_1, 1) = x_1 + a + 1 \notin \clCon{a},
\\
f_0(x_1 + 1, x_2) = x_1 + x_2 + a + 1 \notin \clCon{a} \cup \clYksi{b}.
\end{gather*}
If $a \neq b$, then
\[
f_1(0) = a \notin \clYksi{b},
\qquad
f_1(1) = a + 1 \notin \clCon{a},
\qquad
f_1(x_1 + 1) = x_1 + a + 1 \notin \clCon{a} \cup \clYksi{b}.
\]
These calculations show the non\hyp{}inclusions
$K \clIo \nsubseteq \clYksi{b}$,
$K \clIi \nsubseteq \clCon{a}$,
and
$K \clIstar \nsubseteq \clCon{a} \cup \clYksi{b}$.

\ref{lem:non:K-Ia-Xj}
Assume that $i > j$.
Observe that the reduced polynomial of each one of the functions $g_0(x_1, \dots, x_i, 0)$, $g_0(x_1, \dots, x_i, 1)$, $g_1(x_1, \dots, x_i, 0, 0)$, and  $g_1(x_1, \dots, x_i, 1, \linebreak 1)$ contains the monomial $x_1 \dots x_i$, and it is the only monomial of degree $i$.
Therefore none of them is a member of $\clChar{j}$, which can be seen by considering the characteristic of the set $\nset{i-1}$ that has cardinality at least $j$.
We conclude that $K \clIo \nsubseteq \clChar{j}$ and $K \clIi \nsubseteq \clChar{j}$.

\ref{lem:non:K-LV}
For $i \in \{0,1\}$, the reduced polynomial of each of the functions
$g_i \ast \mathord{\wedge}$,
$g_i \ast \mathord{\vee}$
contains the monomial $x_1 x_2 \dots x_{i+1}$ and hence has degree at least (in fact, exactly) $i + 1$; therefore none of them is an element of $\clD{i}$.
Therefore
$K \clLambdac \nsubseteq \clD{i}$,
$K \clVc \nsubseteq \clD{i}$.

For $i \in \{0,1\}$, the reduced polynomial of each of
$h_i \ast \mathord{\wedge}$,
$h_i \ast \mathord{\vee}$
contains the monomial $x_1 \dots x_{j+1}$, and this is the only monomial of degree $j+1$.
We see that the characteristic of the $j$\hyp{}element set $\nset{j}$ is $1$ in each, so none is an element of $\clChar{j}$;
therefore
$K \clLambdac \nsubseteq \clChar{j}$,
$K \clVc \nsubseteq \clChar{j}$.

\ref{lem:non:K-SM}
For $i \in \{0,1\}$, the reduced polynomial of
$g_i \ast \mu$
contains the monomial $x_1 x_2 \dots x_{i+1}$ and hence has degree at least (in fact, exactly) $i + 1$; therefore $g_i \ast \mu \notin \clD{i}$.
Therefore
$K \clSM \nsubseteq \clD{i}$.

If $j \geq 2$, then
\begin{align*}
h_0 \ast \mu 
= {} & x_1 x_2 x_4 \dots x_{j+2} + x_1 x_3 x_4 \dots x_{j+2} + x_2 x_3 x_4 \dots x_{j+2} + x_{j+3} + a, \\
h_1 \ast \mu 
= {} & x_1 x_2 x_4 \dots x_{j+2} + x_1 x_3 x_4 \dots x_{j+2} + x_2 x_3 x_4 \dots x_{j+2} + a,
\end{align*}
so the characteristic of the $j$\hyp{}element set $\{1, \dots, j+1\} \setminus \{3\}$ is $1$.
Therefore, for $i \in \{0,1\}$, $h_i \ast \mu \notin \clChar{j}$,
which shows that $K \clSM \nsubseteq \clChar{j}$.

\ref{lem:non:X1Ca-X1Ea}
\ref{lem:non:SM-X1Ca}
The following calculations show that $\clSM (\clChar{1} \cap \clCon{a}) \nsubseteq \clChar{1}$ (the first line) and $\clSM (\clChar{1} \cap \clYksi{a}) \nsubseteq \clChar{1}$ (the second line) for $a \in \{0,1\}$:
\begin{align*}
& \mu(x_1, x_2, 0) = x_1 x_2, &
& \mu(x_1 + 1, x_2 + 1, 1) = x_1 x_2 + 1, \\
& \mu(x_1, x_2, 1) = x_1 x_2 + x_1 + x_2, &
& \mu(x_1 + 1, x_2 + 1, 0) = x_1 x_2 + x_1 + x_2 + 1.
\end{align*}

\ref{lem:non:DiXjPa}
\ref{lem:non:K-Ia-Pa}
We have $f := x_1 + x_2 \in \clD{i} \cap \clChar{j} \cap \clEven$ and $f' := x_1 \in \clD{i} \cap \clChar{j} \cap \clOdd$,
but
$f(x_1,0) = x_1 \notin \clEven$, $f'(0) = 0 \notin \clOdd$,
$f(x_1,1) = x_1 + 1 \notin \clEven$, $f'(1) = 1 \notin \clOdd$,
which shows that $K \clIo \nsubseteq \clParity{a}$ and $K \clIi \nsubseteq \clParity{a}$.

\ref{lem:non:K-Istar-Pa}
Assume $j \geq 2$.
We have $g := x_1 x_2 + x_2 \in \clD{i} \cap \clChar{j} \cap \clEven$ and $g' := x_1 x_2 \in \clD{i} \cap \clChar{j} \cap \clOdd$,
but
$g(x_1, x_2 + 1) = x_1 x_2 + x_1 + x_2 \notin \clEven$ and $g'(x_1, x_2 + 1) = x_1 x_2 + x_1 \notin \clOdd$;
therefore $K \clIstar \nsubseteq \clParity{a}$.

\ref{lem:non:LV-K-Pa}
We have $x_1, x_1 + 1 \in \clD{i} \cap \clChar{j} \cap \clOdd$, but
$\mathord{\wedge}(x_1, x_1 + 1) = x_1 \cdot (x_1 + 1) = x_1 + x_1 = 0 \notin \clOdd$,
$\mathord{\vee}(x_1, x_1 + 1) = x_1 \cdot (x_1 + 1) + x_1 + (x_1 + 1) = 1 \notin \clOdd$;
therefore $\clLambdac K \nsubseteq \clParity{1}$ and $\clVc K \nsubseteq \clParity{1}$.
\end{proof}

\begin{proposition}
\label{prop:all-empty}
For every clone $C$, we have
$\clAll C \subseteq \clAll$,
$C \clAll \subseteq \clAll$,
$\clEmpty C \subseteq \clEmpty$,
$C \clEmpty \subseteq \clEmpty$.
\end{proposition}

\begin{proof}
This is obvious.
\end{proof}

\begin{proposition}
\label{prop:Ca}
Let $a \in \{0,1\}$, and let $C$ be a clone.
\begin{enumerate}[label={\upshape{(\roman*)}}, leftmargin=*, widest=iii]
\item\label{lem:Ca:Ca-C}
$\clCon{a} C \subseteq \clCon{a}$ if and only if $C \subseteq \clTo$.
\item\label{lem:Ca:C-Ca}
$C \clCon{a} \subseteq \clCon{a}$ if and only if $C \subseteq \clTa{a}$.
\item\label{lem:Ca:Ea-C}
$\clYksi{a} C \subseteq \clYksi{a}$ if and only if $C \subseteq \clTi$.
\item\label{lem:Ca:C-Ea}
$C \clYksi{a} \subseteq \clYksi{a}$ if and only if $C \subseteq \clTa{a}$.
\end{enumerate}
\end{proposition}

\begin{proof}
\ref{lem:Ca:Ca-C}
Assume first that $C \subseteq \clTo$.
For any $f \in \clCon{a}^{(n)}$ and $g_1, \dots, g_n \in C^{(m)}$, we have
\[
f(g_1, \dots, g_n)(0, \dots, 0)
= f(g_1(0, \dots, 0), \dots, g_n(0, \dots, 0))
= f(0, \dots, 0)
= a,
\]
so $f(g_1, \dots, g_n) \in \clCon{a}$.
We conclude that $\clCon{a} C \subseteq \clCon{a}$.
Conversely, if $\clCon{a} C \subseteq \clCon{a}$, then $C$ includes neither $\clIi$ nor $\clIstar$ by Lemma~\ref{lem:non}\ref{lem:non:K-Ia-CaEb}, so $C \subseteq \clTo$.

\ref{lem:Ca:C-Ca}
Assume first that $C \subseteq \clTa{a}$.
For any $f \in C^{(n)}$ and $g_1, \dots, g_n \in \clCon{a}^{(m)}$, we have
\[
f(g_1, \dots, g_n)(0, \dots, 0)
= f(g_1(0, \dots, 0), \dots, g_n(0, \dots, 0))
= f(a, \dots, a)
= a,
\]
so $f(g_1, \dots g_n) \in \clCon{a}$.
We conclude that $C \clCon{a} \subseteq \clCon{a}$.
Conversely, if $C \clCon{a} \subseteq \clCon{a}$, then $C$ includes neither $\clIa{\overline{a}}$ nor $\clIstar$ by Lemma~\ref{lem:non}\ref{lem:non:Inota-K-CaEb}, \ref{lem:non:Istar-K-CaEb}, so $C \subseteq \clTa{a}$.

\ref{lem:Ca:Ea-C}
Assume first that $C \subseteq \clTi$.
For any $f \in \clYksi{a}^{(n)}$ and $g_1, \dots, g_n \in C^{(m)}$, we have
\[
f(g_1, \dots, g_n)(1, \dots, 1)
= f(g_1(1, \dots, 1), \dots, g_n(1, \dots, 1))
= f(1, \dots, 1)
= a,
\]
so $f(g_1, \dots, g_n) \in \clYksi{a}$.
We conclude that $\clYksi{a} C \subseteq \clYksi{a}$.
Conversely, if $\clYksi{a} C \subseteq \clYksi{a}$, then $C$ includes neither $\clIo$ nor $\clIstar$ by Lemma~\ref{lem:non}\ref{lem:non:K-Ia-CaEb}, so $C \subseteq \clTi$.

\ref{lem:Ca:C-Ea}
Assume first that $C \subseteq \clTa{a}$.
For any $f \in C^{(n)}$ and $g_1, \dots, g_n \in \clYksi{a}^{(m)}$, we have
\[
f(g_1, \dots, g_n)(1, \dots, 1)
= f(g_1(1, \dots, 1), \dots, g_n(1, \dots, 1))
= f(a, \dots a)
= a,
\]
so $f(g_1, \dots g_n) \in \clYksi{a}$.
We conclude that $C \clYksi{a} \subseteq \clYksi{a}$.
Conversely, if $C \clYksi{a} \subseteq \clYksi{a}$, then $C$ includes neither $\clIa{\overline{a}}$ nor $\clIstar$ by Lemma~\ref{lem:non}\ref{lem:non:Inota-K-CaEb}, \ref{lem:non:Istar-K-CaEb}, so $C \subseteq \clTa{a}$.
\end{proof}

\begin{lemma}
\label{lem:products-Even-Odd}
\leavevmode
\begin{enumerate}[label={\upshape{(\roman*)}}, leftmargin=*, widest=ii]
\item\label{lem:products-Even-Odd:Even}
For any $f, g \in \clEven$, we have $f \cdot g \in \clEven$.
\item\label{lem:products-Even-Odd:Odd}
For any $f, g \in \clOdd$, we have $f \cdot g \in \clOdd$ if and only if both $f$ and $g$ have equal constant terms \textup{(}i.e, $f, g \in \clCon{0}$ or $f, g \in \clCon{1}$\textup{)}.
\end{enumerate}
\end{lemma}

\begin{proof}
\ref{lem:products-Even-Odd:Even}
Let $\alpha, \beta \in \clEven \cap \clCon{0}$.
Then both $\alpha$ and $\beta$ are sums of an even number of monomials.
We have $\alpha \cdot \beta \in \clEven$ because the expansion of the product of the two even sums of monomials yields a sum of an even number of monomials.
We clearly also have that
$(\alpha + 1) \cdot \beta = \alpha \cdot \beta + \beta$, $\alpha \cdot (\beta + 1) = \alpha \cdot \beta + \alpha$, and $(\alpha + 1) \cdot (\beta + 1) = \alpha \cdot \beta + \alpha + \beta + 1$ belong to $\clEven$ because they are sums of polynomials with an even number of monomials plus a possible constant term.
The claim now follows because any $f \in \clEven$ is of the form $\alpha$ or $\alpha + 1$ for some $\alpha \in \clEven \cap \clCon{0}$.

\ref{lem:products-Even-Odd:Odd}
Let $\alpha, \beta \in \clOdd \cap \clCon{0}$.
Then both $\alpha$ and $\beta$ are sums of an odd number of monomials.
We have $\alpha \cdot \beta \in \clOdd$ because the expansion of the product of the two odd sums of monomials yields a sum of an odd number of monomials.
Consequently, $(\alpha + 1) \cdot \beta = \alpha \cdot \beta + \beta \in \clEven$, $\alpha \cdot (\beta + 1) = \alpha \cdot \beta + \alpha \in \clEven$, and $(\alpha + 1) \cdot (\beta + 1) = \alpha \cdot \beta + \alpha + \beta + 1 \in \clOdd$.
\end{proof}

\begin{proposition}
\label{prop:odd-even}
Let $C$ be a clone.
\begin{enumerate}[label={\upshape{(\roman*)}}, leftmargin=*, widest=iii]
\item\label{lem:odd-even:even-C}
$\clEven C \subseteq \clEven$ if and only if $C \subseteq \clTc$.
\item\label{lem:odd-even:C-even}
$C \clEven \subseteq \clEven$ for any clone $C$.
\item\label{lem:odd-even:odd-C}
$\clOdd C \subseteq \clOdd$ if and only if $C \subseteq \clTc$.
\item\label{lem:odd-even:C-odd}
$C \clOdd \subseteq \clOdd$ if and only if $C \subseteq \clS$.
\end{enumerate}
\end{proposition}

\begin{proof}
Recall that $\clTc = \clCon{0} \cap \clOdd$.

\ref{lem:odd-even:even-C}
Assume first that $C \subseteq \clTc$.
Let $f \in \clEven^{(n)}$ and $g_1, \dots, g_n \in C^{(m)}$.
Observing that $\clTc = \clBoth{0}{1} = \clOdd \cap \clCon{0}$,
we have
\begin{align*}
f(g_1, \dots, g_n)
= \sum_{S \in \monomials{f}} \prod_{i \in S} g_i
\in \clEven,
\end{align*}
because each summand $\prod_{i \in S} g_i$ is odd by Lemma~\ref{lem:products-Even-Odd}, and there are an even number of such summands since $f$ is even.
We conclude that $\clEven C \subseteq \clEven$.
Conversely, if $\clEven C \subseteq \clEven$, then $C$ includes neither $\clIo$, $\clIi$, nor $\clIstar$ by Lemma~\ref{lem:non}\ref{lem:non:K-Ia-Pa}, \ref{lem:non:K-Istar-Pa}, so $C \subseteq \clTc$.

\ref{lem:odd-even:C-even}
It is enough to prove the claim for $C = \clAll$.
Using the fact that $\clAll = \clonegen{x_1 x_2 + 1}$, we will apply Lemma~\ref{lem:left-stab-gen}.
Let $g_1, g_2 \in \clEven$.
Lemma~\ref{lem:products-Even-Odd} gives $(x_1 x_2 + 1)(g_1, g_2) = g_1 g_2 + 1 \in \clEven$.
Now it follows from Lemma~\ref{lem:left-stab-gen} that $\clAll \clEven \subseteq \clEven$.

\ref{lem:odd-even:odd-C}
Assume first that $C \subseteq \clTc$.
Let $f \in \clOdd^{(n)}$ and $g_1, \dots, g_n \in C^{(m)}$.
If $f \in \clCon{0}$, then $f \in \clCon{0} \cap \clOdd = \clTc$.
It follows immediately from the fact that $\clTc$ is a clone that $f(g_1, \dots, g_n) \in \clTc = \clCon{0} \cap \clOdd \subseteq \clOdd$.
If $f \in \clCon{1}$, then $f' := f + 1 \in \clCon{0} \cap \clOdd = \clTc$.
It follows from Lemma~\ref{lem:comp-minors} that
\begin{align*}
f(g_1, \dots, g_n)
& = (f' + 1)(g_1, \dots, g_n)
= f'(g_1, \dots, g_n) + 1(g_1, \dots, g_n)
\\ &
= \underbrace{f'(g_1, \dots, g_n)}_{\in \clTc = \clCon{0} \cap \clOdd} + 1
\in \clCon{1} \cap \clOdd
\subseteq \clOdd.
\end{align*}
We conclude that $\clOdd C \subseteq \clOdd$.
Conversely, if $\clOdd C \subseteq \clOdd$, then $C$ includes neither $\clIo$, $\clIi$, nor $\clIstar$ by Lemma~\ref{lem:non}\ref{lem:non:K-Ia-Pa}, \ref{lem:non:K-Istar-Pa}, so $C \subseteq \clTc$.

\ref{lem:odd-even:C-odd}
For sufficiency, it is enough to prove the claim for $C = \clS$.
Using the fact that $\clS = \clonegen{\mu, \, x_1 + 1}$, we will apply Lemma~\ref{lem:left-stab-gen}.
Let $g_1, g_2, g_3 \in \clOdd$.
We clearly have $(x_1 + 1)(g_1) = g_1 + 1 \in \clOdd$.
Applying Lemma~\ref{lem:products-Even-Odd}, we see that $\mu(g_1, g_2, g_3) = g_1 g_2 + g_1 g_3 + g_2 g_3 \in \clOdd$; for if $g_1$, $g_2$, $g_3$ have the same constant term, then the three summands $g_1 g_2$, $g_1 g_3$, $g_2 g_3$ belong to $\clOdd$; if they do not all have the same constant term, then it is easy to see that exactly one of the summands belongs to $\clOdd$ and the other two belong to $\clEven$.
Now it follows from Lemma~\ref{lem:left-stab-gen} that $\clS \clOdd \subseteq \clOdd$.

For necessity, assume that $C \clOdd \subseteq \clOdd$.
Then $C$ includes neither $\clIo$, $\clIi$, $\clLambdac$, nor $\clVc$ by Lemma~\ref{lem:non}\ref{lem:non:Ia-K-P1}, \ref{lem:non:LV-K-Pa}, so $C \subseteq \clS$.
\end{proof}

\begin{proposition}
\label{prop:CaEb}
Let $a, b \in \{0,1\}$, and let $C$ be a clone.
\begin{enumerate}[label={\upshape{(\roman*)}}, leftmargin=*, widest=ii]
\item\label{lem:CaEb:CaEb-C}
$\clBoth{a}{b} C \subseteq \clBoth{a}{b}$ if and only if $C \subseteq \clTc$.
\item\label{lem:CaEb:C-CaEb}
$C \clBoth{a}{b} \subseteq \clBoth{a}{b}$ if and only if $C \subseteq \clTa{a} \cap \clTa{b}$.
\end{enumerate}
\end{proposition}

\begin{proof}
\ref{lem:CaEb:CaEb-C}
Lemma~\ref{lem:suff-int} and Proposition~\ref{prop:Ca}\ref{lem:Ca:Ca-C}, \ref{lem:Ca:Ea-C} imply that $\clBoth{a}{b} C \subseteq \clBoth{a}{b}$ whenever $C \subseteq \clTo \cap \clTi = \clTc$.
Conversely, if $\clBoth{a}{b} C \subseteq \clBoth{a}{b}$, then the clone $C$ includes neither $\clIo$, $\clIi$, nor $\clIstar$ by Lemma~\ref{lem:non}\ref{lem:non:K-Ia-CaEb}, so $C \subseteq \clTc$.

\ref{lem:CaEb:C-CaEb}
Lemma~\ref{lem:suff-int} and Proposition~\ref{prop:Ca}\ref{lem:Ca:C-Ca}, \ref{lem:Ca:C-Ea} imply that $C \clBoth{a}{b} \subseteq \clBoth{a}{b}$ whenever $C \subseteq \clTa{a} \cap \clTa{b}$.

Assume now that $C \clBoth{a}{b} \subseteq \clBoth{a}{b}$.
If $a = b$, then the clone $C$ includes neither $\clIa{\overline{a}}$ nor $\clIstar$ by Lemma~\ref{lem:non}\ref{lem:non:Inota-K-CaEb}, \ref{lem:non:Istar-K-CaEb}, so $C \subseteq \clTa{a} = \clTa{a} \cap \clTa{b}$.
If $a \neq b$, then $C$ includes neither $\clIo$, $\clIi$, nor $\clIstar$ by Lemma~\ref{lem:non}\ref{lem:non:Ia-K-CaEb}, \ref{lem:non:Istar-K-CaEb}, so $C \subseteq \clTc = \clTa{a} \cap \clTa{b}$.
\end{proof}

\begin{lemma}
\label{lem:X0apu}
\leavevmode
\begin{enumerate}[label={\upshape{(\roman*)}}, leftmargin=*, widest=ii]
\item\label{lem:X0apu:X0SinX0}
$\clChar{0} \clS \subseteq \clChar{0}$.
\item\label{lem:X0apu:OmegaX0inX0}
$\clAll \clChar{0} \subseteq \clChar{0}$.
\end{enumerate}
\end{lemma}

\begin{proof}
\ref{lem:X0apu:X0SinX0}
Let $f \in \clChar{0}^{(n)}$ and $g_1, \dots, g_n \in \clS^{(m)}$.
Since $\clChar{0}$ is the class of all reflexive functions and $\clS$ is the class of all self\hyp{}dual functions, we have, for any $\vect{a} \in \{0,1\}^m$ that
\begin{align*}
f(g_1, \dots, g_n)(\overline{\vect{a}})
&
= f(g_1(\overline{\vect{a}}), \dots, g_n(\overline{\vect{a}}))
= f(\overline{g_1(\vect{a})}, \dots, \overline{g_n(\vect{a})})
\\ &
= f(g_1(\vect{a}), \dots, g_n(\vect{a}))
= f(g_1, \dots, g_n)(\vect{a}),
\end{align*}
so $f(g_1, \dots, g_n) \in \clChar{0}$.

\ref{lem:X0apu:OmegaX0inX0}
Let $f \in \clAll^{(n)}$, $g_1, \dots, g_n \in \clChar{0}^{(m)}$.
We have, for any $\vect{a} \in \{0,1\}^m$,
\[
f(g_1, \dots, g_n)(\overline{\vect{a}})
= f(g_1(\overline{\vect{a}}), \dots, g_n(\overline{\vect{a}}))
= f(g_1(\vect{a}), \dots, g_n(\vect{a}))
= f(g_1, \dots, g_n)(\vect{a}),
\]
so $f(g_1, \dots, g_n) \in \clChar{0}$.
\end{proof}

\begin{proposition}
\label{prop:Xk}
Let $k \in \IN_{+}$, and let $C$ be a clone.
\begin{enumerate}[label={\upshape{(\roman*)}}, leftmargin=*, widest=iii]
\item\label{lem:Xk:Xk-C}
For $k \geq 2$,
$\clChar{k} C \subseteq \clChar{k}$ if and only if $C \subseteq \clLS$.
\item\label{lem:Xk:X1-C}
$\clChar{1} C \subseteq \clChar{1}$ if and only if $C \subseteq \clS$.
\item\label{lem:Xk:C-Xk}
$C \clChar{k} \subseteq \clChar{k}$ if and only if $C \subseteq \clL$.
\end{enumerate}
\end{proposition}

\begin{proof}
\ref{lem:Xk:Xk-C}
For sufficiency, it is enough to prove the claim for $C = \clLS$.
Using the fact that $\clLS = \clonegen{\mathord{\oplus_3}, \, x_1 + 1}$, we apply Lemma~\ref{lem:right-stab-gen}.
Let $f \in \clChar{k}$.
We have $f \ast \mathord{\oplus_3} = \mathord{\oplus_3}(f_{\sigma_1}, f_{\sigma_2}, f_{\sigma_3})$, where the $\sigma_i$ are as in Lemma~\ref{lem:Lc-simplify}.
Since $\clChar{k}$ is closed under minors and sums by Lemma~\ref{lem:Xk-closed}, we have $\mathord{\oplus_3}(f_{\sigma_1}, f_{\sigma_2}, f_{\sigma_3}) \in \clChar{k}$.
As for $f \ast (x_1 + 1)$, note that $f \ast (x_1 + 1) = f + f'_1$ by Lemma~\ref{lem:f'i-negation}.
Since $f'_1 \in \clChar{k-1} \subseteq \clChar{k}$ by Lemma~\ref{lem:f'i}, we have $f \ast (x_1 + 1) = f + f'_1 \in \clChar{k}$ by Lemma~\ref{lem:Xk-closed}.
It follows from Lemma~\ref{lem:right-stab-gen} that $\clChar{k} \, \clLS \subseteq \clChar{k}$.

For necessity, assume that $\clChar{k} C \subseteq \clChar{k}$.
Then $C$ includes neither $\clIo$, $\clIi$, $\clLambdac$, $\clVc$, nor $\clSM$ by Lemma~\ref{lem:non}\ref{lem:non:K-Ia-Xj}, \ref{lem:non:K-LV}, \ref{lem:non:K-SM}, so $C \nsubseteq \clLS$.

\ref{lem:Xk:X1-C}
Assume first that $C \subseteq \clS$.
Since $\clChar{1} = (\clChar{1} \cap \clEven) \cup (\clChar{1} \cap \clOdd) = \clChar{0} \cup \clS$ and $\clS$ is a clone, it follows from Lemmata~\ref{lem:union-left} and \ref{lem:X0apu}\ref{lem:X0apu:X0SinX0} that
\[
\clChar{1} \clS
\subseteq (\clChar{0} \cup \clS) \clS
= \clChar{0} \clS \cup \clS \clS
\subseteq \clChar{0} \cup \clS
= \clChar{1}.
\]
Conversely, if $\clChar{1} C \subseteq \clChar{1}$, then $C$ includes neither $\clIo$, $\clIi$, $\clLambdac$, nor $\clVc$ by Lemma~\ref{lem:non}\ref{lem:non:K-Ia-Xj}, \ref{lem:non:K-LV}, so $C \subseteq \clS$.

\ref{lem:Xk:C-Xk}
For sufficiency, it is enough to prove the claim for $C = \clL$.
Using the fact that $\clL = \clonegen{x_1 + x_2, 1}$, we apply Lemma~\ref{lem:left-stab-gen}.
For any $g_1, g_2 \in \clChar{k}^{(n)}$, we clearly have
$1(g_1) = 1 \in \clChar{k}$
and
$(x_1 + x_2)(g_1, g_2) = g_1 + g_2 \in \clChar{k}$ by Lemma~\ref{lem:Xk-closed}.
It follows from Lemma~\ref{lem:left-stab-gen} that $\clL \clChar{k} \subseteq \clChar{k}$.

For necessity, assume that $C \clChar{k} \subseteq \clChar{k}$.
Then $C$ includes neither $\clLambdac$, $\clVc$, nor $\clSM$ by Lemma~\ref{lem:non}\ref{lem:non:LV-K-Xj}, \ref{lem:non:SM-K}, \ref{lem:non:SM-X1Ca} so $C \subseteq \clL$.
\end{proof}

\begin{proposition}
\label{prop:XkCa}
Let $k \in \IN_{+}$, $a \in \{0,1\}$, and let $C$ be a clone.
\begin{enumerate}[label={\upshape{(\roman*)}}, leftmargin=*, widest=iii]
\item\label{lem:XkCa:XkCa-C}
For $k \geq 2$,
$(\clChar{k} \cap \clCon{a}) C \subseteq \clChar{k} \cap \clCon{a}$ if and only if $C \subseteq \clLc$.

\item\label{lem:XkCa:X1Ca-C}
$(\clChar{1} \cap \clCon{a}) C \subseteq \clChar{1} \cap \clCon{a}$ if and only if $C \subseteq \clSc$.

\item\label{lem:XkCa:C-XkCa}
$C (\clChar{k} \cap \clCon{a}) \subseteq \clChar{k} \cap \clCon{a}$ if and only if $C \subseteq \clLa{a}$.
\end{enumerate}
\end{proposition}

\begin{proof}
\ref{lem:XkCa:XkCa-C}
Lemma~\ref{lem:suff-int} and Propositions~\ref{prop:Ca}\ref{lem:Ca:Ca-C} and \ref{prop:Xk}\ref{lem:Xk:Xk-C} imply that $(\clChar{k} \cap \clCon{a}) C \subseteq \clChar{k} \cap \clCon{a}$ whenever $C \subseteq \clTo \cap \clLS = \clLc$.
Conversely, if $(\clChar{k} \cap \clCon{a}) C \subseteq \clChar{k} \cap \clCon{a}$, then $C$ includes neither $\clIo$, $\clIi$, $\clIstar$, $\clLambdac$, $\clVc$, nor $\clSM$ by Lemma~\ref{lem:non}\ref{lem:non:K-Ia-CaEb}, \ref{lem:non:K-Ia-Xj}, \ref{lem:non:K-LV}, \ref{lem:non:K-SM}, so $C \subseteq \clLc$.

\ref{lem:XkCa:X1Ca-C}
Lemma~\ref{lem:suff-int} and Propositions~\ref{prop:Ca}\ref{lem:Ca:Ca-C} and \ref{prop:Xk}\ref{lem:Xk:X1-C} imply that $(\clChar{1} \cap \clCon{a}) C \subseteq \clChar{1} \cap \clCon{a}$ whenever $C \subseteq \clTo \cap \clS = \clSc$.
Conversely, if $(\clChar{1} \cap \clCon{a}) C \subseteq \clChar{1} \cap \clCon{a}$, then $C$ includes neither $\clIo$, $\clIi$, $\clIstar$, $\clLambdac$, nor $\clVc$ by Lemma~\ref{lem:non}\ref{lem:non:K-Ia-CaEb}, \ref{lem:non:K-Ia-Xj}, \ref{lem:non:K-LV}, so $C \subseteq \clSc$.

\ref{lem:XkCa:C-XkCa}
Lemma~\ref{lem:suff-int} and Propositions~\ref{prop:Ca}\ref{lem:Ca:C-Ca} and \ref{prop:Xk}\ref{lem:Xk:C-Xk} imply that $C (\clChar{k} \cap \clCon{a}) \subseteq \clChar{k} \cap \clCon{a}$ whenever $C \subseteq \clTa{a} \cap \clL = \clLa{a}$.
Conversely, if $C (\clChar{k} \cap \clCon{a}) \subseteq \clChar{k} \cap \clCon{a}$, then $C$ includes neither $\clIa{\overline{a}}$, $\clIstar$, $\clLambdac$, $\clVc$, nor $\clSM$ by Lemma~\ref{lem:non}\ref{lem:non:Inota-K-CaEb}, \ref{lem:non:Istar-K-CaEb}, \ref{lem:non:LV-K-Xj}, \ref{lem:non:SM-K}, \ref{lem:non:SM-X1Ca}, so $C \subseteq \clLa{a}$.
\end{proof}

\begin{proposition}
\label{prop:XkEa}
Let $k \in \IN_{+}$, $a \in \{0,1\}$, and let $C$ be a clone.
\begin{enumerate}[label={\upshape{(\roman*)}}, leftmargin=*, widest=iii]
\item\label{lem:XkEa:XkEa-C}
For $k \geq 2$,
$(\clChar{k} \cap \clYksi{a}) C \subseteq \clChar{k} \cap \clYksi{a}$ if and only if $C \subseteq \clLc$.

\item\label{lem:XkEa:X1Ea-C}
$(\clChar{1} \cap \clYksi{a}) C \subseteq \clChar{1} \cap \clYksi{a}$ if and only if $C \subseteq \clSc$.

\item\label{lem:XkEa:C-XkEa}
$C (\clChar{k} \cap \clYksi{a}) \subseteq \clChar{k} \cap \clYksi{a}$ if and only if $C \subseteq \clLa{a}$.
\end{enumerate}
\end{proposition}

\begin{proof}
\ref{lem:XkEa:XkEa-C}
Lemma~\ref{lem:suff-int} and Propositions~\ref{prop:Ca}\ref{lem:Ca:Ea-C} and \ref{prop:Xk}\ref{lem:Xk:Xk-C} imply that $(\clChar{k} \cap \clYksi{a}) C \subseteq \clChar{k} \cap \clYksi{a}$ whenever $C \subseteq \clTi \cap \clLS = \clLc$.
Conversely, if $(\clChar{k} \cap \clYksi{a}) C \subseteq \clChar{k} \cap \clYksi{a}$, then $C$ includes neither $\clIo$, $\clIi$, $\clIstar$, $\clLambdac$, $\clVc$, nor $\clSM$ by Lemma~\ref{lem:non}\ref{lem:non:K-Ia-CaEb}, \ref{lem:non:K-Ia-Xj}, \ref{lem:non:K-LV}, \ref{lem:non:K-SM}, so $C \subseteq \clLc$.

\ref{lem:XkEa:X1Ea-C}
Lemma~\ref{lem:suff-int} and Propositions~\ref{prop:Ca}\ref{lem:Ca:Ea-C} and \ref{prop:Xk}\ref{lem:Xk:X1-C} imply that $(\clChar{1} \cap \clYksi{a}) C \subseteq \clChar{1} \cap \clYksi{a}$ whenever $C \subseteq \clTi \cap \clS = \clSc$.
Conversely, if $(\clChar{1} \cap \clYksi{a}) C \subseteq \clChar{1} \cap \clYksi{a}$, then $C$ includes neither $\clIo$, $\clIi$, $\clIstar$, $\clLambdac$, nor $\clVc$ by Lemma~\ref{lem:non}\ref{lem:non:K-Ia-CaEb}, \ref{lem:non:K-Ia-Xj}, \ref{lem:non:K-LV}, so $C \subseteq \clSc$.

\ref{lem:XkEa:C-XkEa}
Lemma~\ref{lem:suff-int} and Propositions~\ref{prop:Ca}\ref{lem:Ca:C-Ea} and \ref{prop:Xk}\ref{lem:Xk:C-Xk} imply that $C (\clChar{k} \cap \clCon{a}) \subseteq \clChar{k} \cap \clCon{a}$ whenever $C \subseteq \clTa{a} \cap \clL = \clLa{a}$.
Conversely, if $C (\clChar{k} \cap \clYksi{a}) \subseteq \clChar{k} \cap \clYksi{a}$, then $C$ includes neither $\clIa{\overline{a}}$, $\clIstar$, $\clLambdac$, $\clVc$, nor $\clSM$ by Lemma~\ref{lem:non}\ref{lem:non:Inota-K-CaEb}, \ref{lem:non:Istar-K-CaEb}, \ref{lem:non:LV-K-Xj}, \ref{lem:non:SM-K}, \ref{lem:non:SM-X1Ca}, so $C \subseteq \clLa{a}$.
\end{proof}

\begin{proposition}
\label{prop:XkP0}
Let $k \in \IN_{+}$, and let $C$ be a clone.
\begin{enumerate}[label={\upshape{(\roman*)}}, leftmargin=*, widest=iii]
\item\label{lem:XkP0:XkP0-C}
For $k \geq 2$,
$(\clChar{k} \cap \clEven) C \subseteq \clChar{k} \cap \clEven$ if and only if $C \subseteq \clLc$.
\item\label{lem:XkP0:X1P0-C}
$(\clChar{1} \cap \clEven) C \subseteq \clChar{1} \cap \clEven$ if and only if $C \subseteq \clS$.
\item\label{lem:XkP0:C-XkP0}
For $k \geq 2$,
$C (\clChar{k} \cap \clEven) \subseteq \clChar{k} \cap \clEven$ if and only if $C \subseteq \clL$.
\item\label{lem:XkP0:C-X1P0}
$C (\clChar{1} \cap \clEven) \subseteq \clChar{1} \cap \clEven$ for any clone $C$.
\end{enumerate}
\end{proposition}

\begin{proof}
\ref{lem:XkP0:XkP0-C}
Lemma~\ref{lem:suff-int} and Propositions~\ref{prop:odd-even}\ref{lem:odd-even:even-C} and \ref{prop:Xk}\ref{lem:Xk:Xk-C} imply that $(\clChar{k} \cap \clEven) C \subseteq \clChar{k} \cap \clEven$ whenever $C \subseteq \clLc \cap \clTc = \clLc$.
Conversely, if $(\clChar{k} \cap \clEven) C \subseteq \clChar{k} \cap \clEven$, then $C$ includes neither $\clIo$, $\clIi$, $\clIstar$, $\clLambdac$, $\clVc$, nor $\clSM$ by Lemma~\ref{lem:non}\ref{lem:non:K-LV}, \ref{lem:non:K-SM}, \ref{lem:non:K-Ia-Pa}, \ref{lem:non:K-Istar-Pa}, so $C \subseteq \clLc$.

\ref{lem:XkP0:X1P0-C}
Assume first that $C \subseteq \clS$.
Since $\clChar{1} \cap \clEven = \clChar{0}$, it follows from Lemma~\ref{lem:X0apu}\ref{lem:X0apu:X0SinX0} that
$(\clChar{1} \cap \clEven) C \subseteq \clChar{0} \clS \subseteq \clChar{0} = \clChar{1} \cap \clEven$.
Conversely, if $(\clChar{1} \cap \clEven) C \subseteq \clChar{1} \cap \clEven$, then $C$ includes neither $\clIo$, $\clIi$, $\clLambdac$, nor $\clVc$ by Lemma~\ref{lem:non}\ref{lem:non:K-LV}, \ref{lem:non:K-Ia-Pa}, so $C \subseteq \clS$.

\ref{lem:XkP0:C-XkP0}
Lemma~\ref{lem:suff-int} and Propositions~\ref{prop:odd-even}\ref{lem:odd-even:C-even} and \ref{prop:Xk}\ref{lem:Xk:C-Xk} imply that $C (\clChar{k} \cap \clEven) \subseteq \clChar{k} \cap \clEven$ whenever $C \subseteq \clL \cap \clAll = \clL$.
Conversely, if $C (\clChar{k} \cap \clEven) \subseteq \clChar{k} \cap \clEven$, then $C$ includes neither $\clLambdac$, $\clVc$, nor $\clSM$ by Lemma~\ref{lem:non}\ref{lem:non:LV-K-Xj}, \ref{lem:non:SM-K} (note that $\clBoth{0}{0} \subseteq \clEven$), so $C \subseteq \clL$.

\ref{lem:XkP0:C-X1P0}
Observing that $\clChar{1} \cap \clEven = \clChar{0}$, this follows immediately from Lemma~\ref{lem:X0apu}\ref{lem:X0apu:OmegaX0inX0}.
\end{proof}

\begin{proposition}
\label{prop:XkP1}
Let $k \in \IN_{+}$, and let $C$ be a clone.
\begin{enumerate}[label={\upshape{(\roman*)}}, leftmargin=*, widest=iii]
\item\label{lem:XkP1:XkP1-C}
For $k \geq 2$,
$(\clChar{k} \cap \clOdd) C \subseteq \clChar{k} \cap \clOdd$ if and only if $C \subseteq \clLc$.

\item\label{lem:XkP1:X1P1-C}
$(\clChar{1} \cap \clOdd) C \subseteq \clChar{1} \cap \clOdd$ if and only if $C \subseteq \clS$.

\item\label{lem:XkP1:C-XkP1}
For $k \geq 2$,
$C (\clChar{k} \cap \clOdd) \subseteq \clChar{k} \cap \clOdd$ if and only if $C \subseteq \clLS$.

\item\label{lem:XkP1:C-X1P1}
$C (\clChar{1} \cap \clOdd) \subseteq \clChar{1} \cap \clOdd$ if and only if $C \subseteq \clS$.
\end{enumerate}
\end{proposition}

\begin{proof}
\ref{lem:XkP1:XkP1-C}
Lemma~\ref{lem:suff-int} and Propositions~\ref{prop:odd-even}\ref{lem:odd-even:odd-C} and \ref{prop:Xk}\ref{lem:Xk:Xk-C} imply that $(\clChar{k} \cap \clOdd) C \subseteq \clChar{k} \cap \clOdd$ whenever $C \subseteq \clLS \cap \clTc = \clLc$.
Conversely, if $(\clChar{k} \cap \clOdd) C \subseteq \clChar{k} \cap \clOdd$, then $C$ includes neither $\clIo$, $\clIi$, $\clIstar$, $\clLambdac$, $\clVc$, nor $\clSM$ by Lemma~\ref{lem:non}\ref{lem:non:K-LV}, \ref{lem:non:K-SM}, \ref{lem:non:K-Ia-Pa}, \ref{lem:non:K-Istar-Pa}, so $C \subseteq \clLc$.

\ref{lem:XkP1:X1P1-C}
Assume first that $C \subseteq \clS$.
Since $\clChar{1} \cap \clOdd = \clS$ and $\clS$ is a clone, it is immediately obvious that $(\clChar{1} \cap \clOdd) C \subseteq \clS \clS \subseteq \clS = \clChar{1} \cap \clOdd$.
Conversely, if $(\clChar{1} \cap \clOdd) C \subseteq \clChar{1} \cap \clOdd$, then $C$ includes neither $\clIo$, $\clIi$, $\clLambdac$, nor $\clVc$ by Lemma~\ref{lem:non}\ref{lem:non:K-LV}, \ref{lem:non:K-Ia-Pa}, so $C \subseteq \clS$.

\ref{lem:XkP1:C-XkP1}
Lemma~\ref{lem:suff-int} and Propositions~\ref{prop:odd-even}\ref{lem:odd-even:C-odd} and \ref{prop:Xk}\ref{lem:Xk:C-Xk} imply that $C (\clChar{k} \cap \clOdd) \subseteq \clChar{k} \cap \clOdd$ whenever $C \subseteq \clL \cap \clS = \clLS$.
Conversely, if $C (\clChar{k} \cap \clOdd) \subseteq \clChar{k} \cap \clOdd$, then $C$ includes neither $\clIo$, $\clIi$, $\clLambdac$, $\clVc$, nor $\clSM$ by Lemma~\ref{lem:non}\ref{lem:non:Ia-K-P1}, \ref{lem:non:LV-K-Xj}, \ref{lem:non:SM-K}, so $C \subseteq \clLS$.

\ref{lem:XkP1:C-X1P1}
Assume first that $C \subseteq \clS$.
Since $\clChar{1} \cap \clOdd = \clS$ and $\clS$ is a clone, it is immediately obvious that $C (\clChar{1} \cap \clOdd) \subseteq \clS \clS \subseteq \clS = \clChar{1} \cap \clOdd$.
Conversely, if $C (\clChar{1} \cap \clOdd) \subseteq \clChar{1} \cap \clOdd$, then $C$ includes neither $\clIo$, $\clIi$, $\clLambdac$, nor $\clVc$ by Lemma~\ref{lem:non}\ref{lem:non:Ia-K-P1}, \ref{lem:non:LV-K-Xj}, so $C \subseteq \clS$.
\end{proof}

\begin{proposition}
\label{prop:XkCaEb}
Let $k \in \IN_{+}$, $a, b \in \{0,1\}$, and let $C$ be a clone.
\begin{enumerate}[label={\upshape{(\roman*)}}, leftmargin=*, widest=iii]
\item\label{lem:XkCaEb:XkCaEb-C}
For $k \geq 2$,
$(\clChar{k} \cap \clBoth{a}{b}) C \subseteq \clChar{k} \cap \clBoth{a}{b}$ if and only if $C \subseteq \clLc$.

\item\label{lem:XkCaEb:X1CaEb-C}
$(\clChar{1} \cap \clBoth{a}{b}) C \subseteq \clChar{1} \cap \clBoth{a}{b}$ if and only if $C \subseteq \clSc$.

\item\label{lem:XkCaEb:C-XkCaEb}
If $k \geq 2$, then
$C (\clChar{k} \cap \clBoth{a}{b}) \subseteq \clChar{k} \cap \clBoth{a}{b}$ if and only if $C \subseteq \clLa{a} \cap \clLa{b}$.

\item\label{lem:XkCaEb:C-X1CaEa}
If $a = b$, then
$C (\clChar{1} \cap \clBoth{a}{b}) \subseteq \clChar{1} \cap \clBoth{a}{b}$ if and only if $C \subseteq \clTa{a}$.

\item\label{lem:XkCaEb:C-X1CaEb}
If $a \neq b$, then
$C (\clChar{1} \cap \clBoth{a}{b}) \subseteq \clChar{1} \cap \clBoth{a}{b}$ if and only if $C \subseteq \clSc$.
\end{enumerate}
\end{proposition}

\begin{proof}
\ref{lem:XkCaEb:XkCaEb-C}
Lemma~\ref{lem:suff-int} and Propositions~\ref{prop:CaEb}\ref{lem:CaEb:CaEb-C} and \ref{prop:Xk}\ref{lem:Xk:Xk-C} imply that $(\clChar{k} \cap \clBoth{a}{b}) C \subseteq \clChar{k} \cap \clBoth{a}{b}$ whenever $C \subseteq \clLS \cap \clTc = \clLc$.
Conversely, if $(\clChar{k} \cap \clBoth{a}{b}) C \subseteq \clChar{k} \cap \clBoth{a}{b}$, then $C$ includes neither $\clIo$, $\clIi$, $\clIstar$, $\clLambdac$, $\clVc$, nor $\clSM$ by Lemma~\ref{lem:non}\ref{lem:non:K-Ia-CaEb}, \ref{lem:non:K-LV}, \ref{lem:non:K-SM}, so $C \subseteq \clLc$.

\ref{lem:XkCaEb:X1CaEb-C}
Lemma~\ref{lem:suff-int} and Propositions~\ref{prop:CaEb}\ref{lem:CaEb:CaEb-C} and \ref{prop:Xk}\ref{lem:Xk:X1-C} imply that $(\clChar{1} \cap \clBoth{a}{b}) C \subseteq \clChar{1} \cap \clBoth{a}{b}$ whenever $C \subseteq \clS \cap \clTc = \clSc$.
Conversely, if $(\clChar{1} \cap \clBoth{a}{b}) C \subseteq \clChar{1} \cap \clBoth{a}{b}$, then $C$ includes neither $\clIo$, $\clIi$, $\clIstar$, $\clLambdac$, nor $\clVc$ by Lemma~\ref{lem:non}\ref{lem:non:K-Ia-CaEb}, \ref{lem:non:K-LV}, so $C \subseteq \clSc$.

\ref{lem:XkCaEb:C-XkCaEb}
Lemma~\ref{lem:suff-int} and Propositions~\ref{prop:CaEb}\ref{lem:CaEb:C-CaEb} and \ref{prop:Xk}\ref{lem:Xk:C-Xk} imply that $C (\clChar{k} \cap \clBoth{a}{b}) \subseteq \clChar{k} \cap \clBoth{a}{b}$ whenever $C \subseteq \clL \cap \clTa{a} \cap \clTa{b} = \clLa{a} \cap \clLa{b}$.

Assume now that $C (\clChar{k} \cap \clBoth{a}{b}) \subseteq \clChar{k} \cap \clBoth{a}{b}$.
Then $C$ includes neither $\clIstar$, $\clLambdac$, $\clVc$, nor $\clSM$ by Lemma~\ref{lem:non}\ref{lem:non:Istar-K-CaEb}, \ref{lem:non:LV-K-Xj}, \ref{lem:non:SM-K}.
If $a \neq b$, then $C$ includes neither $\clIo$ nor $\clIi$ by Lemma~\ref{lem:non}\ref{lem:non:Ia-K-CaEb}, so $C \subseteq \clLc = \clLa{a} \cap \clLa{b}$.
If $a = b$, then $C$ does not include $\clIa{\overline{a}}$ by Lemma~\ref{lem:non}\ref{lem:non:Inota-K-CaEb}, so $C \subseteq \clLa{a} = \clLa{a} \cap \clLa{b}$.

\ref{lem:XkCaEb:C-X1CaEa}
Assume $C \subseteq \clTa{a}$.
We have $C (\clChar{1} \cap \clBoth{a}{a}) \subseteq \clTa{a} (\clChar{0} \cap \clCon{a}) \subseteq \clChar{0} \cap \clCon{a} = \clChar{1} \cap \clBoth{a}{a}$, where the second inclusion holds because
$\clTa{a} (\clChar{0} \cap \clCon{a}) \subseteq \clAll \clChar{0} \subseteq \clChar{0}$ by Lemma~\ref{lem:X0apu}\ref{lem:X0apu:OmegaX0inX0}
and
$\clTa{a} (\clChar{0} \cap \clCon{a}) \subseteq \clTa{a} \clCon{a} \subseteq \clCon{a}$ as can be easily seen.

Assume now that $C (\clChar{1} \cap \clBoth{a}{a}) \subseteq \clChar{1} \cap \clBoth{a}{a}$.
Then the clone $C$ includes neither $\clIa{\overline{a}}$ nor $\clIstar$ by Lemma~\ref{lem:non}\ref{lem:non:Inota-K-CaEb}, \ref{lem:non:Istar-K-CaEb}, so $C \subseteq \clTa{a}$.

\ref{lem:XkCaEb:C-X1CaEb}
Assume first that $C \subseteq \clSc$.
Since $\clChar{1} \cap \clBoth{0}{1} = \clChar{1} \cap \clOdd \cap \clCon{0} = \clS \cap \clCon{0} = \clSc$, we have
\[
C (\clChar{1} \cap \clBoth{0}{1})
\subseteq \clSc \clSc
\subseteq \clSc
= \clChar{1} \cap \clBoth{0}{1}.
\]
Note also that $\clChar{1} \cap \clBoth{1}{0} = \clChar{1} \cap \clOdd \cap \clCon{1} = \clS \cap \clCon{1} = \clS \setminus \clSc$.
For any $f \in \clSc^{(n)}$, $g_1, \dots, g_n \in (\clS \setminus \clSc)^{(m)}$, it holds that $f(g_1, \dots, g_1) \in \clS$ and
\[
f(g_1, \dots, g_n)(0, \dots, 0) = f(g_1(0, \dots, 0), \dots, g_n(0, \dots, 0)) = f(1, \dots, 1) = 1,
\]
so $f(g_1, \dots, g_n) \notin \clSc$, that is, $f(g_1, \dots, g_n) \in \clS \setminus \clSc$.
Consequently, $\clSc (\clS \setminus \clSc) \subseteq \clS \setminus \clSc$, and it follows that
\[
C (\clChar{1} \cap \clBoth{1}{0}) \subseteq \clSc (\clS \setminus \clSc) \subseteq \clS \setminus \clSc = \clChar{1} \cap \clBoth{1}{0}.
\]

Assume now that $C (\clChar{1} \cap \clBoth{a}{b}) \subseteq \clChar{1} \cap \clBoth{a}{b}$.
Then $C$ includes neither $\clIo$, $\clIi$, $\clIstar$, $\clLambdac$, nor $\clVc$ by Lemma~\ref{lem:non}\ref{lem:non:Ia-K-CaEb}, \ref{lem:non:Istar-K-CaEb}, \ref{lem:non:LV-K-Xj}, so $C \subseteq \clSc$.
\end{proof}

\begin{proposition}
\label{prop:Dk}
Let $k \in \IN_{+}$, and let $C$ be a clone.
\begin{enumerate}[label={\upshape{(\roman*)}}, leftmargin=*, widest=ii]
\item\label{lem:Dk:Dk-C}
$\clD{k} C \subseteq \clD{k}$ if and only if $C \subseteq \clL$.
\item\label{lem:Dk:C-Dk}
$C \clD{k} \subseteq \clD{k}$ if and only if $C \subseteq \clL$.
\end{enumerate}
\end{proposition}

\begin{proof}
\ref{lem:Dk:Dk-C}
For sufficiency, it is enough to prove the claim for $C = \clL$.
Using the fact that $\clL = \clonegen{x_1 + x_2, \, 1}$, we apply Lemma~\ref{lem:right-stab-gen}.
It is easy to see that for any function $f \in \clD{k}$, we have $\deg(f \ast (x_1 + x_2)) \leq \deg(f) \leq k$ and $\deg(f \ast 1) \leq \deg(f) \leq k$, so $f \ast (x_1 + x_2), f \ast 1 \in \clD{k}$.
It follows from Lemma~\ref{lem:right-stab-gen} that $\clD{k} \clL \subseteq \clD{k}$.

For necessity, assume that $\clD{k} C \subseteq \clD{k}$.
Then $C$ includes neither $\clLambdac$, $\clVc$, nor $\clSM$ by Lemma~\ref{lem:non}\ref{lem:non:K-LV}, \ref{lem:non:K-SM}, so $C \subseteq \clL$.

\ref{lem:Dk:C-Dk}
For sufficiency, it is enough to prove the claim for $C = \clL$.
Using the fact that $\clL = \clonegen{x_1 + x_2, \, 1}$, we apply Lemma~\ref{lem:left-stab-gen}.
It is clear that for any $g_1, g_2 \in \clD{k}^{(m)}$, the functions
$(x_1 + x_2)(g_1, g_2) = g_1 + g_2$
and
$1(g_1) = 1$
have degree at most $k$, and are therefore members of $\clD{k}$.
It follows from Lemma~\ref{lem:left-stab-gen} that $\clL \clD{k} \subseteq \clD{k}$.

For necessity, assume that $C \clD{k} \subseteq \clD{k}$.
Then $C$ includes neither $\clLambdac$, $\clVc$, nor $\clSM$ by Lemma~\ref{lem:non}\ref{lem:non:LV-K-Xj}, \ref{lem:non:SM-K}, so $C \subseteq \clL$.
\end{proof}

\begin{proposition}
\label{prop:constants}
Let $a \in \{0,1\}$, and let $C$ be a clone.
\begin{enumerate}[label={\upshape{(\roman*)}}, leftmargin=*, widest=iii]
\item\label{lem:constants:D0}
$\clD{0} C \subseteq \clD{0}$
and
$C \clD{0} \subseteq \clD{0}$
for any clone $C$.
\item\label{lem:constants:D0Ca-C}
$(\clD{0} \cap \clCon{a}) C \subseteq \clD{0} \cap \clCon{a}$ for any clone $C$.
\item\label{lem:constants:C-D0Ca}
$C (\clD{0} \cap \clCon{a}) \subseteq \clD{0} \cap \clCon{a}$ if and only if $C \subseteq \clTa{a}$.
\end{enumerate}
\end{proposition}

\begin{proof}
\ref{lem:constants:D0}
Clear, as any composition in which either all inner functions are constant or the outer function is constant is a constant function.

\ref{lem:constants:D0Ca-C}
Clear, as for any $m$\hyp{}ary $g_1, \dots, g_n \in \clAll$ we have $\cf{n}{a}(g_1, \dots, g_n) = \cf{m}{a} \in \clD{0} \cap \clCon{a}$.

\ref{lem:constants:C-D0Ca}
Lemma~\ref{lem:suff-int}, Proposition~\ref{prop:Ca}\ref{lem:Ca:C-Ca}, and part \ref{lem:constants:D0} imply that $C (\clD{0} \cap \clCon{a}) \subseteq \clD{0} \cap \clCon{a}$ whenever $C \subseteq \clAll \cap \clTa{a} = \clTa{a}$.

Assume now that $C \nsubseteq \clTa{a}$.
Then there exists a $g \in C$ that does not preserve $a$, and we have $g(\cf{n}{a}, \dots, \cf{n}{a}) = \cf{n}{1-a} \notin \clCon{a}$.
Therefore $C (\clD{0} \cap \clCon{a}) \nsubseteq \clD{0} \cap \clCon{a}$.
\end{proof}

\begin{proposition}
\label{prop:DkCa}
Let $k \in \IN_{+}$, $a \in \{0,1\}$, and let $C$ be a clone.
\begin{enumerate}[label={\upshape{(\roman*)}}, leftmargin=*, widest=iii]
\item\label{lem:DkCa:DkCa-C}
$(\clD{k} \cap \clCon{a}) C \subseteq (\clD{k} \cap \clCon{a})$ if and only if $C \subseteq \clLo$.
\item\label{lem:DkCa:C-DkCa}
$C (\clD{k} \cap \clCon{a}) \subseteq (\clD{k} \cap \clCon{a})$ if and only if $C \subseteq \clLa{a}$.
\item\label{lem:DkCa:DkEa-C}
$(\clD{k} \cap \clYksi{a}) C \subseteq (\clD{k} \cap \clYksi{a})$ if and only if $C \subseteq \clLi$.
\item\label{lem:DkCa:C-DkEa}
$C (\clD{k} \cap \clYksi{a}) \subseteq (\clD{k} \cap \clYksi{a})$ if and only if $C \subseteq \clLa{a}$.
\end{enumerate}
\end{proposition}

\begin{proof}
\ref{lem:DkCa:DkCa-C}
Lemma~\ref{lem:suff-int} and Propositions~\ref{prop:Ca}\ref{lem:Ca:Ca-C} and \ref{prop:Dk}\ref{lem:Dk:Dk-C} imply that $(\clD{k} \cap \clCon{a}) C \subseteq \clD{k} \cap \clCon{a}$ whenever $C \subseteq \clL \cap \clTo = \clLo$.
Conversely, if $(\clD{k} \cap \clCon{a}) C \subseteq \clD{k} \cap \clCon{a}$, then $C$ includes neither $\clIi$, $\clIstar$, $\clLambdac$, $\clVc$, nor $\clSM$ by Lemma~\ref{lem:non}\ref{lem:non:K-Ia-CaEb}, \ref{lem:non:K-LV}, \ref{lem:non:K-SM}, so $C \subseteq \clLo$.

\ref{lem:DkCa:C-DkCa}
Lemma~\ref{lem:suff-int} and Propositions~\ref{prop:Ca}\ref{lem:Ca:C-Ca} and \ref{prop:Dk}\ref{lem:Dk:C-Dk} imply that $C (\clD{k} \cap \clCon{a}) \subseteq \clD{k} \cap \clCon{a}$ whenever $C \subseteq \clL \cap \clTa{a} = \clLa{a}$.
Conversely, if $C (\clD{k} \cap \clCon{a}) \subseteq \clD{k} \cap \clCon{a}$, then $C$ includes neither $\clIa{\overline{a}}$, $\clIstar$, $\clLambdac$, $\clVc$, nor $\clSM$ by Lemma~\ref{lem:non}\ref{lem:non:Inota-K-CaEb}, \ref{lem:non:Istar-K-CaEb}, \ref{lem:non:LV-K-Xj}, \ref{lem:non:SM-K}, so $C \subseteq \clLa{a}$.

\ref{lem:DkCa:DkEa-C}
Lemma~\ref{lem:suff-int} and Propositions~\ref{prop:Ca}\ref{lem:Ca:Ea-C} and \ref{prop:Dk}\ref{lem:Dk:Dk-C} imply that $(\clD{k} \cap \clYksi{a}) C \subseteq \clD{k} \cap \clYksi{a}$ whenever $C \subseteq \clL \cap \clTi = \clLi$.
Conversely, if $(\clD{k} \cap \clYksi{a}) C \subseteq \clD{k} \cap \clYksi{a}$, then $C$ includes neither $\clIo$, $\clIstar$, $\clLambdac$, $\clVc$, nor $\clSM$ by Lemma~\ref{lem:non}\ref{lem:non:K-Ia-CaEb}, \ref{lem:non:K-LV}, \ref{lem:non:K-SM}, so $C \subseteq \clLi$.

\ref{lem:DkCa:C-DkEa}
Lemma~\ref{lem:suff-int} and Propositions~\ref{prop:Ca}\ref{lem:Ca:C-Ea} and \ref{prop:Dk}\ref{lem:Dk:C-Dk} imply that $C (\clD{k} \cap \clYksi{a}) \subseteq \clD{k} \cap \clYksi{a}$ whenever $C \subseteq \clL \cap \clTa{a} = \clLa{a}$.
Conversely, if $C (\clD{k} \cap \clYksi{a}) \subseteq \clD{k} \cap \clYksi{a}$, then $C$ includes neither $\clIa{\overline{a}}$, $\clIstar$, $\clLambdac$, $\clVc$, nor $\clSM$ by Lemma~\ref{lem:non}\ref{lem:non:Inota-K-CaEb}, \ref{lem:non:Istar-K-CaEb}, \ref{lem:non:LV-K-Xj}, \ref{lem:non:SM-K}, so $C \subseteq \clLa{a}$.
\end{proof}

\begin{proposition}
\label{prop:DkP0}
Let $k \in \IN_{+}$, and let $C$ be a clone.
\begin{enumerate}[label={\upshape{(\roman*)}}, leftmargin=*, widest=iii]
\item\label{lem:DkP0:DkP0-C}
For $k \geq 2$,
$(\clD{k} \cap \clEven) C \subseteq \clD{k} \cap \clEven$ if and only if $C \subseteq \clLc$.
\item\label{lem:DkP0:D1P0-C}
$(\clD{1} \cap \clEven) C \subseteq \clD{1} \cap \clEven$ if and only if $C \subseteq \clLS$.
\item\label{lem:DkP0:C-DkP0}
$C (\clD{k} \cap \clEven) \subseteq \clD{k} \cap \clEven$ if and only if $C \subseteq \clL$.
\end{enumerate}
\end{proposition}

\begin{proof}
\ref{lem:DkP0:DkP0-C}
Lemma~\ref{lem:suff-int} and Propositions~\ref{prop:odd-even}\ref{lem:odd-even:even-C} and \ref{prop:Dk}\ref{lem:Dk:Dk-C} imply that $(\clD{k} \cap \clEven) C \subseteq \clD{k} \cap \clEven$ whenever $C \subseteq \clL \cap \clTc = \clLc$.
Conversely, if $(\clD{k} \cap \clEven) C \subseteq \clD{k} \cap \clEven$, then $C$ includes neither $\clIo$, $\clIi$, $\clIstar$, $\clLambdac$, $\clVc$, nor $\clSM$ by Lemma~\ref{lem:non}\ref{lem:non:K-LV}, \ref{lem:non:K-SM}, \ref{lem:non:K-Ia-Pa}, \ref{lem:non:K-Istar-Pa} so $C \subseteq \clLc$.

\ref{lem:DkP0:D1P0-C}
Assume first that $C \subseteq \clLS$.
Note that $\clLS = \clD{1} \cap \clOdd$ and $\clL = \clD{1}$, and
let $f \in (\clD{1} \cap \clEven)^{(n)}$ and $g_1, \dots, g_n \in (\clD{1} \cap \clOdd)^{(m)}$.
The composition $f(g_1, \dots, g_n)$ is a member of $\clL$ because the outer and inner functions all belong to $\clD{1} = \clL$.
Moreover, it is a sum of an even number of odd polynomials, that is, an even polynomial, so $f(g_1, \dots, g_n) \in \clEven$.
We conclude that $(\clD{1} \cap \clEven) C \subseteq \clD{1} \cap \clEven$.

Assume now that $(\clD{1} \cap \clEven) C \subseteq \clD{1} \cap \clEven$.
Then $C$ includes neither $\clIo$, $\clIi$, $\clLambdac$, $\clVc$, nor $\clSM$ by Lemma~\ref{lem:non}\ref{lem:non:K-LV}, \ref{lem:non:K-SM}, \ref{lem:non:K-Ia-Pa} so $C \subseteq \clLS$.

\ref{lem:DkP0:C-DkP0}
Lemma~\ref{lem:suff-int} and Propositions~\ref{prop:odd-even}\ref{lem:odd-even:C-even} and \ref{prop:Dk}\ref{lem:Dk:C-Dk} imply that $C (\clD{k} \cap \clEven) \subseteq \clD{k} \cap \clEven$ whenever $C \subseteq \clL \cap \clAll = \clL$.
Conversely, if $C (\clD{k} \cap \clEven) \subseteq \clD{k} \cap \clEven$, then $C$ includes neither $\clLambdac$, $\clVc$, nor $\clSM$ by Lemma~\ref{lem:non}\ref{lem:non:LV-K-Xj}, \ref{lem:non:SM-K}, so $C \subseteq \clL$.
\end{proof}

\begin{proposition}
\label{prop:DkP1}
Let $k \in \IN_{+}$, and let $C$ be a clone.
\begin{enumerate}[label={\upshape{(\roman*)}}, leftmargin=*, widest=iii]
\item\label{lem:DkP1:DkP1-C}
For $k \geq 2$,
$(\clD{k} \cap \clOdd) C \subseteq \clD{k} \cap \clOdd$ if and only if $C \subseteq \clLc$.
\item\label{lem:DkP1:D1P1-C}
$(\clD{1} \cap \clOdd) C \subseteq \clD{1} \cap \clOdd$ if and only if $C \subseteq \clLS$.
\item\label{lem:DkP1:C-DkP1}
$C (\clD{k} \cap \clOdd) \subseteq \clD{k} \cap \clOdd$ if and only if $C \subseteq \clLS$.
\end{enumerate}
\end{proposition}

\begin{proof}
\ref{lem:DkP1:DkP1-C}
Lemma~\ref{lem:suff-int} and Propositions~\ref{prop:odd-even}\ref{lem:odd-even:odd-C} and \ref{prop:Dk}\ref{lem:Dk:Dk-C} imply that $(\clD{k} \cap \clOdd) C \subseteq \clD{k} \cap \clOdd$ whenever $C \subseteq \clL \cap \clTc = \clLc$.
Conversely, if $(\clD{k} \cap \clOdd) C \subseteq \clD{k} \cap \clOdd$, then $C$ includes neither $\clIo$, $\clIi$, $\clIstar$, $\clLambdac$, $\clVc$, nor $\clSM$ by Lemma~\ref{lem:non}\ref{lem:non:K-LV}, \ref{lem:non:K-SM}, \ref{lem:non:K-Ia-Pa}, \ref{lem:non:K-Istar-Pa}, so $C \subseteq \clLc$.

\ref{lem:DkP1:D1P1-C}
If $C \subseteq \clLS$, then,
since $\clD{1} \cap \clOdd = \clLS$ and $\clLS$ is a clone, it holds that
$(\clD{1} \cap \clOdd) C \subseteq \clLS \, \clLS \subseteq \clLS = \clD{1} \cap \clOdd$.
Conversely, if $(\clD{1} \cap \clOdd) C \subseteq \clD{1} \cap \clOdd$,
then $C$ includes neither $\clIo$, $\clIi$, $\clLambdac$, $\clVc$, nor $\clSM$ by Lemma~\ref{lem:non}\ref{lem:non:K-LV}, \ref{lem:non:K-SM}, \ref{lem:non:K-Ia-Pa}, so $C \subseteq \clLS$.

\ref{lem:DkP1:C-DkP1}
Lemma~\ref{lem:suff-int} and Propositions~\ref{prop:odd-even}\ref{lem:odd-even:C-odd} and \ref{prop:Dk}\ref{lem:Dk:C-Dk} imply that $C (\clD{k} \cap \clOdd) \subseteq \clD{k} \cap \clOdd$ whenever $C \subseteq \clL \cap \clS = \clLS$.
Conversely, if $C (\clD{k} \cap \clOdd) \subseteq \clD{k} \cap \clOdd$, then $C$ includes neither $\clIo$, $\clIi$, $\clLambdac$, $\clVc$, nor $\clSM$ by Lemma~\ref{lem:non}\ref{lem:non:Ia-K-P1}, \ref{lem:non:LV-K-Xj}, \ref{lem:non:SM-K}, so $C \subseteq \clLS$.
\end{proof}

\begin{proposition}
\label{prop:DkCaEb}
Let $k \in \IN_{+}$, $a, b \in \{0, 1\}$, and let $C$ be a clone.
\begin{enumerate}[label={\upshape{(\roman*)}}, leftmargin=*, widest=ii]
\item\label{lem:DkCaEb:DkCaEb-C}
$(\clD{k} \cap \clBoth{a}{b}) C \subseteq \clD{k} \cap \clBoth{a}{b}$ if and only if $C \subseteq \clLc$.
\item\label{lem:DkCaEb:C-DkCaEb}
$C (\clD{k} \cap \clBoth{a}{b}) \subseteq \clD{k} \cap \clBoth{a}{b}$ if and only if $C \subseteq \clLa{a} \cap \clLa{b}$.
\end{enumerate}
\end{proposition}

\begin{proof}
\ref{lem:DkCaEb:DkCaEb-C}
Lemma~\ref{lem:suff-int} and Propositions~\ref{prop:CaEb}\ref{lem:CaEb:CaEb-C} and \ref{prop:Dk}\ref{lem:Dk:Dk-C} imply that $(\clD{k} \cap \clBoth{a}{b}) C \subseteq \clD{k} \cap \clBoth{a}{b}$ whenever $C \subseteq \clL \cap \clTc = \clLc$.
Conversely, if $(\clD{k} \cap \clBoth{a}{b}) C \subseteq \clD{k} \cap \clBoth{a}{b}$, then $C$ includes neither $\clIo$, $\clIi$, $\clIstar$, $\clLambdac$, $\clVc$, nor $\clSM$ by Lemma~\ref{lem:non}\ref{lem:non:K-Ia-CaEb}, \ref{lem:non:K-LV}, \ref{lem:non:K-SM}, so $C \subseteq \clLc$.

\ref{lem:DkCaEb:C-DkCaEb}
Lemma~\ref{lem:suff-int} and Propositions~\ref{prop:CaEb}\ref{lem:CaEb:C-CaEb} and \ref{prop:Dk}\ref{lem:Dk:C-Dk} imply that $C (\clD{k} \cap \clBoth{a}{b}) \subseteq \clD{k} \cap \clBoth{a}{b}$ whenever $C \subseteq \clL \cap \clTa{a} \cap \clTa{b} = \clLa{a} \cap \clLa{b}$.

Assume now that $C (\clD{k} \cap \clBoth{a}{b}) \subseteq \clD{k} \cap \clBoth{a}{b}$.
Then $C$ includes neither $\clIstar$, $\clLambdac$, $\clVc$, nor $\clSM$ by Lemma~\ref{lem:non}\ref{lem:non:Istar-K-CaEb}, \ref{lem:non:LV-K-Xj}, \ref{lem:non:SM-K}.
If $a = b$, then $C$ does not include $\clIa{\overline{a}}$ by Lemma~\ref{lem:non}\ref{lem:non:Inota-K-CaEb}, so $C \subseteq \clLa{a} = \clLa{a} \cap \clLa{b}$.
If $a \neq b$, then $C$ includes neither $\clIo$ nor $\clIi$ by Lemma~\ref{lem:non}\ref{lem:non:Ia-K-CaEb}, so $C \subseteq \clLc = \clLa{a} \cap \clLa{b}$.
\end{proof}

\begin{proposition}
\label{prop:DiXj}
Let $i, j \in \IN_{+}$ with $i > j \geq 1$, and let $C$ be a clone.
\begin{enumerate}[label={\upshape{(\roman*)}}, leftmargin=*, widest=ii]
\item\label{lem:DiXj:DiXj-C}
$(\clD{i} \cap \clChar{j}) C \subseteq \clD{i} \cap \clChar{j}$ if and only if $C \subseteq \clLS$.
\item\label{lem:DiXj:C-DiXj}
$C (\clD{i} \cap \clChar{j}) \subseteq \clD{i} \cap \clChar{j}$ if and only if $C \subseteq \clL$.
\end{enumerate}
\end{proposition}

\begin{proof}
\ref{lem:DiXj:DiXj-C}
Lemma~\ref{lem:suff-int} and Propositions~\ref{prop:Xk}\ref{lem:Xk:Xk-C}, \ref{lem:Xk:X1-C} and \ref{prop:Dk}\ref{lem:Dk:Dk-C} imply $(\clD{i} \cap \clChar{j}) C \subseteq \clD{i} \cap \clChar{j}$ whenever $C \subseteq \clLS \cap \clL = \clLS$ if $k \geq 2$ and whenever $C \subseteq \clS \cap \clL = \clLS$ if $k = 1$.
Conversely, if $(\clD{i} \cap \clChar{j}) C \subseteq \clD{i} \cap \clChar{j}$, then $C$ includes neither $\clIo$, $\clIi$, $\clLambdac$, $\clVc$, nor $\clSM$ by Lemma~\ref{lem:non}\ref{lem:non:K-Ia-Xj}, \ref{lem:non:K-LV}, \ref{lem:non:K-SM}, so $C \subseteq \clLS$.

\ref{lem:DiXj:C-DiXj}
Lemma~\ref{lem:suff-int} and Propositions~\ref{prop:Xk}\ref{lem:Xk:C-Xk} and \ref{prop:Dk}\ref{lem:Dk:C-Dk} imply that $C (\clD{i} \cap \clChar{j}) \subseteq \clD{i} \cap \clChar{j}$ whenever $C \subseteq \clL \cap \clL = \clL$.
Conversely, if $C (\clD{i} \cap \clChar{j}) \subseteq \clD{i} \cap \clChar{j}$, then $C$ includes neither $\clLambdac$, $\clVc$, nor $\clSM$ by Lemma~\ref{lem:non}\ref{lem:non:LV-K-Xj}, \ref{lem:non:SM-K}, so $C \subseteq \clL$.
\end{proof}

\begin{proposition}
\label{prop:DiXjCa}
Let $i, j \in \IN_{+}$ with $i > j \geq 1$, $a \in \{0,1\}$, and let $C$ be a clone.
\begin{enumerate}[label={\upshape{(\roman*)}}, leftmargin=*, widest=iii]
\item\label{lem:DiXjCa:DiXjCa-C}
$(\clD{i} \cap \clChar{j} \cap \clCon{a}) C \subseteq \clD{i} \cap \clChar{j} \cap \clCon{a}$ if and only if $C \subseteq \clLc$.
\item\label{lem:DiXjCa:C-DiXjCa}
$C (\clD{i} \cap \clChar{j} \cap \clCon{a}) \subseteq \clD{i} \cap \clChar{j} \cap \clCon{a}$ if and only if $C \subseteq \clLa{a}$.
\item\label{lem:DiXjCa:DiXjEa-C}
$(\clD{i} \cap \clChar{j} \cap \clYksi{a}) C \subseteq \clD{i} \cap \clChar{j} \cap \clYksi{a}$ if and only if $C \subseteq \clLc$.
\item\label{lem:DiXjCa:C-DiXjEa}
$C (\clD{i} \cap \clChar{j} \cap \clYksi{a}) \subseteq \clD{i} \cap \clChar{j} \cap \clYksi{a}$ if and only if $C \subseteq \clLa{a}$.
\end{enumerate}
\end{proposition}

\begin{proof}
\ref{lem:DiXjCa:DiXjCa-C}
Lemma~\ref{lem:suff-int} and Propositions~\ref{prop:Ca}\ref{lem:Ca:Ca-C} and \ref{prop:DiXj}\ref{lem:DiXj:DiXj-C} imply that $(\clD{i} \cap \clChar{j} \cap \clCon{a}) C \subseteq \clD{i} \cap \clChar{j} \cap \clCon{a}$ whenever $C \subseteq \clLS \cap \clTo = \clLc$.
Conversely, if $(\clD{i} \cap \clChar{j} \cap \clCon{a}) C \subseteq \clD{i} \cap \clChar{j} \cap \clCon{a}$, then $C$ includes neither $\clIo$, $\clIi$, $\clIstar$, $\clLambdac$, $\clVc$, nor $\clSM$ by Lemma~\ref{lem:non}\ref{lem:non:K-Ia-CaEb}, \ref{lem:non:K-Ia-Xj}, \ref{lem:non:K-LV}, \ref{lem:non:K-SM}, so $C \subseteq \clLc$.

\ref{lem:DiXjCa:C-DiXjCa}
Lemma~\ref{lem:suff-int} and Propositions~\ref{prop:Ca}\ref{lem:Ca:C-Ca}, \ref{prop:DiXj}\ref{lem:DiXj:C-DiXj} imply $C (\clD{i} \cap \clChar{j} \cap \clCon{a}) \subseteq \clD{i} \cap \clChar{j} \cap \clCon{a}$ whenever $C \subseteq \clL \cap \clTa{a} = \clLa{a}$.
Conversely, if $C (\clD{i} \cap \clChar{j} \cap \clCon{a}) \subseteq \clD{i} \cap \clChar{j} \cap \clCon{a}$, then $C$ includes neither $\clIa{\overline{a}}$, $\clIstar$, $\clLambdac$, $\clVc$, nor $\clSM$ by Lemma~\ref{lem:non}\ref{lem:non:Inota-K-CaEb}, \ref{lem:non:Istar-K-CaEb}, \ref{lem:non:LV-K-Xj}, \ref{lem:non:SM-K}, so $C \subseteq \clLa{a}$.

\ref{lem:DiXjCa:DiXjEa-C}
Lemma~\ref{lem:suff-int} and Propositions~\ref{prop:Ca}\ref{lem:Ca:Ea-C} and \ref{prop:DiXj}\ref{lem:DiXj:DiXj-C} imply that $(\clD{i} \cap \clChar{j} \cap \clYksi{a}) C \subseteq \clD{i} \cap \clChar{j} \cap \clYksi{a}$ whenever $C \subseteq \clLS \cap \clTi = \clLc$.
Conversely, if $(\clD{i} \cap \clChar{j} \cap \clYksi{a}) C \subseteq \clD{i} \cap \clChar{j} \cap \clYksi{a}$, then $C$ includes neither $\clIo$, $\clIi$, $\clIstar$, $\clLambdac$, $\clVc$, nor $\clSM$ by Lemma~\ref{lem:non}\ref{lem:non:K-Ia-CaEb}, \ref{lem:non:K-Ia-Xj}, \ref{lem:non:K-LV}, \ref{lem:non:K-SM}, so $C \subseteq \clLc$.

\ref{lem:DiXjCa:C-DiXjEa}
Lemma~\ref{lem:suff-int} and Propositions~\ref{prop:Ca}\ref{lem:Ca:C-Ea} and \ref{prop:DiXj}\ref{lem:DiXj:C-DiXj} imply that $C (\clD{i} \cap \clChar{j} \cap \clYksi{a}) \subseteq \clD{i} \cap \clChar{j} \cap \clYksi{a}$ whenever $C \subseteq \clL \cap \clTa{a} = \clLa{a}$.
Conversely, if $C (\clD{i} \cap \clChar{j} \cap \clYksi{a}) \subseteq \clD{i} \cap \clChar{j} \cap \clYksi{a}$, then $C$ includes neither $\clIa{\overline{a}}$, $\clIstar$, $\clLambdac$, $\clVc$, nor $\clSM$ by Lemma~\ref{lem:non}\ref{lem:non:Inota-K-CaEb}, \ref{lem:non:Istar-K-CaEb}, \ref{lem:non:LV-K-Xj}, \ref{lem:non:SM-K}, so $C \subseteq \clLa{a}$.
\end{proof}

\begin{proposition}
\label{prop:DiXjP0}
Let $i, j \in \IN_{+}$ with $i > j \geq 1$, and let $C$ be a clone.
\begin{enumerate}[label={\upshape{(\roman*)}}, leftmargin=*, widest=iii]
\item\label{lem:DiXjP0:DiXjP0-C}
For $j \geq 2$,
$(\clD{i} \cap \clChar{j} \cap \clEven) C \subseteq \clD{i} \cap \clChar{j} \cap \clEven$ if and only if $C \subseteq \clLc$.
\item\label{lem:DiXjP0:DiX1P0-C}
$(\clD{i} \cap \clChar{1} \cap \clEven) C \subseteq \clD{i} \cap \clChar{1} \cap \clEven$ if and only if $C \subseteq \clLS$.
\item\label{lem:DiXjP0:C-DiXjP0}
$C (\clD{i} \cap \clChar{j} \cap \clEven) \subseteq \clD{i} \cap \clChar{j} \cap \clEven$ if and only if $C \subseteq \clL$.
\end{enumerate}
\end{proposition}

\begin{proof}
\ref{lem:DiXjP0:DiXjP0-C}
Lemma~\ref{lem:suff-int} and Propositions~\ref{prop:odd-even}\ref{lem:odd-even:even-C} and \ref{prop:DiXj}\ref{lem:DiXj:DiXj-C} imply that $(\clD{i} \cap \clChar{j} \cap \clEven) C \subseteq \clD{i} \cap \clChar{j} \cap \clEven$ whenever $C \subseteq \clLS \cap \clTc = \clLc$.
Conversely, if $(\clD{i} \cap \clChar{j} \cap \clEven) C \subseteq \clD{i} \cap \clChar{j} \cap \clEven$, then $C$ includes neither $\clIo$, $\clIi$, $\clIstar$, $\clLambdac$, $\clVc$, nor $\clSM$ by Lemma~\ref{lem:non}\ref{lem:non:K-LV}, \ref{lem:non:K-SM}, \ref{lem:non:K-Ia-Pa}, \ref{lem:non:K-Istar-Pa}, so $C \subseteq \clLc$.

\ref{lem:DiXjP0:DiX1P0-C}
Lemma~\ref{lem:suff-int} and Propositions~\ref{prop:XkP0}\ref{lem:XkP0:X1P0-C} and \ref{prop:Dk}\ref{lem:Dk:Dk-C} imply that $(\clD{i} \cap \clChar{1} \cap \clEven) C \subseteq \clD{i} \cap \clChar{1} \cap \clEven$ whenever $C \subseteq \clS \cap \clL = \clLS$.
Conversely, if $(\clD{i} \cap \clChar{1} \cap \clEven) C \subseteq \clD{i} \cap \clChar{1} \cap \clEven$, then $C$ includes neither $\clIo$, $\clIi$, $\clLambdac$, $\clVc$, nor $\clSM$ by Lemma~\ref{lem:non}\ref{lem:non:K-LV}, \ref{lem:non:K-SM}, \ref{lem:non:K-Ia-Pa}, so $C \subseteq \clLS$.

\ref{lem:DiXjP0:C-DiXjP0}
Lemma~\ref{lem:suff-int} and Propositions~\ref{prop:odd-even}\ref{lem:odd-even:C-even} and \ref{prop:DiXj}\ref{lem:DiXj:C-DiXj} imply that $C (\clD{i} \cap \clChar{j} \cap \clEven) \subseteq \clD{i} \cap \clChar{j} \cap \clEven$ whenever $C \subseteq \clL \cap \clAll = \clL$.
Conversely, if $C (\clD{i} \cap \clChar{j} \cap \clEven) \subseteq \clD{i} \cap \clChar{j} \cap \clEven$, then $C$ includes neither $\clLambdac$, $\clVc$, nor $\clSM$ by Lemma~\ref{lem:non}\ref{lem:non:LV-K-Xj}, \ref{lem:non:SM-K}, so $C \subseteq \clL$.
\end{proof}

\begin{proposition}
\label{prop:DiXjP1}
Let $i, j \in \IN_{+}$ with $i > j \geq 1$, and let $C$ be a clone.
\begin{enumerate}[label={\upshape{(\roman*)}}, leftmargin=*, widest=iii]
\item\label{lem:DiXjP1:DiXjP1-C}
For $j \geq 2$,
$(\clD{i} \cap \clChar{j} \cap \clOdd) C \subseteq \clD{i} \cap \clChar{j} \cap \clOdd$ if and only if $C \subseteq \clLc$.
\item\label{lem:DiXjP1:DiX1P1-C}
$(\clD{i} \cap \clChar{1} \cap \clOdd) C \subseteq \clD{i} \cap \clChar{1} \cap \clOdd$ if and only if $C \subseteq \clLS$.
\item\label{lem:DiXjP1:C-DiXjP1}
$C (\clD{i} \cap \clChar{j} \cap \clOdd) \subseteq \clD{i} \cap \clChar{j} \cap \clOdd$ if and only if $C \subseteq \clLS$.
\end{enumerate}
\end{proposition}

\begin{proof}
\ref{lem:DiXjP1:DiXjP1-C}
Lemma~\ref{lem:suff-int} and Propositions~\ref{prop:odd-even}\ref{lem:odd-even:odd-C} and \ref{prop:DiXj}\ref{lem:DiXj:DiXj-C} imply that $(\clD{i} \cap \clChar{j} \cap \clOdd) C \subseteq \clD{i} \cap \clChar{j} \cap \clOdd$ whenever $C \subseteq \clLS \cap \clTc = \clLc$.
Conversely, if $(\clD{i} \cap \clChar{j} \cap \clOdd) C \subseteq \clD{i} \cap \clChar{j} \cap \clOdd$, then $C$ includes neither $\clIo$, $\clIi$, $\clIstar$, $\clLambdac$, $\clVc$, nor $\clSM$ by Lemma~\ref{lem:non}\ref{lem:non:K-LV}, \ref{lem:non:K-SM}, \ref{lem:non:K-Ia-Pa}, \ref{lem:non:K-Istar-Pa}, so $C \subseteq \clLc$.

\ref{lem:DiXjP1:DiX1P1-C}
Lemma~\ref{lem:suff-int} and Propositions~\ref{prop:XkP1}\ref{lem:XkP1:X1P1-C} and \ref{prop:Dk}\ref{lem:Dk:Dk-C} imply that $(\clD{i} \cap \clChar{1} \cap \clOdd) C \subseteq \clD{i} \cap \clChar{1} \cap \clOdd$ whenever $C \subseteq \clS \cap \clL = \clLS$.
Conversely, if $(\clD{i} \cap \clChar{1} \cap \clOdd) C \subseteq \clD{i} \cap \clChar{1} \cap \clOdd$, then $C$ includes neither $\clIo$, $\clIi$, $\clLambdac$, $\clVc$, nor $\clSM$ by Lemma~\ref{lem:non}\ref{lem:non:K-LV}, \ref{lem:non:K-SM}, \ref{lem:non:K-Ia-Pa}, so $C \subseteq \clLS$.

\ref{lem:DiXjP1:C-DiXjP1}
Lemma~\ref{lem:suff-int} and Propositions~\ref{prop:odd-even}\ref{lem:odd-even:C-odd} and \ref{prop:DiXj}\ref{lem:DiXj:C-DiXj} imply that $C (\clD{i} \cap \clChar{j} \cap \clOdd) \subseteq \clD{i} \cap \clChar{j} \cap \clOdd$ whenever $C \subseteq \clL \cap \clS = \clLS$.
Conversely, if $C (\clD{i} \cap \clChar{j} \cap \clOdd) \subseteq \clD{i} \cap \clChar{j} \cap \clOdd$, then $C$ includes neither $\clIo$, $\clIi$, $\clLambdac$, $\clVc$, nor $\clSM$ by Lemma~\ref{lem:non}\ref{lem:non:Ia-K-P1}, \ref{lem:non:LV-K-Xj}, \ref{lem:non:SM-K}, so $C \subseteq \clLS$.
\end{proof}

\begin{proposition}
\label{prop:DiXjCaEb}
Let $i, j \in \IN_{+}$ with $i > j \geq 1$, $a, b \in \{0,1\}$, and let $C$ be a clone.
\begin{enumerate}[label={\upshape{(\roman*)}}, leftmargin=*, widest=ii]
\item\label{lem:DiXjCaEb:DiXjCaEb-C}
$(\clD{i} \cap \clChar{j} \cap \clBoth{a}{b}) C \subseteq \clD{i} \cap \clChar{j} \cap \clBoth{a}{b}$ if and only if $C \subseteq \clLc$.
\item\label{lem:DiXjCaEb:C-DiXjCaEb}
$C (\clD{i} \cap \clChar{j} \cap \clBoth{a}{b}) \subseteq \clD{i} \cap \clChar{j} \cap \clBoth{a}{b}$ if and only if $C \subseteq \clLa{a} \cap \clLa{b}$.
\end{enumerate}
\end{proposition}

\begin{proof}
\ref{lem:DiXjCaEb:DiXjCaEb-C}
Lemma~\ref{lem:suff-int} and Propositions~\ref{prop:Xk}\ref{lem:Xk:Xk-C}, \ref{lem:Xk:X1-C} and \ref{prop:DkCaEb}\ref{lem:DkCaEb:DkCaEb-C} imply $(\clD{i} \cap \clChar{j} \cap \clBoth{a}{b}) C \subseteq \clD{i} \cap \clChar{j} \cap \clBoth{a}{b}$ whenever $C \subseteq \clLc \cap \clLS = \clLc$ if $j \geq 2$ and whenever $C \subseteq \clLc \cap \clS = \clLc$ if $j = 1$.
Conversely, if $(\clD{i} \cap \clChar{j} \cap \clBoth{a}{b}) C \subseteq \clD{i} \cap \clChar{j} \cap \clBoth{a}{b}$, then $C$ includes neither $\clIo$, $\clIi$, $\clIstar$, $\clLambdac$, $\clVc$, nor $\clSM$ by Lemma~\ref{lem:non}\ref{lem:non:K-Ia-CaEb}, \ref{lem:non:K-LV}, \ref{lem:non:K-SM}, so $C \subseteq \clLc$.

\ref{lem:DiXjCaEb:C-DiXjCaEb}
Lemma~\ref{lem:suff-int} and Propositions~\ref{prop:Xk}\ref{lem:Xk:C-Xk} and \ref{prop:DkCaEb}\ref{lem:DkCaEb:C-DkCaEb} imply that $C (\clD{i} \cap \clChar{j} \cap \clBoth{a}{b}) \subseteq \clD{i} \cap \clChar{j} \cap \clBoth{a}{b}$ whenever $C \subseteq \clLa{a} \cap \clLa{b} \cap \clL = \clLa{a} \cap \clLa{b}$.

Assume now that $C (\clD{i} \cap \clChar{j} \cap \clBoth{a}{b}) \subseteq \clD{i} \cap \clChar{j} \cap \clBoth{a}{b}$.
If $a = b$, then $C$ includes neither $\clIa{\overline{a}}$, $\clIstar$, $\clLambdac$, $\clVc$, nor $\clSM$ by Lemma~\ref{lem:non}\ref{lem:non:Inota-K-CaEb}, \ref{lem:non:Istar-K-CaEb}, \ref{lem:non:LV-K-Xj}, \ref{lem:non:SM-K}, so $C \subseteq \clLa{a} = \clLa{a} \cap \clLa{b}$.
If $a \neq b$, then $C$ includes neither $\clIo$, $\clIi$, $\clIstar$, $\clLambdac$, $\clVc$, nor $\clSM$ by Lemma~\ref{lem:non}\ref{lem:non:Ia-K-CaEb}, \ref{lem:non:Istar-K-CaEb}, \ref{lem:non:LV-K-Xj}, \ref{lem:non:SM-K}, so $C \subseteq \clLc = \clLa{a} \cap \clLa{b}$.
\end{proof}

\begin{proof}[Proof of Theorem~\ref{thm:C1C2-stability}]
The theorem puts together Propositions
\ref{prop:all-empty},
\ref{prop:Ca},
\ref{prop:odd-even},
\ref{prop:CaEb},
\ref{prop:Xk},
\ref{prop:XkCa},
\ref{prop:XkEa},
\ref{prop:XkP0},
\ref{prop:XkP1},
\ref{prop:XkCaEb},
\ref{prop:Dk},
\ref{prop:constants},
\ref{prop:DkCa},
\ref{prop:DkP0},
\ref{prop:DkP1},
\ref{prop:DkCaEb},
\ref{prop:DiXj},
\ref{prop:DiXjCa},
\ref{prop:DiXjP0},
\ref{prop:DiXjP1},
\ref{prop:DiXjCaEb}.
\end{proof}

With the help of Post's lattice (Figure~\ref{fig:Post}) and by reading off from Table~\ref{table:stability}, we can determine for any pair $(C_1,C_2)$ of clones which $\clLc$\hyp{}stable classes are $(C_1,C_2)$\hyp{}stable.
If $\clLc \subseteq C_2$, then any $(C_1,C_2)$\hyp{}stable class is $(\clIc,\clLc)$\hyp{}stable by Lemma~\ref{lem:stable-impl-stable} and hence also $\clLc$\hyp{}stable by Lemma~\ref{lem:Lc-simplify}.
Therefore, in the case when $\clLc \subseteq C_2$, the $(C_1,C_2)$\hyp{}stable classes are among the $\clLc$\hyp{}stable ones and they can be easily picked out from Table~\ref{table:stability}.
In particular, we have an explicit description of $(\clIc, C)$\hyp{}stable classes (``clonoids'' of Aichinger and Mayr \cite{AicMay-2016}) and $C$\hyp{}stable classes for $\clLc \subseteq C$.
The $\clLo$\hyp{}stable classes (see Corollary~\ref{cor:C-stable}\ref{cor:C-stable:L0}) were determined earlier by Kreinecker~\cite[Theorem~3.12]{Kreinecker-2019}.

\begin{corollary}
\label{cor:IcC-stable}
\leavevmode
\begin{enumerate}[label={\upshape{(\roman*)}}, leftmargin=*, widest=viii]
\item\label{cor:IcC-stable:Lc}
The $(\clIc, \clLc)$\hyp{}stable classes are
$\clAll$, $\clCon{a}$, $\clYksi{a}$, $\clParity{a}$, $\clBoth{a}{b}$, $\clD{k}$, $\clD{k} \cap \clCon{a}$, $\clD{k} \cap \clYksi{a}$, $\clD{k} \cap \clParity{a}$, $\clD{k} \cap \clBoth{a}{b}$, $\clChar{k}$, $\clChar{k} \cap \clCon{a}$, $\clChar{k} \cap \clYksi{a}$, $\clChar{k} \cap \clParity{a}$, $\clChar{k} \cap \clBoth{a}{b}$, $\clD{i} \cap \clChar{j}$, $\clD{i} \cap \clChar{j} \cap \clCon{a}$, $\clD{i} \cap \clChar{j} \cap \clYksi{a}$, $\clD{i} \cap \clChar{j} \cap \clParity{a}$, $\clD{i} \cap \clChar{j} \cap \clBoth{a}{b}$, $\clD{0}$, $\clD{0} \cap \clCon{a}$, $\clEmpty$, for $a, b \in \{0,1\}$, $\QuantifyParRel$, and $i, j, k \in \IN_{+}$ with $i > j \geq 1$.
\item
The $(\clIc, \clLS)$\hyp{}stable classes are
$\clAll$, $\clParity{a}$, $\clChar{k}$, $\clChar{k} \cap \clParity{a}$, $\clD{k}$, $\clD{k} \cap \clParity{a}$,
$\clD{i} \cap \clChar{j}$, $\clD{i} \cap \clChar{j} \cap \clParity{a}$,
$\clD{0}$, $\clEmpty$,
for $\QuantifyParRel$, and $i, j, k \in \IN_{+}$ with $i > j \geq 1$.
\item
The $(\clIc, \clLo)$\hyp{}stable classes are
$\clAll$, $\clCon{0}$, $\clYksi{0}$, $\clEven$, $\clBoth{0}{0}$, $\clChar{k}$, $\clChar{k} \cap \clCon{0}$, $\clChar{k} \cap \clYksi{0}$, $\clChar{k} \cap \clEven$, $\clChar{k} \cap \clBoth{0}{0}$, $\clD{k}$, $\clD{k} \cap \clCon{0}$, $\clD{k} \cap \clYksi{0}$, $\clD{k} \cap \clEven$, $\clD{k} \cap \clBoth{0}{0}$, $\clD{i} \cap \clChar{j}$, $\clD{i} \cap \clChar{j} \cap \clCon{0}$, $\clD{i} \cap \clChar{j} \cap \clYksi{0}$, $\clD{i} \cap \clChar{j} \cap \clEven$, $\clD{i} \cap \clChar{j} \cap \clBoth{0}{0}$, $\clD{0}$, $\clD{0} \cap \clCon{0}$, $\clEmpty$,
for $k \in \IN_{+}$.
\item
The $(\clIc, \clLi)$\hyp{}stable classes are
$\clAll$, $\clCon{1}$, $\clYksi{1}$, $\clEven$, $\clBoth{1}{1}$, $\clChar{k}$, $\clChar{k} \cap \clCon{1}$, $\clChar{k} \cap \clYksi{1}$, $\clChar{k} \cap \clEven$, $\clChar{k} \cap \clBoth{1}{1}$, $\clD{k}$, $\clD{k} \cap \clCon{1}$, $\clD{k} \cap \clYksi{1}$, $\clD{k} \cap \clEven$, $\clD{k} \cap \clBoth{1}{1}$, $\clD{i} \cap \clChar{j}$, $\clD{i} \cap \clChar{j} \cap \clCon{1}$, $\clD{i} \cap \clChar{j} \cap \clYksi{1}$, $\clD{i} \cap \clChar{j} \cap \clEven$, $\clD{i} \cap \clChar{j} \cap \clBoth{1}{1}$, $\clD{0}$, $\clD{0} \cap \clCon{1}$, $\clEmpty$,
for $k \in \IN_{+}$.
\item
The $(\clIc, \clL)$\hyp{}stable classes are
$\clAll$, $\clEven$, $\clChar{k}$, $\clChar{k} \cap \clEven$, $\clD{k}$, $\clD{k} \cap \clEven$, $\clD{i} \cap \clChar{j}$, $\clD{i} \cap \clChar{j} \cap \clEven$, $\clD{0}$, $\clEmpty$, for $k \in \IN_{+}$.
\item\label{cor:IcC-stable:Sc}
The $(\clIc, \clSc)$\hyp{}stable classes are
$\clAll$, $\clCon{a}$, $\clYksi{a}$, $\clParity{a}$, $\clBoth{a}{b}$, $\clChar{1} \cap \clParity{a}$, $\clChar{1} \cap \clBoth{a}{b}$, $\clD{0}$, $\clD{0} \cap \clCon{a}$, $\clEmpty$,
for $a, b \in \{0,1\}$ and $\QuantifyParRel$.
\item
The $(\clIc, \clS)$\hyp{}stable classes are
$\clAll$, $\clParity{a}$, $\clChar{1} \cap \clParity{a}$, $\clD{0}$, $\clEmpty$,
for $\QuantifyParRel$.
\item
The $(\clIc, \clTc)$\hyp{}stable classes are
$\clAll$, $\clCon{a}$, $\clYksi{a}$, $\clEven$, $\clBoth{a}{b}$, $\clChar{1} \cap \clEven$, $\clChar{1} \cap \clBoth{a}{a}$, $\clD{0}$, $\clD{0} \cap \clCon{a}$, $\clEmpty$,
for $a, b \in \{0,1\}$.
\item
The $(\clIc, \clTo)$\hyp{}stable classes are
$\clAll$, $\clCon{0}$, $\clYksi{0}$, $\clEven$, $\clBoth{0}{0}$, $\clChar{1} \cap \clEven$, $\clChar{1} \cap \clBoth{0}{0}$, $\clD{0}$, $\clD{0} \cap \clCon{0}$, $\clEmpty$.
\item
The $(\clIc, \clTi)$\hyp{}stable classes are
$\clAll$, $\clCon{1}$, $\clYksi{1}$, $\clEven$, $\clBoth{1}{1}$, $\clChar{1} \cap \clEven$, $\clChar{1} \cap \clBoth{1}{1}$, $\clD{0}$, $\clD{0} \cap \clCon{1}$, $\clEmpty$.
\item
The $(\clIc, \clAll)$\hyp{}stable classes are
$\clAll$, $\clEven$, $\clChar{1} \cap \clEven$, $\clD{0}$, $\clEmpty$.
\end{enumerate}
\end{corollary}

\begin{corollary}
\label{cor:C-stable}
\leavevmode
\begin{enumerate}[label={\upshape{(\roman*)}}, leftmargin=*, widest=viii]
\item\label{cor:C-stable:Lc}
The $\clLc$\hyp{}stable classes are
$\clAll$, $\clCon{a}$, $\clYksi{a}$, $\clParity{a}$, $\clBoth{a}{b}$, $\clD{k}$, $\clD{k} \cap \clCon{a}$, $\clD{k} \cap \clYksi{a}$, $\clD{k} \cap \clParity{a}$, $\clD{k} \cap \clBoth{a}{b}$, $\clChar{k}$, $\clChar{k} \cap \clCon{a}$, $\clChar{k} \cap \clYksi{a}$, $\clChar{k} \cap \clParity{a}$, $\clChar{k} \cap \clBoth{a}{b}$, $\clD{i} \cap \clChar{j}$, $\clD{i} \cap \clChar{j} \cap \clCon{a}$, $\clD{i} \cap \clChar{j} \cap \clYksi{a}$, $\clD{i} \cap \clChar{j} \cap \clParity{a}$, $\clD{i} \cap \clChar{j} \cap \clBoth{a}{b}$, $\clD{0}$, $\clD{0} \cap \clCon{a}$, $\clEmpty$, for $a, b \in \{0,1\}$, $\QuantifyParRel$, and $i, j, k \in \IN_{+}$ with $i > j \geq 1$.
\item
The $\clLS$\hyp{}stable classes are
$\clAll$, $\clChar{k}$, $\clChar{1} \cap \clParity{a}$, $\clD{k}$, $\clD{1} \cap \clParity{a}$,
$\clD{i} \cap \clChar{j}$, $\clD{i} \cap \clChar{1} \cap \clParity{a}$,
$\clD{0}$, $\clEmpty$,
for $\QuantifyParRel$ and $i, j, k \in \IN_{+}$ with $i > j \geq 1$.
\item\label{cor:C-stable:L0}
The $\clLo$\hyp{}stable classes are
$\clAll$, $\clCon{0}$, $\clD{k}$, $\clD{k} \cap \clCon{0}$, $\clD{0}$, $\clD{0} \cap \clCon{0}$, $\clEmpty$,
for $k \in \IN_{+}$.
\item
The $\clLi$\hyp{}stable classes are
$\clAll$, $\clYksi{1}$, $\clD{k}$, $\clD{k} \cap \clYksi{1}$, $\clD{0}$, $\clD{0} \cap \clCon{1}$, $\clEmpty$,
for $k \in \IN_{+}$.
\item
The $\clL$\hyp{}stable classes are
$\clAll$, $\clD{k}$, $\clD{0}$, $\clEmpty$, for $k \in \IN_{+}$.
\item\label{cor:C-stable:Sc}
The $\clSc$\hyp{}stable classes are
$\clAll$, $\clCon{a}$, $\clYksi{a}$, $\clParity{a}$, $\clBoth{a}{b}$, $\clChar{1} \cap \clParity{a}$, $\clChar{1} \cap \clBoth{a}{b}$, $\clD{0}$, $\clD{0} \cap \clCon{a}$, $\clEmpty$,
for $a, b \in \{0,1\}$ and $\QuantifyParRel$.
\item
The $\clS$\hyp{}stable classes are
$\clAll$, $\clChar{1} \cap \clParity{a}$, $\clD{0}$, $\clEmpty$,
for $\QuantifyParRel$.
\item
The $\clTc$\hyp{}stable classes are
$\clAll$, $\clCon{a}$, $\clYksi{a}$, $\clEven$, $\clBoth{a}{b}$, $\clD{0}$, $\clD{0} \cap \clCon{a}$, $\clEmpty$,
for $a, b \in \{0,1\}$.
\item
The $\clTo$\hyp{}stable classes are
$\clAll$, $\clCon{0}$, $\clD{0}$, $\clD{0} \cap \clCon{0}$, $\clEmpty$.
\item
The $\clTi$\hyp{}stable classes are
$\clAll$, $\clYksi{1}$, $\clD{0}$, $\clD{0} \cap \clCon{1}$, $\clEmpty$.
\item
The $\clAll$\hyp{}stable classes are
$\clAll$, $\clD{0}$, $\clEmpty$.
\end{enumerate}
\end{corollary}

Recall from Lemma~\ref{lem:Lc-simplify}\ref{lem:Lc-simplify:Lc} that $(\clIc,\clLc)$\hyp{}stability is equivalent to $\clLc$\hyp{}stability.
Therefore, as expected, the classes listed in
Corollary~\ref{cor:IcC-stable}\ref{cor:IcC-stable:Lc} are the same as those in
Corollary~\ref{cor:C-stable}\ref{cor:C-stable:Lc}.
By comparing Corollary~\ref{cor:IcC-stable}\ref{cor:IcC-stable:Sc} with Corollary~\ref{cor:C-stable}\ref{cor:C-stable:Sc}, we see also that $(\clIc,\clSc)$\hyp{}stability is equivalent to $\clSc$\hyp{}stability.
Whether the reason for this is a relationship similar to Lemma~\ref{lem:Lc-simplify} is beyond the scope of this paper.

\begin{corollary}
$\clSc$\hyp{}stability is equivalent to $(\clIc, \clSc)$\hyp{}stability.
\end{corollary}

%%%%%%%%%%%%%%%%%%%%%%%%%%%%%%%%%%%%%%%%%%%%%%%%%%

\section{Final remarks and perspectives}
\label{sec:remarks}

Looking into directions of future research,
one may consider arbitrary pairs of clones $C_1$ and $C_2$ on arbitrary sets $A$ and $B$ and describe the closure system of $(C_1,C_2)$\hyp{}stable sets, which we shall denote by $\closys{(C_1,C_2)}$.
However, this task is challenging.
Firstly, there are uncountably many clones on sets with at least three elements (see \cite{YanMuc-1959}), and not all of them are known.
Secondly, for given clones $C_1$ and $C_2$, there may be uncountably many $(C_1,C_2)$\hyp{}stable classes, in which case an explicit description may be unattainable.

For this reason, a natural next step would be to consider $(C_1,C_2)$\hyp{}stability for clones $C_1$ and $C_2$ on the two\hyp{}element set $\{0,1\}$, which are well known (see Post~\cite{Post}).
Moreover, the cardinality of the closure system $\closys{(\clIc,C)}$ of $(\clIc,C)$\hyp{}stable classes of Boolean functions is known for every clone $C$ on $\{0,1\}$, due to the following result by Sparks~\cite{Sparks-2019}.
However, this result does not provide an explicit description of the $(\clIc,C)$\hyp{}stable classes, even for the cases where the number of $(\clIc,C)$\hyp{}stable classes is finite.

\begin{theorem}[{Sparks~\cite[Theorem~1.3]{Sparks-2019}}]
\label{thm:Sparks}
Let $A$ be a finite set with $\card{A} > 1$, and let $B := \{0,1\}$.
Denote by $\clProj{A}$ the clone of projections on $A$, and let $C$ be a clone on $B$.
Then the following statements hold.
\begin{enumerate}[label={\upshape{(\roman*)}}, leftmargin=*, widest=iii]
\item\label{Sparks:finite}
$\closys{(\clProj{A}, C)}$ is finite if and only if $C$ contains a near\hyp{}unanimity operation.
\item\label{Sparks:countable}
$\closys{(\clProj{A}, C)}$ is countably infinite if and only if $C$ contains a Mal'cev operation but no majority operation.
\item\label{Sparks:uncountable}
$\closys{(\clProj{A}, C)}$ has the cardinality of the continuum if and only if $C$ contains neither a near\hyp{}unanimity operation nor a Mal'cev operation.
\end{enumerate}
\end{theorem}

Recall that an $n$\hyp{}ary operation $f \in \mathcal{O}_B$ with $n \geq 3$ is called a \emph{near\hyp{}unanimity operation} if
$f(x, \dots, x, y, x, \dots, x) = x$
for all $x, y \in B$, where the single occurrence of $y$ can occur in any of the $n$ argument positions.
A ternary near\hyp{}unanimity operation is called a \emph{majority operation.}
A ternary operation $f \in \mathcal{O}_B$ is called a \emph{Mal'cev operation} if
$f(y, y, x) = f(x, y, y) = x$
for all $x, y \in B$.

A clone $C$ on $\{0,1\}$ contains a Mal'cev operation but no majority operation (statement \ref{Sparks:countable}) if and only if $\clLc \subseteq C \subseteq \clL$; the $(C_1,C_2)$\hyp{}stable classes for clones $C_1$ and $C_2$ such that $\clLc \subseteq C_2$ are completely described in the current paper.
Regarding the situation when the closure system $\closys{(\clProj{A}, C)}$ is finite, that is, $C$ contains a near\hyp{}unanimity function (statement \ref{Sparks:finite}), the second author \cite{Lehtonen:SM} recently described completely the $(C_1,C_2)$\hyp{}stable classes in the special case when $C_2$ contains a majority operation, i.e., $\clSM \subseteq C_2$.
In the case when $C_2$ contains a near\hyp{}unanimity operation but no majority operation, we know from Theorem~\ref{thm:Sparks} \ref{Sparks:finite} and Lemma~\ref{lem:stable-impl-stable} that the closure system $\closys{(C_1,C_2)}$ is finite, but an explicit description thereof still eludes us.
In view of Theorem~\ref{thm:Sparks} \ref{Sparks:uncountable}, an explicit description of the $(C_1,C_2)$\hyp{}stable classes may be unattainable when $C_2$ contains neither a near\hyp{}unanimity operation nor a Mal'cev operation.
However, if the clone $C_1$ is large enough, the combination $(C_1,C_2)$ might provide a finite or countable closure system $\closys{(C_1,C_2)}$, the description of which might still be feasible.
This raises the following question: \emph{Which pairs $(C_1,C_2)$ of clones give rise to finite or countable closure systems?}

%%%%%%%%%%%%%%%%%%%%%%%%%%%%%%%%%%%%%%%%%%%%%%%%%%

\section*{Acknowledgments}

This work is funded by National Funds through the FCT -- Funda\c{c}\~ao para a Ci\^encia e a Tecnologia, I.P., under the scope of the project UIDB/00297/2020 (Center for Mathematics and Applications) and the project PTDC/MAT-PUR/31174/\discretionary{}{}{}2017.

%%%%%%%%%%%%%%%%%%%%%%%%%%%%%%%%%%%%%%%%%%%%%%%%%%

\end{document}